\documentclass[12pt]{amsart}

\usepackage{amsmath}
\usepackage{amsfonts}
\usepackage{amssymb}
\usepackage{latexsym}
\usepackage[all]{xy}
\newtheorem{theorem}{Theorem}[section]
\newtheorem{proposition}[theorem]{Proposition}

\newtheorem{definition}[theorem]{Definition}
\theoremstyle{remark}
\newtheorem{remark}[theorem]{Remark}
\theoremstyle{definition}
\newtheorem{example}[theorem]{Example}
\numberwithin{equation}{section}

   \def\sB{{\mathfrak B}}   
\def\sD{{\mathfrak D}}      
   \def\sH{{\mathfrak H}}   
   \def\sK{{\mathfrak K}}   \def\sL{{\mathfrak L}}
\def\sM{{\mathfrak M}}   \def\sN{{\mathfrak N}}

      \def\dC{{\mathbb C}}
\def\dD{{\mathbb D}}

   \def\dN{{\mathbb N}}   
      
   \def\dT{{\mathbb T}}   
      
   \def\dZ{{\mathbb Z}}

\def\cA{{\mathcal A}}   \def\cB{{\mathcal B}}   \def\cC{{\mathcal C}}
\def\cD{{\mathcal D}}   \def\cE{{\mathcal E}}   \def\cF{{\mathcal F}}
\def\cG{{\mathcal G}}   \def\cH{{\mathcal H}}   
   \def\cK{{\mathcal K}}   \def\cL{{\mathcal L}}
\def\cM{{\mathcal M}}      
      
\def\cS{{\mathcal S}}   \def\cT{{\mathcal T}}   \def\cU{{\mathcal U}}
\def\cV{{\mathcal V}}   \def\cW{{\mathcal W}}   \def\cX{{\mathcal X}}

\def\bL{{\mathbf L}}

\def\clos{{\rm clos\,}}
\def\RE{{\rm Re\,}}

\def\wt{\widetilde}
\def\wh{\widehat}

\def\f{\varphi}

\def\uphar{{\upharpoonright\,}}
\def\ovl{\overline}

\def\ran{{\rm ran\,}}

\def\cran{{\rm \overline{ran}\,}}
\def\cspan{{\rm \overline{span}\, }}
\begin{document}
\title[Block operator CMV
matrices] {Conservative discrete time-invariant systems and block
operator CMV matrices}
\author{
Yury Arlinski\u{i}}
\address{Department of Mathematical Analysis \\
East Ukrainian National University \\
Kvartal Molodyozhny 20-A \\
Lugansk 91034 \\
Ukraine} \email{yma@snu.edu.ua}


\thispagestyle{empty}
\begin{abstract}
It is well known that an operator-valued function $\Theta$ from the
Schur class ${\bf S}(\mathfrak M,\mathfrak N)$, where $\mathfrak M$
and $\mathfrak N$ are separable Hilbert spaces, can be realized as
the transfer function of a simple conservative discrete
time-invariant linear system. The known realizations involve the
function $\Theta$ itself, the Hardy spaces or the reproducing kernel
Hilbert spaces. On the other hand, as in the classical scalar case,
the Schur class operator-valued function is uniquely determined by
its so called "Schur parameters". In this paper we construct simple
conservative realizations using the Schur parameters only. It turns
out that the unitary operators corresponding to the systems take the
form of five diagonal block operator matrices, which are the analogs
of Cantero--Moral--Vel\'azquez (CMV) matrices appeared recently in
the theory of scalar orthogonal polynomials on the unit circle. We
obtain new models given by truncated block operator CMV matrices for
an arbitrary completely non-unitary contraction. It is shown that
the minimal unitary dilations of a contraction in a Hilbert space
and the minimal Naimark dilations of a semi-spectral operator
measure on the unit circle can also be expressed by means of block
operator CMV matrices.
\end{abstract}
\maketitle
 \tableofcontents
\section{Introduction}

In what follows the class of all continuous linear operators defined
on a complex Hilbert space $\sH_1$ and taking values in a complex
Hilbert space $\sH_2$ is denoted by $\bL(\sH_1,\sH_2)$ and
${\bL}(\sH):= {\bL}(\sH,\sH)$. We denote by $I_\cH$ the identity
operator in a Hilbert space $\cH$ and by $P_\cL$ the orthogonal
projection onto the subspace (the closed linear manifold) $\cL$. The
notation $T\uphar \cL$ means the restriction of a linear operator
$T$ on the set $\cL$. The range and the null-space of a linear
operator $T$ are denoted by $\ran T$ and $\ker T$, respectively.

 Recall that an operator
$T\in\bL(\sH_1,\sH_2)$ is said to be
\begin{itemize}
\item \textit{contractive} if $\|T\|\le 1$;

\item \textit{isometric} if $\|Tf\|=\|f\|$ for all $f\in \sH_1$
$\iff T^*T=I_{\sH_1}$;

\item \textit{co-isometric} if $T^*$ is isometric $\iff
TT^*=I_{\sH_2}$;
\item \textit{unitary} if it is both isometric and co-isometric.
\end{itemize}
Given a contraction $T\in \bL(\sH_1,\sH_2)$. The operators
\[
D_T:=(I-T^*T)^{1/2},\qquad D_{T^*}:=(I-TT^*)^{1/2}
\]
are called the \textit{defect operators} of $T$, and the subspaces
$\sD_T=\cran D_T,$ $\sD_{T^*}=\cran D_{T^*}$ the \textit{defect
subspaces} of $T$. The dimensions $\dim\sD_T,$ $\dim\sD_{T^*}$ are
known as the \textit{defect numbers} of $T$. The defect operators
satisfy the following intertwining relations
\begin{equation}
\label{defect} TD_{T}=D_{T^*}T,\qquad T^*D_{T^*}=D_{T}T^*.
\end{equation}
It follows from \eqref{defect} that $T\sD_T\subset\sD_{T^*}$,
$T^*\sD_{T^*}\subset\sD_T$, and $T(\ker D_T)=\ker D_{T^*},$
$T^*(\ker D_{T^*})=\ker D_{T}$. Moreover, the operators $T\uphar\ker
D_{T}$ and $T^*\uphar\ker D_{T^*}$ are isometries and $T\uphar\sD_T$
and $T^*\uphar\sD_{T^*}$ are \textit{pure} contractions, i.e.,
$||Tf||<||f||$ for $f\in\sH\setminus\{0\}$.

 The \textit{Schur class}
${\bf S}(\sH_1,\sH_2)$ is the set of all holomorphic and contractive
$\bL(\sH_1,\sH_2)$-valued functions on the unit disk
$\dD=\{\lambda\in\dC:|\lambda|<1\}$.  This class is a natural
generalization of the Schur class ${\bf S}$ of scalar analytic
functions mapping the unit disk $\dD$ into the closed unit disk
$\overline{\dD}$ \cite{Schur} and is intimately connected 
with spectral theory and models for Hilbert space contraction
operators \cite{SF}, \cite{BrR1}, \cite{BrR2}, \cite{VilBr},
\cite{Br}, \cite{Br1}, the Lax-Phillips scattering theory
\cite{LPh}, \cite{AA}, \cite{Ball-C-Ue}, the theory of scalar and
matrix orthogonal polynomials on the unit circle
$\dT=\{\xi\in\dC:|\xi|=1\}$ \cite{DPS}, \cite{Si1}, \cite{DGK1},
\cite{DGK}, the theory of passive (contractive) discrete
time-invariant linear systems \cite{Helton1}, \cite{Helton2},
\cite{A}, \cite{Arov}, \cite{ArKaaP}, \cite{Ball-Coehn},
\cite{Ball}. One of the characterization of the operator-valued
Schur class is that any $\Theta\in{\bf S}(\sM,\sN)$ can be realized
as the transfer (characteristic) function of the form
\[
\Theta(\lambda)=D+\lambda C(I_\sH-\lambda A)^{-1}B,\;\lambda\in\dD
\]
of a discrete time-invariant system (colligation)
\[
\tau=\left\{\begin{bmatrix}D&C\cr B&A
\end{bmatrix};\sM,\sN,\sH\right\}
\]
with the input space $\sM$, the output space $\sN$, and some state
space $\sH$. Moreover, if the operator $U_\tau$ is given by the
block operator matrix
\[
U_\tau=\begin{bmatrix}D&C\cr B&A
\end{bmatrix}:\begin{array}{l}\sM\\\oplus\\\sH\end{array}\to
\begin{array}{l}\sN\\\oplus\\\sH\end{array},
\]
then the system $\tau$ can be chosen {\bf(a)} passive ($U_\tau$ is
contractive) and minimal,  {\bf(b)} co-isometric ($U_\tau$ is
co-isometry) and observable, {\bf(c)} isometric ($U_\tau$ is
isometry) and controllable, {\bf(d)} conservative ($U_\tau$ is
unitary) and simple (see Section \ref{secS}). The corresponding
models of the systems $\tau$ and the state space operators $A$ are
well-known. We mention the de Branges--Rovnyak functional model of a
co-isometric system \cite{BrR2}, \cite{Ando}, \cite{NV3}, the
Sz.-Nagy--Foias \cite{SF}, the Pavlov \cite{Pavlov1},
\cite{Pavlov2}, \cite{Pavlov3}, and the
 Nikol'ski\u{i}--Vasyunin \cite{NV1}, \cite{NV2} functional models
 of completely non-unitary contractions, the Brodski\u{i} \cite{Br1} functional model of a simple unitary
colligation, the Arov--Kaashoek--Pik \cite{ArKaaP} functional model
of a passive minimal and optimal system. All these models involve
the Schur class function and/or the Hardy spaces, the de
Branges--Rovnyak reproducing kernel Hilbert space.

The main goal of the present paper is constructions of models for
simple conservative systems and completely non-unitary contractions
by means of the operator analogs of the scalar CMV matrices, which
recently appeared in the theory of orthogonal polynomials on the
unit circle \cite{CMV1}, \cite{Si1}, \cite{S3}, \cite{DPS}.

In the paper of M.J.~Cantero,
L.~Moral, and L.~Vel\'azquez \cite{CMV1} it is established that the
semi-infinite matrices of the form
\begin{equation}
\label{CMV11} \cC=\cC(\{\alpha_n\})=\begin{pmatrix}
\bar{\alpha}_0&\bar{\alpha}_1\rho_0&\rho_1\rho_0&0&0&\ldots\cr
\rho_0&-\bar{\alpha}_1\alpha_0&-\rho_1\alpha_0&0&0&\ldots\cr
0&\bar{\alpha}_2\rho_1&-\bar{\alpha}_2\alpha_1&\bar{\alpha}_3\rho_2&\rho_3
\rho_2& \ldots\cr
0&\rho_2\rho_1&-\rho_2\alpha_1&-\bar{\alpha}_3\alpha_2&-\rho_3\alpha_2&\ldots
\cr 0&0&0&\bar{\alpha}_4\rho_3&-\bar{\alpha}_4\alpha_3&\ldots\cr
\vdots&\vdots&\vdots&\vdots&\vdots&\vdots
\end{pmatrix}
\end{equation} and
\begin{equation}
\label{CMV12} \wt\cC=\wt\cC(\{\alpha_n\})=\begin{pmatrix}
\bar{\alpha}_0&\rho_0&0&0&0&\ldots\cr
\bar{\alpha}_1\rho_0&-\bar{\alpha}_1\alpha_0&\bar{\alpha}_2\rho_1&\rho_2\rho_1&0&\ldots\cr
\rho_1\rho_0&-\rho_1\alpha_0&-\bar{\alpha}_2\alpha_1&-\rho_2\alpha_1&0&\ldots\cr
0&0&\bar{\alpha}_3\rho_2&-\bar{\alpha}_3\alpha_2&\bar{\alpha}_4\rho_3&\ldots
\cr
0&0&\rho_3\rho_2&-\rho_3\alpha_2&-\bar{\alpha}_4\alpha_3&\ldots\cr
\vdots&\vdots&\vdots&\vdots&\vdots&\vdots
\end{pmatrix}
\end{equation}
give  representations of the unitary operator $(Uf)(\zeta)=\zeta
f(\zeta)$ in $L_2(\dT,d\mu)$, where the $d\mu$ is a nontrivial
probability measure on the unite circle, with respect to the
orthonormal systems  obtained by orthonormalization of the sequences
$\{1,\zeta,\zeta^{-1},\zeta^2,\zeta^{-2},\ldots\}$ and
$\{1,\zeta^{-1},\zeta,\zeta^{-2},\zeta^2,\ldots\}$, respectively.
The Verblunsky coefficients $\{\alpha_n\}$, $|\alpha_n|< 1$, arise
in the Szeg\H{o} recurrence formula
\[
\zeta \Phi_n(\zeta)=\Phi_{n+1}(\zeta)+
\bar\alpha_n\zeta^n\ovl{\Phi_n(1/\bar\zeta)},\qquad n=0,1, \ldots
\]
for monic orthogonal with respect to $d\mu$ polynomials
$\{\Phi_n\}$, and $\rho_n:=\sqrt{1-|\alpha_n|^2}$. The matrices
$\cC(\{\alpha_n\})$ $\wt\cC(\{\alpha_n\})$ and are called the {\it
CMV matrices}. Note that the matrix $\wt C$ is transpose to $\cC$.

 Given a probability measure
$\mu$ on $\dT$, define the {\it Carath\'eodory function} by
\[
 F(\lambda)=
F(\lambda,\mu):=\int_{\dT}\frac{\zeta+\lambda}{\zeta-\lambda}\,d\mu(\zeta)=
1+2\sum_{n=1}^\infty \beta_n\lambda^n, \quad \beta_n=\int_{\dT}
\zeta^{-n}d\mu \]
 the moments of $\mu$. $F$ is an analytic
function in $\dD$ which obeys $\RE F>0$, $F(0)=1$. The  Schur class
function $f(\lambda)$ is then defined by
\[
 f(\lambda)= f(\lambda,\mu):=\frac{1}{\lambda}\,\frac{F(\lambda)-1}{F(\lambda)+1},
\]
Given a Schur function $f(\lambda)$, which is not a finite Blaschke
product, define inductively
\[
f_0(\lambda)=f(\lambda),\;
f_{n+1}(\lambda)=\frac{f_n(\lambda)-f_n(0)}{\lambda(1-\overline{f_n(0)}f_n(\lambda))},\;
n\ge 0.
\]
It is clear that  $\{f_n\}$ is an {\it infinite} sequence of Schur
functions called the \textit{$n-th$ Schur iterates} and neither of
its terms is a finite Blaschke product. The numbers
$\gamma_n:=f_n(0)$ are called the {\it Schur parameters:}
\[
\cS f=\{\gamma_0,\gamma_1,\ldots\}.
\]
Note that
\[
f_n(\lambda)=\frac{\gamma_n+\lambda
f_{n+1}(\lambda)}{1+\bar\gamma_n\lambda
f_{n+1}}=\gamma_n+(1-|\gamma_n|^2)\frac{\lambda
f_{n+1}(\lambda)}{1+\bar\gamma_n\lambda f_{n+1}(\lambda)},\; n\ge 0.
\]
The method of labeling $f\in{\bf S}$ by its Schur parameters is
known as the \textit{Schur algorithm} and is due to I.~Schur
\cite{Schur}. In the case when
\[
f(\lambda)=e^{i\f}\prod_{k=1}^N \frac{\lambda-\lambda_k}{1-\bar
\lambda_k \lambda}
\]
is a finite Blaschke product of order $N$, the Schur algorithm
terminates at the $N$-th step. The sequence of Schur parameters
$\{\gamma_n\}_{n=0}^N$ is finite, $|\gamma_n|<1$ for
$n=0,1,\ldots,N-1$, and $|\gamma_N|=1$.

Due to Geronimus theorem  for the function $f(\lambda,\mu)$ the
relations $\gamma_n=\alpha_n$ hold true for all $n=0,1,\ldots$.

There is a nice multiplicative structure of the CMV matrices. In the
semi-infinite case $\cC$ and $\wt \cC$ are the products of two
matrices: $\cC=\cL\cM$, $\wt\cC=\cM\cL$, where
\[
\begin{split}
\cL
&=\Psi(\alpha_0)\oplus\Psi(\alpha_2)\oplus\ldots\Psi(\alpha_{2m})
\oplus\ldots,\\
\cM &={\bf 1}_{1\times 1}\oplus\Psi(\alpha_1)\oplus\Psi(\alpha_3)
\oplus\ldots\oplus\Psi(\alpha_{2m+1})\oplus\ldots,
\end{split}
\]
and
$\Psi(\alpha)=\begin{pmatrix}\bar\alpha&\rho\cr\rho&-\alpha\end{pmatrix}.$
The finite $(N+1)\times (N+1)$ CMV matrices $\cC$ and $\wt\cC$ obey
$\alpha_0,\alpha_1,\ldots,\alpha_{N-1}\in \dD$ and $
|\alpha_{N}|=1$, and also $\cC=\cL\cM$, $\wt\cC=\cM\cL$, where in
this case $\Psi(\alpha_{N})=\left(\bar\alpha_{N}\right)$.

In the paper \cite{AGT} it is established that the {\it truncated}
CMV matrices
\[
\cT_0=\cT_0(\{\alpha_n\})=\begin{pmatrix}
-\bar{\alpha}_1\alpha_0&-\rho_1\alpha_0&0&0&\ldots\cr
\bar{\alpha}_2\rho_1&-\bar{\alpha}_2\alpha_1&\bar{\alpha}_3\rho_2&\rho_3
\rho_2& \ldots\cr
\rho_2\rho_1&-\rho_2\alpha_1&-\bar{\alpha}_3\alpha_2&-\rho_3\alpha_2&\ldots
\cr 0&0&\bar{\alpha}_4\rho_3&-\bar{\alpha}_4\alpha_3&\ldots\cr
\vdots&\vdots&\vdots&\vdots&\vdots&\vdots
\end{pmatrix}
\]
and
\[
\wt\cT_0=\wt\cT_0(\{\alpha_n\})=\begin{pmatrix}
-\bar{\alpha}_1\alpha_0&\bar{\alpha}_2\rho_1&\rho_2\rho_1&0&\ldots\cr
-\rho_1\alpha_0&-\bar{\alpha}_2\alpha_1&-\rho_2\alpha_1&0&\ldots\cr
0&\bar{\alpha}_3\rho_2&-\bar{\alpha}_3\alpha_2&\bar{\alpha}_4\rho_3&\ldots
\cr 0&\rho_3\rho_2&-\rho_3\alpha_2&-\bar{\alpha}_4\alpha_3&\ldots\cr
\vdots&\vdots&\vdots&\vdots&\vdots
\end{pmatrix}
\]
obtained from the ``full'' CMV matrices $\cC=\cC(\{\alpha_n\})$ and
$\wt\cC=\wt\cC(\{\alpha_n\})$ by deleting the first row and the
first column, provide the models of completely non-unitary
contractions with rank one defect operators.

 As pointed out by Simon in \cite{S3}, the
history of CMV matrices is started with the papers of Bunse-Gerstner
and Elsner \cite{BGE91} (1991) and Watkins \cite{Wa93} (1993), where
unitary semi-infinite five-diagonal matrices were introduces and
studied. In \cite{CMV1} Cantero, Moral, and Velazquez (CMV)
re-discovered them. In a context different from orthogonal
polynomials on the unit circle, Bourget, Howland, and Joye
\cite{BHJ03} introduced a set of doubly infinite family of matrices
with three sets of parameters which for special choices of the
parameters reduces to two-sided CMV matrices on $\ell^2(\dZ)$.

The Schur algorithm for matrix valued Schur class functions and its
connection with the matrix orthogonal polynomials on the unit circle
have been considered in the paper of Delsarte, Genin, and Kamp
\cite{DGK} and in the book of Dubovoj, Fritzsche, and Kirstein
\cite{DFK}. The CMV matrices, connected with matrix orthogonal
polynomials on the unit circle with respect to nontrivial
matrix-valued measures are considered in \cite{S3}, \cite{DPS}. If
the $k\times k$ matrix-valued non-trivial measure $\mu$ on $\dT$,
$\mu(\dT)=I_{k\times k}$ is given, then there are the left and the
right orthonormal matrix polynomials. The Szeg\H{o} recursions take
slightly different form than in the scalar case and the Verblunsky
$k\times k$ matrix coefficients (the Schur parameters of the
corresponding matrix-valued Schur function) $\{\alpha_n\}$ satisfy
the inequality $||\alpha_n||<1$ for all $n$. The latter condition is
in fact equivalent to the non-triviality of the measure. The entries
of the corresponding CMV matrix have the size $k\times k$ and the
numbers $\rho_n$ are replaced by the $k\times k$ defect matrices
$\rho^{L}_n=D_{\alpha_n}=(I-\alpha^*_n\alpha_n)^{1/2}$ and
$\rho^{R}_n=D_{\alpha^*_n}=(I-\alpha_n\alpha^*_n)^{1/2}$, where
$\alpha^*$ is the adjoint matrix. In these notations the CMV matrix
is of the form \cite{DPS}
\begin{equation}
\label{CCMV} \cC=\cC(\{\alpha_n\})=\begin{pmatrix}
\alpha^*_0&\rho^{L}_0\alpha^*_1&\rho^{L}_0\rho^{L}_1&0&0&\ldots\cr
\rho^{R}_0&-\alpha_0\alpha^*_1&-\alpha_0\rho^{L}_1&0&0&\ldots\cr
0&\alpha^*_2\rho^{R}_1&-\alpha^*_2\alpha_1&\rho^{L}_2\alpha^*_3&\rho^{L}_2
\rho^{L}_3& \ldots\cr
0&\rho^{R}_2\rho^{R}_1&-\rho^{R}_2\alpha_1&-\alpha_2\alpha^*_3&-\alpha_2\rho^{L}_3&\ldots
\cr 0&0&0&\alpha^*_4\rho^{R}_3&-\alpha^*_4\alpha_3&\ldots\cr
\vdots&\vdots&\vdots&\vdots&\vdots&\vdots
\end{pmatrix}.
\end{equation}
The operator extension of the Schur algorithm was developed by
T.~Constantinescu in \cite{Const} and with numerous applications is
presented in the monographs \cite{BC}, \cite{Const2}. The next
theorem goes back to Shmul'yan \cite{Shmul1}, \cite{Shmul2} and
T.~Constantinescu \cite{Const} (see also \cite{BC}, \cite{ARL1},
\cite{Arlarxiv}) and plays a key role in the operator Schur
algorithm.
\begin{theorem}
\label{MO} Let $\sM$ and $\sN$ be separable Hilbert spaces and let
the function $\Theta(\lambda)$ be from the Schur class ${\bf
S}(\sM,\sN).$ Then there exists a function $Z(\lambda)$ from the
Schur class ${\bf S}(\sD_{\Theta(0)},\sD_{\Theta^*(0)})$ such that
\begin{equation}
\label{MREP}
\Theta(\lambda)=\Theta(0)+D_{\Theta^*(0)}Z(\lambda)(I+\Theta^*(0)Z(\lambda))^{-1}D_{\Theta(0)},\;\lambda\in\dD.
\end{equation}
\end{theorem}
The representation \eqref{MREP} of a function $\Theta(\lambda)$ from
the Schur class is called the M\"obius representation of
$\Theta(\lambda)$ and the function $Z(\lambda)$ is called the
M\"obius parameter of $\Theta(\lambda)$ (see \cite{ARL1},
\cite{Arlarxiv}). Clearly, $Z(0)=0$ and by Schwartz's lemma we
obtain that
\[
||Z(\lambda)||\le|\lambda|,\;\lambda\in\dD.
\]

\textit{The operator Schur's algorithm} \cite{BC}. Fix
$\Theta(\lambda)\in{\bf S}(\sM,\sN)$, put
$\Theta_0(\lambda)=\Theta(\lambda)$ and let $Z_0(\lambda)$ be the
M\"obius parameter of $\Theta$. Define
\[
\Gamma_0=\Theta(0),\;
\Theta_1(\lambda)=\frac{Z_0(\lambda)}{\lambda}\in {\bf
S}(\sD_{\Gamma_0},\sD_{\Gamma^*_0}),\;\Gamma_1= \Theta_1(0)=Z'_0(0).
\]
If $\Theta_0(\lambda),\ldots,\Theta_n(\lambda)$ and
$\Gamma_0,\ldots, \Gamma_n$ have been chosen, then let
$Z_{n+1}(\lambda)\in {\bf S}(\sD_{\Gamma_n},\sD_{\Gamma^*_n})$ be
the M\"obius parameter of $\Theta_n$. Put
\[
\Theta_{n+1}(\lambda)=\frac{Z_{n+1}(\lambda)}{\lambda},\;
\Gamma_{n+1}=\Theta_{n+1}(0).
\]
 The contractions $\Gamma_0\in\bL(\sM,\sN),$
$\Gamma_n\in\bL(\sD_{\Gamma_{n-1}},\sD_{\Gamma^*_{n-1}})$,
$n=1,2,\ldots$ are called the \textit{Schur parameters} of
$\Theta(\lambda)$ and the function $\Theta_n(\lambda) \in {\bf
S}(\sD_{\Gamma_{n-1}},\sD_{\Gamma^*_{n-1}})$ we will call the $n-th$
\textit{Schur iterate} of  $\Theta(\lambda)$.

Formally we have
\[
\Theta_{n+1}(\lambda)\uphar\ran
D_{\Gamma_n}=\frac{1}{\lambda}D_{\Gamma^*_n}(I_{\sD_{\Gamma^*_n}}-\Theta_n(\lambda)\Gamma^*_n)^{-1}
(\Theta_n(\lambda)-\Gamma_n)D^{-1}_{\Gamma_n}\uphar\ran
D_{\Gamma_n}.
\]
Clearly, the sequence of Schur parameters $\{\Gamma_n\}$ is infinite
if and only if all operators $\Gamma_n$ are non-unitary. The
sequence of Schur parameters consists of a finite number of
operators $\Gamma_0,$ $\Gamma_1,\ldots, \Gamma_N$ if and only if
$\Gamma_N\in\bL(\sD_{\Gamma_{N-1}},\sD_{\Gamma^*_{N-1}})$ is
unitary. If $\Gamma_N$ is isometric (co-isometric) then $\Gamma_n=0$
for all $n>N$.
 The following generalization of the classical Schur result is proved in \cite{Const}
(see also \cite{BC}).
\begin{theorem} \label{SchurAlg} There is a one-to-one
correspondence between the Schur class functions ${\bf S}(\sM,\sN)$
and the set of all sequences of contractions $\{\Gamma_n\}_{n\ge 0}$
such that
\begin{equation}
\label{CHSEQ} \Gamma_0\in\bL(\sM,\sN),\;\Gamma_n\in
\bL(\sD_{\Gamma_{n-1}},\sD_{\Gamma^*_{n-1}}),\; n\ge 1.
\end{equation}
\end{theorem}
A sequence of contractions of the form \eqref{CHSEQ} is called the
\textit{choice sequence} \cite{CF}. Such objects are used
 for the indexing of contractive intertwining
dilations, of positive Toeplitz forms, and of the Naimark dilations
of semi-spectral measures on the unit circle (see \cite{CF},
\cite{Const1}, \cite{Const3}, \cite{BC}, \cite{Const2}). Observe
that the Naimark dilation and the model of a simple conservative
system are given in \cite{Const1}, \cite{Const}, and \cite{BC}  by
infinite in all sides block operator matrix whose entries are
expressed by means of the choice sequence or the Schur parameters.

Let us describe the main results of our paper. Given a choice
sequence \eqref{CHSEQ}. We construct the Hilbert spaces
$\sH_0=\sH_0(\{\Gamma_n\}_{n\ge 0})$,
$\wt\sH_0=\wt\sH_0(\{\Gamma_n\}_{n\ge 0})$ , the unitary operators
\[
\cU_0=\cU_0(\{\Gamma_n\}_{n\ge 0})=\begin{bmatrix} \Gamma_0&\cG_0\cr
\cF_0
&\cT_0\end{bmatrix}:\begin{array}{l}\sM\\\oplus\\\sH_0\end{array}\to
\begin{array}{l}\sN\\\oplus\\\sH_0\end{array},\;
\wt\cU_0=\wt\cU_0(\{\Gamma_n\}_{n\ge 0})=\begin{bmatrix}
\Gamma_0&\wt\cG_0\cr\wt\cF_0
&\wt\cT_0\end{bmatrix}:\begin{array}{l}\sM\\\oplus\\\wt
\sH_0\end{array}\to
\begin{array}{l}\sN\\\oplus\\\wt\sH_0\end{array},
\]
and the unitarily equivalent simple conservative systems
\[
\zeta_0=\left\{\begin{bmatrix} \Gamma_0&\cG_0\cr \cF_0
&\cT_0\end{bmatrix};\sM,\sN,\sH_0\right\},\;
\wt\zeta_0=\left\{\begin{bmatrix} \Gamma_0&\wt\cG_0\cr\wt\cF_0
&\wt\cT_0\end{bmatrix};\sM,\sN,\wt\sH_0\right\},
\]
such that the Schur parameters of the transfer function $\Theta$ of
the systems $\zeta_0$ and $\wt\zeta_0$ are precisely
$\{\Gamma_n\}_{n\ge 0}$. Moreover, the operators $\cU_0$ and
$\wt\cU_0$ in such constructions are given by the operator analogs
of the CMV matrices. In the case when the operators $\Gamma_n$ are
neither isometric nor co-isometric for each $n=0,1,\ldots$, the
Hilbert spaces $\sH_0$ and $\wt\sH_0$ are of the form
\[
\sH_0=\sum\limits_{n\ge
0}\bigoplus\begin{array}{l}\sD_{\Gamma_{2n}}\\\oplus\\\sD_{\Gamma^*_{2n+1}}\end{array},\;
\wt\sH_0=\sum\limits_{n\ge
0}\bigoplus\begin{array}{l}\sD_{\Gamma^*_{2n}}\\\oplus\\\sD_{\Gamma_{2n+1}}\end{array},
\]
and the operators $\cU_0$ and $\wt\cU_0$
 are given by the products of unitary diagonal operator matrices
\[
\cU_0=\begin{bmatrix}{\bf J}_{\Gamma_0}\cr &{\bf J}_{\Gamma_2}\cr &
&{\bf J}_{\Gamma_4}\cr & &&\ddots
\end{bmatrix}\begin{bmatrix}I_{\sM}\cr&{\bf J}_{\Gamma_1}\cr &&{\bf
J}_{\Gamma_3}\cr & &&\ddots\end{bmatrix},\;
\]
\[
\wt\cU_0=\begin{bmatrix}I_{\sN}\cr &{\bf J}_{\Gamma_1}\cr &&{\bf
J}_{\Gamma_3}\cr & & &\ddots\end{bmatrix}\begin{bmatrix}{\bf
J}_{\Gamma_0}\cr &{\bf J}_{\Gamma_2}\cr & &{\bf J}_{\Gamma_4}\cr &
&&\ddots
\end{bmatrix},
\]
where
\[
{\bf J}_{\Gamma_0}=\begin{bmatrix} \Gamma_0& D_{\Gamma^*_0}\cr
D_{\Gamma_0}&-\Gamma^*_0\end{bmatrix}:\begin{array}{l}\sM\\\oplus\\\sD_{\Gamma^*_0}\end{array}\to
\begin{array}{l}\sN\\\oplus\\\sD_{\Gamma_{0}}\end{array},\;
{\bf J}_{\Gamma_k}=\begin{bmatrix} \Gamma_k& D_{\Gamma^*_k}\cr
D_{\Gamma_k}&-\Gamma^*_k\end{bmatrix}:\begin{array}{l}\sD_{\Gamma_{k-1}}\\\oplus\\\sD_{\Gamma^*_k}\end{array}\to
\begin{array}{l}\sD_{\Gamma^*_{k-1}}\\\oplus\\\sD_{\Gamma_{k}}\end{array},\;
k=1,2,\ldots
\]
are the unitary operators called "elementary rotations" \cite{BC}.
The operators $\cU_0$ and $\wt\cU_0$ take the form of  five-diagonal
block operator matrices
\[
\cU_0=
\begin{bmatrix}
\Gamma_0&D_{\Gamma^*_0}\Gamma_1&D_{\Gamma^*_0}D_{\Gamma^*_1}&0&0&0&0&0&\ldots\cr
D_{\Gamma_0}&-\Gamma^*_0\Gamma_1&-\Gamma^*_0D_{\Gamma^*_1}&0&0&0&0&0&\ldots\cr
0&\Gamma_2D_{\Gamma_1}&-\Gamma_2\Gamma^*_1&D_{\Gamma^*_2}\Gamma_3&D_{\Gamma^*_2}D_{\Gamma^*_3}&0&0&0&\ldots\cr
0&D_{\Gamma_2}D_{\Gamma_1}&-D_{\Gamma_2}\Gamma^*_1&-\Gamma^*_2\Gamma_3&-\Gamma^*_2D_{\Gamma^*_3}&0&0&0&\ldots\cr
0&0&0&\Gamma_4D_{\Gamma_3}&-\Gamma_4\Gamma^*_3&D_{\Gamma^*_4}\Gamma_5&D_{\Gamma^*_4}D_{\Gamma^*_5}&0&\dots\cr
\vdots&\vdots&\vdots&\vdots&\vdots&\vdots&\vdots&\vdots&\vdots
\end{bmatrix}
\]
and
\[
 \wt\cU_0=
\begin{bmatrix}
\Gamma_0&D_{\Gamma^*_0}&0&0&0&0&0&\ldots\cr
\Gamma_1D_{\Gamma_0}&-\Gamma_1\Gamma^*_0&D_{\Gamma^*_1}\Gamma_2&D_{\Gamma^*_1}D_{\Gamma^*_2}&0&0&0&\ldots\cr
D_{\Gamma_1}D_{\Gamma_0}&-D_{\Gamma_1}\Gamma^*_0&-\Gamma^*_1\Gamma_2&-\Gamma^*_1D_{\Gamma^*_2}&0&0&0&\ldots\cr
0&0&\Gamma_3D_{\Gamma_2}&-\Gamma_3\Gamma^*_2&D_{\Gamma^*_3}\Gamma_4&D_{\Gamma^*_3}D_{\Gamma^*_4}&0&\ldots\cr
0&0&D_{\Gamma_3}D_{\Gamma_2}&-D_{\Gamma_3}\Gamma^*_2&-\Gamma^*_3\Gamma_4&-\Gamma^*_3D_{\Gamma^*_4}&0&\ldots\cr
\vdots&\vdots&\vdots&\vdots&\vdots&\vdots&\vdots&\vdots&
\end{bmatrix}
\]
Note that the following relation
\[
\wt\cU_0(\{\Gamma_n\}_{n\ge 0})=\left(\cU_0(\{\Gamma^*_n\}_{n\ge
0}\right)^*.
\]
holds true. Hence the CMV matrix \eqref{CCMV} corresponds to the
case
\[
\begin{array}{l}
\sM=\sN=\sD_{\Gamma_0}=\sD_{\Gamma^*_0}=\sD_{\Gamma_1}=\sD_{\Gamma^*_1}=\ldots=\sD_{\Gamma_n}=\sD_{\Gamma^*_n}
=\ldots=\dC^k,\\
\alpha_n=\Gamma^*_n,\; n=0,1,\ldots,\;
\end{array}
\]
Thus,
\[
\cC(\{\alpha_n\})=\cU_0(\{\Gamma_n^*\}_{n\ge
0})=\left(\wt\cU_0(\{\Gamma_n\}_{n\ge 0}\right)^*,\;
\wt\cC(\{\alpha_n\})=\wt\cU_0(\{\Gamma_n^*\}_{n\ge
0})=\left(\cU_0(\{\Gamma_n\}_{n\ge 0}\right)^*.
\]
The block operator truncated CMV matrices
\[
\cT_0=\cT_0(\{\Gamma_n\}_{n\ge
0}):=P_{\sH_0}\cU_0\uphar\sH_0\;\mbox{ and}\;
\wt\cT_0=\wt\cT_0(\{\Gamma_n\}_{n\ge
0}):=P_{\wt\sH_0}\wt\cU_0\uphar\wt\sH_0
\] are given by
\[
\cT_0=\begin{bmatrix}-\Gamma^*_0\cr &{\bf J}_{\Gamma_2}\cr & &{\bf
J}_{\Gamma_4}\cr & &&\ddots\end{bmatrix}
\begin{bmatrix}{\bf
J}_{\Gamma_1}\cr &{\bf J}_{\Gamma_3}\cr & &\ddots\end{bmatrix},\;
\wt\cT_0=\begin{bmatrix}{\bf J}_{\Gamma_1}\cr &{\bf J}_{\Gamma_3}\cr
& &\ddots\end{bmatrix}\begin{bmatrix}-\Gamma^*_0\cr &{\bf
J}_{\Gamma_2}\cr & &{\bf J}_{\Gamma_4}\cr & &&\ddots\end{bmatrix},
\]
and can be rewritten in the three diagonal block operator matrix
form with $2\times 2$ entries
\[
  \cT_0=\begin{bmatrix}
\cB_1 & \cC_1 & 0 &0& 0&
\cdot  \\
 \cA_1 & \cB_2 & \cC_2 &0& 0 &
\cdot   \\
0&\cA_2&\cB_3&\cC_3&0&\cdot\\
 \vdots & \vdots & \vdots & \vdots & \vdots
& \vdots
\end{bmatrix},\;
  \wt\cT_0=\begin{bmatrix}
\wt\cB_1 & \wt\cC_1 & 0 &0& 0&
\cdot  \\
 \wt\cA_1 & \wt\cB_2 & \wt\cC_2 &0& 0 &
\cdot   \\
0&\wt\cA_2&\wt\cB_3&\wt\cC_3&0&\cdot\\
 \vdots & \vdots & \vdots & \vdots & \vdots
& \vdots
\end{bmatrix}.
\]
The constructions above and the corresponding results are presented
in Section \ref{BCMVR}. We essentially rely on the constructions of
simple conservative realizations of the Schur iterates
$\{\Theta_n(\lambda)\}_{n\ge 1}$ by means of a given simple
conservative realization of the function $\Theta\in{\bf S}(\sM,\sN)$
\cite{Arlarxiv}. A brief survey of the results in \cite{Arlarxiv}
are given in Section \ref{gener}. The cases when the Schur parameter
$\Gamma_m\in\bL(\sD_{\Gamma_{m-1}},\sD_{\Gamma^*_{m-1}})$ of the
function  $\Theta\in{\bf S}(\sM,\sN)$, is isometric, co-isometric,
unitary are considered in detail in Section \ref{REST}. Observe that
in fact we give another prove of Theorem \ref{SchurAlg} (the
uniqueness of the function from ${\bf S}(\sM,\sN)$ with with given
its Schur parameters is proved in Section \ref{UNIQUENESS}). In
Section \ref{DIL} we obtain in the block operator CMV matrix form
the minimal unitary dilations of a contraction and the minimal
Naimark dilations of a semi-spectral measure on the unite circle.
Another and more complicated constructions of the minimal Naimark
dilation and a simple conservative realization for a function
$\Theta\in{\bf S}(\sM,\sN)$ by means of its Schur parameters are
given in \cite{Const1} and in \cite{KailBruck}, respectively (see
also \cite{BC}). Simple conservative realizations of scalar Schur
functions with operators $A,B$, $C$, and $D$ expressed via
corresponding Schur parameters have been obtained by V.~Duboboj
\cite{D1}.

We also prove in Section \ref{DIL} that a unitary operator $U$ in a
separable Hilbert space $\sK$ having a cyclic subspace $\sM$ ($
\cspan\{U^n \sM,\;n\in\dZ\}=\sK$) is unitarily equivalent to the
block operator CMV matrices $\cU_0$ and $\wt\cU_0$ constructed by
means of the Schur parameters of the function
$\Theta(\lambda)=\cfrac{1}{\lambda}\,(F^*_\sM(\bar\lambda)-I_\sM)(F^*_\sM(\bar\lambda)+I_\sM)^{-1},$
where $F_\sM(\lambda)=P_\sM(U+\lambda I_\sK)(U-\lambda
I_\sK)^{-1}\uphar\sM,$ $\lambda\in\dD$.

In the last Section \ref{MATRMOD} we prove that the Sz.-Nagy--Foias
\cite{SF} characteristic functions of truncated block operator CMV
matrices $\cT_0$ and $\wt\cT_0$, constructing by means of the Schur
parameters $\{\Gamma_n\}_{n\ge 0}$ of a purely contractive function
$\Theta\in{\bf S}(\sM,\sN)$, coincide with $\Theta$ in the sense of
\cite{SF}.

\section{The Schur class functions and their iterates}
\label{UNIQUENESS}

In the sequel we need the well known fact \cite{SF}, \cite{BC} that
if $T\in\bL(\sH_1,\sH_2)$ is a contraction which is neither
isometric nor co-isometric, then the operator (\textit{elementary
rotation} \cite{BC}) ${\bf J}_T$ given by the operator matrix
\[
{\bf J}_T=\begin{bmatrix} T &D_{T^*}\cr
D_T&-T^*\end{bmatrix}:\begin{array}{l}\sH_1\\\oplus\\\sD_{T^*}\end{array}\to
\begin{array}{l}\sH_2\\\oplus\\\sD_{T}\end{array}
\]
is unitary.  Clearly, ${\bf J}^*_{T}={\bf J}_{T^*}$.
 If $T$ is
isometric or co-isometric, then the corresponding unitary elementary
rotation takes the row or the column form
\[
{\bf J}^{(r)}_T=\begin{bmatrix} T
&I_{\sD_{T^*}}\end{bmatrix}:\begin{array}{l}\sH_1\\\oplus\\\sD_{T^*}\end{array}\to
\sH_2,\; {\bf J}^{(c)}_T =\begin{bmatrix} T \cr
D_T\end{bmatrix}:\sH_1\to\begin{array}{l}\sH_2\\\oplus\\\sD_{T}\end{array},
\]
and
\[
\left({\bf J}^{(r)}_T\right)^*={\bf J}^{(c)}_{T^*}.
\]
In Section \ref{BCMVR} we will need the following statement.
\begin{proposition}
\label{ranges} \cite{AHS2}. Let $T$ be a contraction. Then
$Th=D_{T^*}g$ if and only if there exists a vector $\f\in\sD_{T}$
such that $h=D_{T}\f$ and $g=T\f$.
\end{proposition}

Recall that if $\Theta(\lambda)\in{\bf S}(\sH_1,\sH_2)$ then there
is a uniquely determined decomposition \cite[Proposition V.2.1]{SF}
\[
\Theta(\lambda)=\begin{bmatrix}\Theta_p(\lambda)&0\cr
0&\Theta_u\end{bmatrix}:\begin{array}{l}\sD_{\Theta(0)}\\\oplus\\\ker
D_{\Theta(0)}\end{array}\to \begin{array}{l}
 \sD_{\Theta^*(0)}\\\oplus\\\ker D_{\Theta^*(0)}\end{array},
\]
where $\Theta_p(\lambda)\in{\bf
S}(\sD_{\Theta(0)},\sD_{\Theta^*(0)})$,  $\Theta_p(0)$ is a pure
contraction and $\Theta_u$ is a unitary constant. The function
$\Theta_p(\lambda)$ is called the \textit{pure part} of
$\Theta(\lambda)$ (see {\cite{BC}). If $\Theta(0)$ is isometric
(respect., co-isometric) then the pure part is of the form
$\Theta_p(\lambda)=0\in {\bf S}(\{0\},\sD_{\Theta^*(0)})$ (respect.,
$\Theta_p(\lambda)=0\in {\bf S}(\sD_{\Theta(0)}, \{0\})$). The
function $\Theta$ is called purely contractive if $\ker
D_{\Theta(0)}=\{0\}$. Two operator-valued functions $\Theta\in{\bf
S}(\sM,\sN)$ and $\Omega\in{\bf S}(\sK,\sL)$ \textit{coincide}
\cite{SF} if there are two unitary operators $V:\sN\to\sL$ and
$U:\sK\to\sM$ such that
\begin{equation}
\label{COINC} V\Theta(\lambda)U=\Omega(\lambda),\quad \lambda\in\dD.
\end{equation}
For the corresponding Schur parameters and the Schur iterates
relation \eqref{COINC} yields the equalities
\begin{equation}
\label{COINC1}
\begin{array}{l}
G_n=V\Gamma_nU,\\
 \sD_{G_n}=U^*\sD_{\Gamma_n},\;
\sD_{G^*_n}=V\sD_{\Gamma^*_n},\;
D_{G_n}=U^*D_{\Gamma_n}U,\;D_{G^*_n}=VD_{\Gamma^*_n}V^*,\\
V\Theta_n(\lambda)U=\Omega_n(\lambda),\;\lambda\in\dD
\end{array}
\end{equation}
for all $n=0,1,\ldots.$

 In what follows we
give a proof of Theorem \ref{CHSEQ} different from the original one
in \cite{Const}. First of all we will prove the uniqueness. The
existence will be proved in Section \ref{BCMVR}.
\begin{theorem}
\label{UNIQ} Any choice sequence uniquely determines a Schur class
function.
\end{theorem}
\begin{proof} Let $\Gamma_0\in\bL(\sM,\sN),$ $\Gamma_n\in
\bL(\sD_{\Gamma_{n-1}},\sD_{\Gamma^*_{n-1}})$, $n\ge 1$ be a choice
sequence. Suppose the functions $\Theta_0(\lambda)$ and $\wh
\Theta_0(\lambda)$ from the Schur class ${\bf S}(\sM,\sN)$ have
$\{\Gamma_n\}_0^\infty$ as their Schur parameters. Then for every
$n=0,1,\ldots$ hold the relations
\[
\begin{array}{l}
\Theta_n(\lambda)=\Gamma_n+\lambda D_{\Gamma^*_n}(I+\lambda\Theta_{n+1}(\lambda)\Gamma^*_n)^{-1}\Theta_{n+1}(\lambda)D_{\Gamma_n},\\
\wh \Theta_n(\lambda)=\Gamma_n+\lambda D_{\Gamma^*_n}(I+\lambda\wh
\Theta_{n+1}(\lambda)\Gamma^*_n)^{-1}\wh\Theta_{n+1}(\lambda)D_{\Gamma_n},
\end{array}
\]
where $\{\Theta_n(\lambda)\}$ and $\{\wh\Theta_n(\lambda)\}$ are the
Schur iterates of $\Theta$ and $\wh \Theta$, respectively. Then one
has for every $n$ the equalities
\begin{equation}
\label{diff}
\begin{array}{l}
\Theta_n(\lambda)-\wh\Theta_n(\lambda)=\\
\quad=\lambda
D_{\Gamma^*_n}(I+\lambda\Theta_{n+1}(\lambda)\Gamma^*_n)^{-1}(\Theta_{n+1}(\lambda)-\wh
\Theta_{n+1}(\lambda))(I+\lambda\wh
\Theta_{n+1}(\lambda)\Gamma^*_n)^{-1}D_{\Gamma_n},\;\lambda\in\dD.
\end{array}
\end{equation}
Since $||\Theta_{n+1}(\lambda)-\wh\Theta_{n+1}(\lambda)||\le 2$ for
all $\lambda\in\dD$ and
$\Theta_{n+1}(0)=\wh\Theta_{n+1}(0)=\Gamma_{n+1}$, by Schwartz's
lemma we get
\[
||\Theta_{n+1}(\lambda)-\wh\Theta_{n+1}(\lambda)||\le
2\,|\lambda|,\; \lambda\in\dD.
\]
Further
\[
\begin{array}{l}
||(I+\lambda\Theta_{n+1}(\lambda)\Gamma^*_n)f||\ge(1-|\lambda|)||f||,\\
||(I+\lambda\wh\Theta_{n+1}(\lambda)\Gamma^*_n)f||\ge(1-|\lambda|)||f||
\end{array}
\]
for all $\lambda\in\dD$ and for all $f\in\sD_{\Gamma^*_{n-1}}$.
These relations imply
\[
||(I+\lambda\Theta_{n+1}(\lambda)\Gamma^*_n)^{-1}||\le\frac{1}{1-|\lambda|},\;
||(I+\lambda\wh\Theta_{n+1}(\lambda)\Gamma^*_n)^{-1}||\le\frac{1}{1-|\lambda|}
\]
for all $\lambda\in\dD$ and for all $n=0,1,\ldots.$
 Hence and from \eqref{diff} we have
\[
||\Theta_{n}(\lambda)-\wh\Theta_{n}(\lambda)||\le
2|\lambda|\,\frac{|\lambda|}{(1-|\lambda|)^2},\;\lambda\in\dD.
\]
Then applying \eqref{diff} for $\Theta_{n-1}$ and $\wh\Theta_{n-1}$
in the left hand side, we see that
\[
||\Theta_{n-1}(\lambda)-\wh\Theta_{n-1}(\lambda)||\le
2|\lambda|\,\left(\frac{|\lambda|}{(1-|\lambda|)^2}\right)^2,\;\lambda\in\dD,
\]
and finally
\begin{equation}
\label{prelimit} ||\Theta_{0}(\lambda)-\wh\Theta_{0}(\lambda)||\le
2|\lambda|\,\left(\frac{|\lambda|}{(1-|\lambda|)^2}\right)^{n+1},\;\lambda\in\dD
\end{equation}
for all $n=0,1,\ldots$.

Let $|\lambda|<(3-\sqrt{5})/2$. Then
\[
\frac{|\lambda|}{(1-|\lambda|)^2}<1.
\]
Letting $n\to\infty$ in \eqref{prelimit} we get
\[
\Theta_0(\lambda)=\wh\Theta_0(\lambda),\;|\lambda|<\frac{3-\sqrt{5}}{2}.
\]
Since $\Theta_0$ and $\wh\Theta_0$ are holomorphic in $\dD$, they
are equal on $\dD$.
\end{proof}

\section{Conservative discrete-time linear systems and their
transfer functions}  \label{secS}
 Let
$\sM,\sN$, and $\sH$ be separable Hilbert spaces. A linear system
$\tau=\left\{\begin{bmatrix} D&C \cr B&A\end{bmatrix};\sM,\sN,
\sH\right\}$ with bounded linear operators $A$, $B$, $C$, $D$  of
the form
\begin{equation}
\label{passive}
 \left\{
 \begin{array}{l}
    \sigma_k=Ch_k+D\xi_k,\\
    h_{k+1}=Ah_k+B\xi_k
\end{array}
\right. \qquad k\ge 0,
\end{equation}
where $\{\xi_k\}\subset \sM$, $\{\sigma_k\}\subset \sN$,
$\{h_k\}\subset \sH$ is called a \textit{discrete time-invariant
system}. The Hilbert spaces $\sM$ and $\sN$ are called the input and
the output spaces, respectively, and the Hilbert space $\sH$ is
called the state space. The operators $A$, $B$, $C$, and $D$ are
called the state space operator, the control operator, the
observation operator, and the feedthrough operator of $\tau$,
respectively. Put
\[
U_\tau=\begin{bmatrix} D&C \cr B&A\end{bmatrix} :
\begin{array}{l} \sM \\\oplus\\ \sH \end{array} \to
\begin{array}{l} \sN \\\oplus\\ \sH \end{array}
\]
If $U_\tau$ is contractive, then the corresponding discrete-time
system is said to be \textit{passive} \cite{A}. If the operator
 $U_\tau$ is isometric (respect., co-isometric, unitary), then the
system is said to be \textit{isometric} (respect.,
\textit{co-isometric}, \textit{conservative}). Isometric,
co-isometric, conservative, and passive discrete time-invariant
systems have been studied in \cite{BrR1}, \cite{BrR2}, \cite{Ando},
\cite{SF}, \cite{Helton1}, \cite{Helton2}, \cite{VilBr}, \cite{Br1},
\cite{Ball-Coehn}, \cite{ADRS}, \cite{A}, \cite{Arov},
\cite{ArKaaP}, \cite{ArNu1}, \cite{ArNu2}, \cite{ArSt},
\cite{Staf1}, \cite{Staf2}, \cite{AHS1}, \cite{ARL1},
\cite{Arlarxiv}, \cite{FKatsKr}. It is relevant to remark that a
brief history of System Theory  is presented in the recent preprint
of B.~Fritzsche, V.~Katsnelson, and B.~Kirstein \cite{FKatsKr}.

The subspaces
\begin{equation}
\label{CO} \sH^c:=\cspan\{A^{n}B\sM:\,n=0,1,\ldots\} \mbox{ and }
\sH^o:=\cspan\{A^{*n}C^*\sN:\,n=0,1,\ldots\}
\end{equation}
are said to be the \textit{controllable} and \textit{observable}
subspaces of the system $\tau$, respectively. The system $\tau$ is
said to be \textit{controllable} (respect., \textit{observable}) if
$\sH^c=\sH$ (respect., $\sH^o=\sH$), and it is called
\textit{minimal} if $\tau$ is both controllable and observable. The
system $\tau$ is said to be \textit{simple} if
\[
\sH=\clos
\{\sH^c+\sH^o\}=\cspan\{A^kB\sM,\;A^{*l}C^*\sN,\;k,l=0,1,\ldots\}
\]
It follows from \eqref{CO} that
\[
 (\sH^c)^\perp=\bigcap\limits_{n=0}^\infty\ker(B^*A^{*n}),\quad
 (\sH^o)^\perp=\bigcap\limits_{n=0}^\infty\ker(CA^{n}),
\]
and therefore there are the following alternative characterizations:
\begin{enumerate}\def\labelenumi{\rm (\alph{enumi})}
\item $\tau$ is controllable $\iff
\;\bigcap\limits_{n=0}^\infty\ker(B^*A^{*n})=\{0\}$;
\item  $\tau$ is observable $\iff
\;\bigcap\limits_{n=0}^\infty\ker(CA^{n})=\{0\}$;
\item  $\tau$ is simple $\iff
\left(\bigcap\limits_{n=0}^\infty\ker(B^*A^{*n})\right)\cap
\left(\bigcap\limits_{n=0}^\infty\ker(CA^{n})\right)=\{0\}.$
\end{enumerate}
 A contraction $A$ acting in a Hilbert space $\sH$ is
called \textit{completely non-unitary} \cite{SF} if there is no
nontrivial reducing subspace of $A$, on which $A$ generates a
unitary operator. Given a contraction $A$ in $\sH$ then there is a
canonical orthogonal decomposition \cite [Theorem I.3.2]{SF}
\[
\sH=\sH_0\oplus \sH_1, \qquad A=A_0\oplus A_1, \quad A_j=A\uphar
\sH_j, \quad j=0,1,
\]
where $\sH_0$ and $\sH_1$ reduce $A$, the operator $A_0$ is a
completely non-unitary contraction, and $A_1$ is a unitary operator.
Moreover,
\[
\sH_1= \left(\bigcap\limits_{n\ge 1}\ker
D_{A^n}\right)\bigcap\left(\bigcap\limits_{n\ge 1}\ker
D_{A^{*n}}\right).
\]
Since
\[
\bigcap\limits_{k=0}^{n-1} \ker(D_{A}A^{k})=\ker D_{A^{n}},\;
\bigcap\limits_{k=0}^{n-1} \ker(D_{A^*}A^{*k})=\ker D_{A^{*n}},
\]
we get
\begin{equation}
\label{perp}
\begin{split}
&\bigcap\limits_{n\ge 1}\ker
D_{A^n}=\sH\ominus\cspan\left\{A^{*n}D_A \sH,\;
n=0,1,\ldots\right\},\\
&\bigcap\limits_{n\ge 1}\ker
D_{A^{*n}}=\sH\ominus\cspan\left\{A^nD_{A^*} \sH,\;
n=0,1,\ldots\right\}.
\end{split}
\end{equation}
It follows that
\begin{equation}
\label{cu}\begin{array}{l}
 A \;\mbox{is completely
non-unitary}\;\iff\left(\bigcap\limits_{n\ge 1}\ker
D_{A^n}\right)\bigcap\left(\bigcap\limits_{n\ge 1}\ker
D_{A^{*n}}\right)=\{0\}\iff\\
 \iff \cspan\{A^{*n}D_{A},\;A^mD_{A^*},\;n,m\ge
0\}=\sH.
\end{array}
\end{equation}
If $\tau=\left\{\begin{bmatrix} D&C \cr B&A\end{bmatrix};\sM,\sN,
\sH\right\}$ is a conservative system then
 $\tau$ is simple if and only if the state space operator $A$ is a
completely non-unitary contraction \cite{Br1}, \cite{Ball-Coehn}.

The \textit{transfer function}
\[
\Theta_\tau(\lambda):=D+\lambda C(I_{\sH}-\lambda
A)^{-1}B, \quad \lambda \in \dD,
\]
of the passive system $\tau$ belongs to the Schur class ${\bf
S}(\sM,\sN)$ \cite{A}. Conservative systems are also called the
unitary colligations  and their transfer functions are called the
characteristic functions \cite{Br1}.

The examples of conservative systems are given by
\[
\Sigma=\left\{\begin{bmatrix}-A& D_{A^*}\cr D_{A}& A^*\end{bmatrix};
 \sD_{A},\sD_{A^*},\sH\right\},\;\Sigma_*=\left\{\begin{bmatrix}-A^*&
D_{A}\cr D_{A^*}& A\end{bmatrix};  \sD_{A^*},\sD_{A},\sH\right\}.
\]
The transfer functions of these systems
\[
\Phi_\Sigma(\lambda)=\left(-A+\lambda D_{A^*}(I_\sH-\lambda
A^*)^{-1}D_{A}\right)\uphar\sD_{A}, \quad \lambda \in \dD
\]
and
\[
\Phi_{\Sigma_*}(\lambda)=\left(-A^*+\lambda D_{A}(I_\sH-\lambda
A)^{-1}D_{A^*}\right)\uphar\sD_{A^*}, \quad \lambda \in \dD
\]
are precisely the Sz.Nagy--Foias characteristic functions \cite{SF}
of $A$ and $A^*$, correspondingly.

It is well known that every operator-valued function
$\Theta(\lambda)$ from the Schur class ${\bf S}(\sM,\sN)$ can be
realized as the transfer function of some passive system, which can
be chosen as controllable isometric (respect., observable
co-isometric, simple conservative, minimal passive); cf.
\cite{BrR2}, \cite{SF}, \cite{Br1}, \cite{Ando} \cite{A},
\cite{ArKaaP}, \cite{ADRS}. Moreover, two controllable isometric
(respect., observable co-isometric, simple conservative) systems
with the same transfer function are \textit{unitarily equivalent}:
two discrete-time systems
\[
\tau_1=\left\{\begin{bmatrix} D&C_1 \cr
B_1&A_1\end{bmatrix};\sM,\sN,\sH_{1}\right\} \quad \mbox{and} \quad
\tau_2=\left\{\begin{bmatrix} D&C_2 \cr
B_2&A_2\end{bmatrix};\sM,\sN,\sH_{2}\right\}
\]
are said to be unitarily equivalent if there exists a unitary
operator $V$ from $\sH_{1}$ onto $\sH_{2}$ such that
\begin{equation}
\label{UNSYS}
\begin{array}{l}
A_1 =V^{-1}A_2V,\quad B_1=V^{-1}B_2,\quad C_1=C_2
V\iff\\
\iff\begin{bmatrix}I_\sN&0\cr 0&V \end{bmatrix}\begin{bmatrix} D&C_1
\cr B_1&A_1\end{bmatrix}= \begin{bmatrix} D&C_2 \cr
B_2&A_2\end{bmatrix}\begin{bmatrix}I_\sM&0\cr 0&V \end{bmatrix}
\end{array}
\end{equation}
cf. \cite{BrR1}, \cite{BrR2}, \cite{Ando}, \cite{Br1}, \cite{ADRS}.

\section{Conservative realizations of the Schur iterates}
\label{gener}
Let $A$ be a completely non-unitary contraction in a separable
Hilbert space $\sH$. Suppose $\ker D_A\ne \{0\}$. Define the
subspaces and operators (see \cite{Arlarxiv})
\begin{equation}
\label{hnm} \left\{
\begin{array}{l}
\sH_{0,0}:=\sH\\
 \sH_{n,0}=\ker D_{A^n},\;
  \sH_{0,m}:=\ker
D_{A^{*m}},\\
\sH_{n, m}:=\ker D_{A^{n}}\cap \ker D_{A^{*m}},\; m,n\in\dN,
\end{array}
\right.
\end{equation}
\begin{equation}
\label{Anm}
 A_{n, m}:=P_{n,
m}A\uphar\sH_{n, m}\in\bL(\sH_{n,m}),
\end{equation}
where $P_{n,m}$ are the orthogonal projections in $\sH$ onto
$\sH_{n,m}$. The next results have been established in
\cite{Arlarxiv}.
\begin{theorem}
\label{RELATT}  \cite{Arlarxiv}.
\begin{enumerate}
\item Hold the relations
\[
 \ker D_{A^k_{n,m}}=\sH_{n+k,m},\; \ker
 D_{A^{*k}_{n,m}}=\sH_{n,m+k},\;k=1,2,\ldots,
\]
\[
 \left\{\begin{array}{l}
\sD_{A_{n,m}}=\cran{(P_{n,m}D_{A^{n+1}})},\\
\sD_{A^*_{n,m}}=\cran{(P_{n,m}D_{A^{*{m+1}}})}
\end{array}\right.,
\]
\[
\left\{ \begin{array}{l} A\sH_{n,m}=\sH_{n-1,
m+1},\;n\ge 1,\\
 A^*\sH_{n,m}=\sH_{n+1, m-1},\; m\ge 1
 \end{array}\right.,
\]
\begin{equation}
\label{UE1}
 \left(A_{n,m}\right)_{k,l}=A_{n+k,m+l}.
\end{equation}
\item
 The operators $\left\{A_{n,m}\right\}$ are completely non-unitary
contractions.
\item
 The operators
 \[
A_{n,0},\; A_{n-1,1},\; \ldots, A_{n-k,k},\ldots,A_{0,n}
\]
are unitarily equivalent and 
\[
A_{n-1,m+1}Af=AA_{n,m}f,\;f\in\sH_{n,m},\; n\ge 1.
\]
\end{enumerate}
\end{theorem}
The relation \eqref{UE1} yields the following picture for the
creation of the operators $A_{n,m}$:

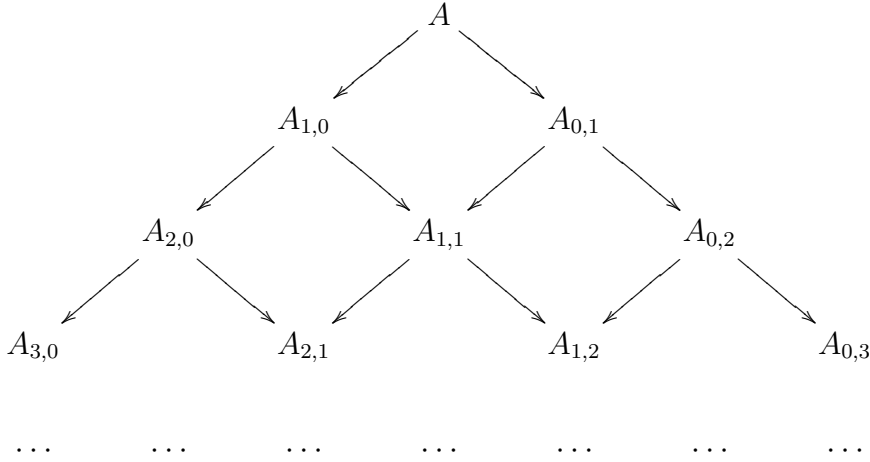
\xymatrix{
&&&A\ar[ld]\ar[rd]\\
&&A_{1,0}\ar[ld]\ar[rd]&&A_{0,1}\ar[ld]\ar[rd]\\
&A_{2,0}\ar[ld]\ar[rd]&&A_{1,1}\ar[ld]\ar[rd]&&A_{0,2}\ar[ld]\ar[rd]\\
A_{3,0}&&
A_{2,1}&&A_{1,2}&&A_{0,3}\\
\cdots&\cdots&\cdots&\cdots&\cdots&\cdots&\cdots}

The process terminates on the $N$-th step if and only if
\[
\begin{array}{l}
\ker D_{A^N}=\{0\}\iff \ker D_{A^{N-1}}\cap\ker
D_{A^*}=\{0\}\iff\ldots\\ \ker D_{A^{N-k}}\cap\ker
D_{A^{*k}}=\{0\}\iff\ldots \ker D_{A^{*N}}=\{0\}.
\end{array}
\]
\begin{theorem}
\label{ITERATES} \cite{Arlarxiv}. Let $A$ be a completely
non-unitary contraction in a separable Hilbert space $\sH$. Assume
$\ker D_A\ne\{0\}$ and let the contractions $A_{n,m}$ be defined by
\eqref{hnm} and \eqref{Anm}. Then the characteristic functions of
the operators
\[
A_{n,0},A_{n-1, 1},\ldots, A_{n-m, m},\ldots A_{1,n-1}, A_{0, n}
\]
coincide with the pure part of the $n$-th Schur iterate of the
characteristic function $\Phi(\lambda)$ of $A$. Moreover, each
operator from the set $\{A_{n-k,k}\}_{k=0}^n$ is
\begin{enumerate}
\item  a unilateral shift (respect., co-shift) if and only if the $n$-th Schur
parameter $\Gamma_{n}$ of $\Phi$ is isometric (respect.,
co-isometric),
\item  the orthogonal sum of  a unilateral shift and
co-shift if and only if
\begin{equation}
\label{SHCOSH} \sD_{\Gamma_{n-1}}\ne \{0\},\;
\sD_{\Gamma^*_{n-1}}\ne \{0\}\quad\mbox{and}\quad
\Gamma_m=0\quad\mbox{for all}\quad m\ge n.
\end{equation}
\end{enumerate}
 Each subspace from the set $\{\sH_{n-k,k}\}_{k=0}^n$ is trivial
  if and only if $\Gamma_n$ is unitary.
\end{theorem}
\begin{theorem}
\label{ITERATES11} \cite{Arlarxiv}. Let $\Theta(\lambda)\in {\bf
S}(\sM,\sN)$ and let
\[
\tau_0=\left\{\begin{bmatrix}\Gamma_0&C\cr
B&A\end{bmatrix};\sM,\sN,\sH\right\}
\]
be a simple conservative realization of $\Theta$. Then the Schur
parameters $\{\Gamma_n\}_{n \ge 1}$ of $\Theta$ can be calculated as
follows
\begin{equation}
\label{gamman}
\begin{array}{l}
\Gamma_1=D^{-1}_{\Gamma^*_0}C\left(D^{-1}_{\Gamma_0}B^*\right)^*,�\;
\Gamma_2=D^{-1}_{\Gamma^*_1}D^{-1}_{\Gamma^*_0}CA
\left(D^{-1}_{\Gamma_1}D^{-1}_{\Gamma_0}\left(B^*\uphar\sH_{1,0}\right)\right)^*,\ldots,\\
\Gamma_n=D^{-1}_{\Gamma^*_{n-1}}\cdots
D^{-1}_{\Gamma^*_0}CA^{n-1}\left(D^{-1}_{\Gamma_{n-1}}\cdots
D^{-1}_{\Gamma_0}\left(B^*\uphar\sH_{n-1,0}\right)\right)^*,\ldots.
\end{array}
\end{equation}
Here the operators $D^{-1}_{\Gamma_k}$ and $D^{-1}_{\Gamma^*_k},
k=0,1,\dots$ are the Moore-Penrose pseudo inverses, the operator
\[
\left(D^{-1}_{\Gamma_{n-1}}\cdots
D^{-1}_{\Gamma_0}\left(B^*\uphar\sH_{n-1,0}\right)\right)^*\in\bL(\sD_{\Gamma_{n-1}},\sH_{n-1,0})
\]
is  the adjoint to the operator
\[
D^{-1}_{\Gamma_{n-1}}\cdots
D^{-1}_{\Gamma_0}\left(B^*\uphar\sH_{n-1,0}\right)\in\bL(\sH_{n-1,0},\sD_{\Gamma_{n-1}}),
\]
and
\[\begin{array}{l}
\ran\left(D^{-1}_{\Gamma_{n-1}}\cdots
D^{-1}_{\Gamma_0}\left(B^*\uphar\sH_{n,0}\right) \right)\subset\ran
D_{\Gamma_n},\\
\ran\left(D^{-1}_{\Gamma^*_{n-1}}\cdots
D^{-1}_{\Gamma^*_0}\left(C\uphar\sH_{0,n}\right)\right)\subset\ran
D_{\Gamma^*_n}
\end{array}
\]
for every $n\ge 1$.
 Moreover, for each $n\ge 1$ the unitarily equivalent simple conservative systems
\begin{equation}
\label{taun}
\begin{array}{l}
\tau^{(k)}_{n}=\left\{\begin{bmatrix}\Gamma_n&D^{-1}_{\Gamma^*_{n-1}}
\cdots D^{-1}_{\Gamma^*_{0}}(CA^{n-k})\cr
A^k\left(D^{-1}_{\Gamma_{n-1}}\cdots D^{-1}_{\Gamma_{0}}
\left(B^*\uphar\sH_{n,0}\right)\right)^*&A_{n-k,k}\end{bmatrix};
\sD_{\Gamma_{n-1}},\sD_{\Gamma^*_{n-1}}, \sH_{n-k,k}\right\},\\
k=0,1,\ldots,n \end{array}
\end{equation}
are realizations of the $n$-th Schur iterate $\Theta_n$ of $\Theta$.
 Here the operator
\[
B_n=\left(D^{-1}_{\Gamma_{n-1}}\cdots D^{-1}_{\Gamma_{0}}
\left(B^*\uphar\sH_{n,0}\right)\right)^*\in\bL(\sD_{\Gamma_{n-1}},\sH_{n,0})
\]
is the adjoint to the operator
\[
D^{-1}_{\Gamma_{n-1}}\cdots D^{-1}_{\Gamma_{0}}
\left(B^*\uphar\sH_{n,0}\right)\in\bL(\sH_{n,0},\sD_{\Gamma_{n-1}}).
\]
\end{theorem}

Note that if

\noindent 1) $\Gamma_m$ is isometric then  $\sD_{\Gamma_n}=0$,
$\Gamma^*_n=0\in\bL(\sD_{\Gamma^*_m},\{0\})$,
$\sD_{\Gamma^*_n}=\sD_{\Gamma^*_m}$, and $\sH_{0,n}=\sH_{0,m}$ for
$n\ge m$. The unitarily equivalent observable conservative systems
\[
\tau^{(k)}_m=\left\{\begin{bmatrix}\Gamma_m&D^{-1}_{\Gamma^*_{m-1}}
\cdots D^{-1}_{\Gamma^*_{0}}(CA^{m-k})\cr 0&A_{m-k,k}\end{bmatrix};
\sD_{\Gamma_{m-1}},\sD_{\Gamma^*_{m-1}}, \sH_{m-k,k}\right\},\;
k=0,1,\ldots,m
\]
have transfer functions $\Theta_m(\lambda)=\Gamma_m$ and the
operators $A_{m-k,k}$ are unitarily equivalent co-shifts of
multiplicity $\dim\sD_{\Gamma^*_m},$ the Schur iterates $\Theta_n$
are null operators from $\bL(\{0\},\sD_{\Gamma^*_m})$ for $n\ge m+1$
and are transfer functions of the conservative observable system
\[
\tau_{m+1}=\left\{\begin{bmatrix}0&D^{-1}_{\Gamma^*_{m-1}} \cdots
D^{-1}_{\Gamma^*_{0}}C\cr 0&A_{0,m}\end{bmatrix};
\{0\},\sD_{\Gamma^*_{m}}, \sH_{0,m}\right\}.
\]

\noindent 2) $\Gamma_m$ is co-isometric then $\sD_{\Gamma^*_n}=0$,
$\sD_{\Gamma_n}=\sD_{\Gamma_m}$, and
$\Gamma_n=0\in\bL(\sD_{\Gamma_m},\{0\})$, $\sH_{n,0}=\sH_{m,0}$ for
$n\ge m$. The unitarily equivalent controllable conservative systems
\[
\tau^{(k)}_m=\left\{\begin{bmatrix}\Gamma_m&0\cr
A^k\left(D^{-1}_{\Gamma_{m-1}}\cdots D^{-1}_{\Gamma_{0}}
\left(B^*\uphar\sH_{m,0}\right)\right)^*&A_{m-k,k}\end{bmatrix};
\sD_{\Gamma_{m-1}},\sD_{\Gamma^*_{m-1}}, \sH_{m-k,k}\right\}
\]
have transfer functions $\Theta_m(\lambda)=\Gamma_m$ and the
operators $A_{m-k,k}$ are unitarily equivalent unilateral shifts of
multiplicity $\dim\sD_{\Gamma_m},$ the Schur iterates $\Theta_n$ are
null operators from $\bL(\sD_{\Gamma_m},\{0\})$ for $n\ge m+1$ and
are transfer functions of the conservative controllable system
\[
\tau_{m+1}=\left\{\begin{bmatrix}0&0\cr
\left(D^{-1}_{\Gamma_{m}}\cdots D^{-1}_{\Gamma_{0}}
\left(B^*\uphar\sH_{m+1,0}\right)\right)^* &A_{m,0}\end{bmatrix};
\sD_{\Gamma_{m}},\{0\}, \sH_{m,0}\right\}.
\]
We also mention  that if
 $\Theta(\lambda)\in{\bf S}(\sM,\sN)$,
$\Gamma_0=\Theta(0)$,
 $\Theta_1(\lambda)$ is the first Schur
iterate of $\Theta$, and if
\[
\tau=\left\{\begin{bmatrix}\Gamma_0&C\cr
B&A\end{bmatrix};\sM,\sN,\sH\right\}
\]
is a simple conservative system with transfer function $\Theta$,
then the simple conservative  systems
\begin{equation}\label{RealFirst}
\begin{array}{l}
\zeta_{0,1}=\left\{\begin{bmatrix}D^{-1}_{\Gamma^*_0}C(D^{-1}_{\Gamma_0}B^*)^*&D^{-1}_{\Gamma^*_0}C\uphar\ker
D_{A^*} \cr AP_{\ker D_A}D^{-1}_{A^*}B&P_{\ker D_{A^*}}A\uphar\ker
D_{A^*}\end{bmatrix};\sD_{\Gamma_0},\sD_{\Gamma^*_0},\ker
D_{A^*}\right\},\\
\zeta_{1,0}=\left\{\begin{bmatrix}D^{-1}_{\Gamma^*_0}C(D^{-1}_{\Gamma_0}B^*)^*&
D^{-1}_{\Gamma^*_0}CA\uphar\ker {D_A}\cr P_{\ker
D_A}D^{-1}_{A^*}B&P_{\ker D_A}A\uphar\ker
D_A\end{bmatrix};\sD_{\Gamma_0},\sD_{\Gamma^*_0},\ker D_A\right\}
 \end{array}
\end{equation}
have transfer functions $\Theta_1(\lambda)$ (see \cite{Arlarxiv}).
Here the operators $D^{-1}_{\Gamma_0},$ $D^{-1}_{\Gamma^*_0},$ and
$D^{-1}_{A^*}$ are the Moore--Penrose pseudo-inverses. In the sequel
the transformations of the conservative system
\[
\tau\to\zeta_{0,1},\; \tau\to\zeta_{1,0}
\]
we will denote by $\Omega_{0,1}(\tau)$ and $\Omega_{1,0}(\tau)$,
respectively.
\begin{remark}
The problem of isometric, co-isometric, and conservative
realizations of the Schur iterates for a scalar function from the
generalized Schur class has been studied in \cite{AADL1},
\cite{AADLW1}, \cite{AADLW2}, \cite{ADL}. For a scalar finite
Blaschke product the realizations of the Schur iterates are
constructed in \cite{FKatsKr}.
\end{remark}

\section{Block operator CMV matrices and conservative realizations of the Schur class function
(the case when the operator $\Gamma_n$ is neither an isometry nor a
co-isometry for each $n$)} \label{BCMVR}

  Let
\[
\Gamma_0\in\bL(\sM,\sN),\;\Gamma_n\in
\bL(\sD_{\Gamma_{n-1}},\sD_{\Gamma^*_{n-1}}),\; n\ge 1
\]
be a choice sequence. In this and next Section \ref{REST} we are
going to construct by means of $\{\Gamma_n\}_{n\ge 0}$ two unitary
equivalent simple conservative systems with such transfer function
$\Theta\in {\bf S}(\sM,\sN)$ that $\{\Gamma_n\}_{n\ge 0}$ are its
Schur parameters. In particular, this leads to the existence part of
Theorem \ref{SchurAlg} and to the well known result that any
$\Theta\in {\bf S}(\sM,\sN)$ admits a realization as the transfer
function of a simple conservative system.
 We begin with
constructions of block operator CMV matrices for given choice
sequence $\{\Gamma_n\}_{n\ge 0}$ and will suppose that all operators
$\Gamma_n$ are neither isometries nor co-isometries. We will use the
well known constructions of finite and infinite orthogonal sums of
Hilbert spaces. Namely, if $\{H_k\}_{k=1}^\infty$ is a given
sequence of Hilbert spaces, then
\[
\sH=\sum\limits_{k=1}^N \bigoplus H_k
\]
is the Hilbert space with the inner product
\[
(f,g)=\sum\limits_{k=0}^N(f_k,g_k)_{H_k}
\]
for $f=(f_1, \ldots, f_N)^T$ and $g=(g_1, \ldots, g_N)^T$, $f_k,
g_k\in H_k,$ $k=1,\ldots,N$ and the
 norm
\[
||f||^2=\sum\limits_{k=0}^N||f_k||^2_{H_k}.
\]
 The Hilbert space
\[
\sH=\sum\limits_{k=0}^\infty \bigoplus H_k
\]
consists of all vectors of the form $f=(f_1,f_2, \ldots)^T,$ $f_k\in
H_k$, $k=1,2,\ldots,$ such that
\[
||f||^2=\sum\limits_{k=1}^\infty ||f_k||^2_{H_k}<\infty.
\]
The inner product is given by
\[
(f,g)=\sum\limits_{k=1}^\infty(f_k,g_k)_{H_k}.
\]
\subsection{Block operator CMV matrices}
Define the Hilbert spaces
\begin{equation}
\label{statespaces}
\sH_0=\sH_0(\{\Gamma_n\}_{n\ge 0}):=\sum\limits_{n\ge
0}\bigoplus\begin{array}{l}\sD_{\Gamma_{2n}}\\\oplus\\\sD_{\Gamma^*_{2n+1}}\end{array},\;
\wt\sH_0=\wt\sH_0(\{\Gamma_n\}_{n\ge 0}):=\sum\limits_{n\ge
0}\bigoplus\begin{array}{l}\sD_{\Gamma^*_{2n}}\\\oplus\\\sD_{\Gamma_{2n+1}}\end{array}.
\end{equation}
From these definitions it follows, that
\[
\wt\sH_0(\{\Gamma^*_n\}_{n\ge 0})=\sH_0(\{\Gamma_n\}_{n\ge 0}),\;
\sH_0(\{\Gamma^*_n\}_{n\ge 0})=\wt\sH_0(\{\Gamma_n\}_{n\ge 0}).
\]
The spaces $\sN\bigoplus\sH_0$ and $\sM\bigoplus\wt\sH_0$ we
represent in the form
\[
\sN\bigoplus\sH_0=\begin{array}{l}\sM\\\oplus\\\sD_{\Gamma_0}\end{array}\bigoplus\sum\limits_{n\ge
1}\bigoplus\begin{array}{l}\sD_{\Gamma^*_{2n-1}}\\\oplus\\\sD_{\Gamma_{2n}}\end{array},
\]
\[
\sM\bigoplus\wt\sH_0=\begin{array}{l}\sM\\\oplus\\\sD_{\Gamma^*_0}\end{array}\bigoplus\sum\limits_{n\ge
1}\bigoplus\begin{array}{l}\sD_{\Gamma_{2n-1}}\\\oplus\\\sD_{\Gamma^*_{2n}}\end{array}.
\]
Let
\[\begin{array}{l}
{\bf J}_{\Gamma_0}=\begin{bmatrix} \Gamma_0& D_{\Gamma^*_0}\cr
D_{\Gamma_0}&-\Gamma^*_0\end{bmatrix}:\begin{array}{l}\sM\\\oplus\\\sD_{\Gamma^*_0}\end{array}\to
\begin{array}{l}\sN\\\oplus\\\sD_{\Gamma_{0}}\end{array},\\
{\bf J}_{\Gamma_k}=\begin{bmatrix} \Gamma_k& D_{\Gamma^*_k}\cr
D_{\Gamma_k}&-\Gamma^*_k\end{bmatrix}:\begin{array}{l}\sD_{\Gamma_{k-1}}\\\oplus\\\sD_{\Gamma^*_k}\end{array}\to
\begin{array}{l}\sD_{\Gamma^*_{k-1}}\\\oplus\\\sD_{\Gamma_{k}}\end{array},\;
k=1,2,\ldots
\end{array}
\]
 be the elementary rotations. Define the following unitary operators
\begin{equation}\label{OLM}
\begin{array}{l}
\cM_0=\cM_0(\{\Gamma_n\}_{n\ge 0}):=I_\sM\bigoplus\sum\limits_{n\ge
1}\bigoplus{\bf
J}_{\Gamma_{2n-1}}:\sM\bigoplus\sH_0\to \sM\bigoplus\wt\sH_0,\\
\wt \cM_0=\wt\cM_0(\{\Gamma_n\}_{n\ge
0}):=I_\sN\bigoplus\sum\limits_{n\ge 1}\bigoplus{\bf
J}_{\Gamma_{2n-1}}:\sN\bigoplus\sH_0\to\sN\bigoplus\wt\sH_0,\\
\cL_0=\cL_0(\{\Gamma_n\}_{n\ge 0}):={\bf
J}_{\Gamma_0}\bigoplus\sum\limits_{n\ge 1}\bigoplus{\bf
J}_{\Gamma_{2n}}:\sM\bigoplus\wt\sH_0\to\sN\bigoplus\sH_0.\\
\end{array}
\end{equation}
Observe that
\[
\left(\cL_0(\{\Gamma_n\}_{n\ge
0})\right)^*=\cL_0(\{\Gamma^*_n\}_{n\ge 0}).
\]
 Let
\begin{equation}
\label{V0} \cV_0=\cV_0(\{\Gamma_n\}_{n\ge 0}):=\sum\limits_{n\ge
1}\bigoplus{\bf J}_{\Gamma_{2n-1}}:\sH_0\to\wt\sH_0.
\end{equation}
Clearly, the operator $\cV_0$ is unitary and
\begin{equation}
\label{MOV}
\cM_0=I_\sM\bigoplus\cV_0,\;\wt\cM_0=I_\sN\bigoplus\cV_0.
\end{equation}
It follows that
\[
\left(\wt\cM_0(\{\Gamma_n\}_{n\ge
0})\right)^*=\cM_0(\{\Gamma^*_n\}_{n\ge
0}),\;\left(\cM_0(\{\Gamma_n\}_{n\ge
0})\right)^*=\wt\cM_0(\{\Gamma^*_n\}_{n\ge 0})
\]
Finally define the unitary operators
\begin{equation}
\label{defcmv}
\begin{array}{l}
\cU_0=\cU_0(\{\Gamma_n\}_{n\ge 0}):=\cL_0\cM_0:\sM\bigoplus\sH_0\to\sN\bigoplus\sH_0,\\
\wt\cU_0=\wt\cU_0(\{\Gamma_n\}_{n\ge
0}):=\wt\cM_0\cL_0:\sM\bigoplus\wt\sH_0\to\sN\bigoplus\wt\sH_0.
\end{array}
\end{equation}
By calculations we get
\[
\cU_0=
\begin{bmatrix}
\Gamma_0&D_{\Gamma^*_0}\Gamma_1&D_{\Gamma^*_0}D_{\Gamma^*_1}&0&0&0&0&0&\ldots&\ldots\cr
D_{\Gamma_0}&-\Gamma^*_0\Gamma_1&-\Gamma^*_0D_{\Gamma^*_1}&0&0&0&0&0&\ldots&\ldots\cr
0&\Gamma_2D_{\Gamma_1}&-\Gamma_2\Gamma^*_1&D_{\Gamma^*_2}\Gamma_3&D_{\Gamma^*_2}D_{\Gamma^*_3}&0&0&0&\ldots&\ldots\cr
0&D_{\Gamma_2}D_{\Gamma_1}&-D_{\Gamma_2}\Gamma^*_1&-\Gamma^*_2\Gamma_3&-\Gamma^*_2D_{\Gamma^*_3}&0&0&0&\ldots&\ldots\cr
0&0&0&\Gamma_4D_{\Gamma_3}&-\Gamma_4\Gamma^*_3&D_{\Gamma^*_4}\Gamma_5&D_{\Gamma^*_4}D_{\Gamma^*_5}&0&\dots&\ldots\cr
0&0&0&D_{\Gamma_4}D_{\Gamma_3}&-D_{\Gamma_4}\Gamma^*_3&-\Gamma^*_4\Gamma_5&-\Gamma_4D_{\Gamma^*_5}&0&\ldots&\ldots\cr
0&0&0&0&0&\Gamma_6D_{\Gamma_5}&-\Gamma_6\Gamma^*_5&D_{\Gamma^*_6}\Gamma_7&D_{\Gamma^*_6}D_{\Gamma^*_7}&\ldots\cr
\vdots&\vdots&\vdots&\vdots&\vdots&\vdots&\vdots&\vdots&\vdots&\vdots
\end{bmatrix}
\]
and
\[
 \wt\cU_0=
\begin{bmatrix}
\Gamma_0&D_{\Gamma^*_0}&0&0&0&0&0&0&\ldots&\ldots\cr
\Gamma_1D_{\Gamma_0}&-\Gamma_1\Gamma^*_0&D_{\Gamma^*_1}\Gamma_2&D_{\Gamma^*_1}D_{\Gamma^*_2}&0&0&0&0&\ldots&\ldots\cr
D_{\Gamma_1}D_{\Gamma_0}&-D_{\Gamma_1}\Gamma^*_0&-\Gamma^*_1\Gamma_2&-\Gamma^*_1D_{\Gamma^*_2}&0&0&0&0&\ldots&\ldots\cr
0&0&\Gamma_3D_{\Gamma_2}&-\Gamma_3\Gamma^*_2&D_{\Gamma^*_3}\Gamma_4&D_{\Gamma^*_3}D_{\Gamma^*_4}&0&0&\ldots&\ldots\cr
0&0&D_{\Gamma_3}D_{\Gamma_2}&-D_{\Gamma_3}\Gamma^*_2&-\Gamma^*_3\Gamma_4&-\Gamma^*_3D_{\Gamma^*_4}&0&0&\ldots&\ldots\cr
0&0&0&0&\Gamma_5D_{\Gamma_4}&-\Gamma_5\Gamma^*_4&D_{\Gamma^*_5}\Gamma_6&D_{\Gamma^*_5}D_{\Gamma^*_6}&0&\ldots\cr
0&0&0&0&D_{\Gamma_5}D_{\Gamma_4}&-D_{\Gamma_5}\Gamma^*_4&-\Gamma^*_5\Gamma_6&-\Gamma^*_5D_{\Gamma^*_6}&0&\ldots\cr
\vdots&\vdots&\vdots&\vdots&\vdots&\vdots&\vdots&\vdots&\vdots&\vdots
\end{bmatrix}.
\]
 Let
\[
\cC_0=\begin{bmatrix}D_{\Gamma^*_0}\Gamma_1&D_{\Gamma^*_0}D_{\Gamma^*_1}\end{bmatrix}:
\begin{array}{l}\sD_{\Gamma_{0}}\\\oplus\\\sD_{\Gamma^*_{1}}\end{array}
\to \sN,\;\cA_0=\begin{bmatrix} D_{\Gamma_0}\cr 0
\end{bmatrix}:\sM\to
\begin{array}{l}\sD_{\Gamma_{0}}\\\oplus\\\sD_{\Gamma^*_{1}}\end{array},
\]
\begin{equation}
\label{BLOKIT} \left\{
\begin{array}{l}
\cB_n=\begin{bmatrix}-\Gamma^*_{2n-2}\Gamma_{2n-1}&-\Gamma^*_{2n-2}D_{\Gamma^*_{2n-1}}\cr
\Gamma_{2n}D_{\Gamma_{2n-1}}&-\Gamma_{2n}\Gamma^*_{2n-1}\end{bmatrix}:\begin{array}{l}\sD_{\Gamma_{2n-2}}\\\oplus\\\sD_{\Gamma^*_{2n-1}}\end{array}
\to\begin{array}{l}\sD_{\Gamma_{2n-2}}\\\oplus\\\sD_{\Gamma^*_{2n-1}}\end{array},\\
\cC_n=\begin{bmatrix}0&0\cr
D_{\Gamma^*_{2n}}\Gamma_{2n+1}&D_{\Gamma^*_{2n}}D_{\Gamma^*_{2n+1}}\end{bmatrix}:\begin{array}{l}\sD_{\Gamma_{2n}}\\\oplus\\\sD_{\Gamma^*_{2n+1}}\end{array}
\to\begin{array}{l}\sD_{\Gamma_{2n-2}}\\\oplus\\\sD_{\Gamma^*_{2n-1}}\end{array},\\
\cA_n=\begin{bmatrix}D_{\Gamma_{2n}}D_{\Gamma_{2n-1}}&
-D_{\Gamma_{2n}}\Gamma^*_{2n-1}\cr
0&0\end{bmatrix}:\begin{array}{l}\sD_{\Gamma_{2n-2}}\\\oplus\\\sD_{\Gamma^*_{2n-1}}\end{array}
\to\begin{array}{l}\sD_{\Gamma_{2n}}\\\oplus\\\sD_{\Gamma^*_{2n+1}}\end{array},
\end{array}\right.
\end{equation}
\[
\wt
C_0=\begin{bmatrix}D_{\Gamma^*_0}&0\end{bmatrix}:\begin{array}{l}\sD_{\Gamma^*_0}\\\oplus\\\sD_{\Gamma_1}\end{array}\to\sN
,\;\wt\cA_0=\begin{bmatrix}\Gamma_1 D_{\Gamma_0}\cr
D_{\Gamma_1}D_{\Gamma_0}\end{bmatrix}:\sM\to\begin{array}{l}\sD_{\Gamma^*_0}\\\oplus\\\sD_{\Gamma_1}\end{array},
\]
\begin{equation}
\label{BLOKIWT} \left\{
\begin{array}{l}
 \wt\cB_n=\begin{bmatrix}
-\Gamma_{2n-1}\Gamma^*_{2n-2}&D_{\Gamma^*_{2n-1}}\Gamma_{2n}\cr
-D_{\Gamma_{2n-1}}\Gamma^*_{2n-2}&-\Gamma^*_{2n-1}\Gamma_{2n}\end{bmatrix}:\begin{array}{l}\sD_{\Gamma^*_{2n-2}}\\\oplus\\\sD_{\Gamma_{2n-1}}\end{array}\to
\begin{array}{l}\sD_{\Gamma^*_{2n-2}}\\\oplus\\\sD_{\Gamma_{2n-1}}\end{array},\\
\wt\cC_n=\begin{bmatrix} D_{\Gamma^*_{2n-1}}D_{\Gamma^*_{2n}}&0\cr
-\Gamma^*_{2n-1}D_{\Gamma^*_{2n}}&0\end{bmatrix}:\begin{array}{l}\sD_{\Gamma^*_{2n}}\\\oplus\\\sD_{\Gamma_{2n+1}}\end{array}\to
\begin{array}{l}\sD_{\Gamma^*_{2n-2}}\\\oplus\\\sD_{\Gamma_{2n-1}}\end{array},\\
\wt\cA_n=\begin{bmatrix}0&\Gamma_{2n+1}D_{\Gamma_{2n}}\cr
0&D_{\Gamma_{2n+1}}D_{\Gamma_{2n}}\end{bmatrix}:\begin{array}{l}\sD_{\Gamma^*_{2n-2}}\\\oplus\\\sD_{\Gamma_{2n-1}}\end{array}\to
\begin{array}{l}\sD_{\Gamma^*_{2n}}\\\oplus\\\sD_{\Gamma_{2n+1}}\end{array}
\end{array}\right.
\end{equation}
It is easy to see that the operators $\cU_0$ and $\wt\cU_0$ take the
following three-diagonal block operator matrix form
\[
  \cU_0=\cU_0\left(\{\Gamma_n\}_{n\ge 0}\right)=\begin{bmatrix} \Gamma_0 & \cC_0 & 0 &0   & 0 &
\cdot &
\cdot  \\
\cA_0 & \cB_1 & \cC_1 & 0 &0& \cdot &
\cdot  \\
0    & \cA_1 & \cB_2 & \cC_2 &0& \cdot &
\cdot   \\
\vdots & \vdots & \vdots & \vdots & \vdots & \vdots & \vdots
\end{bmatrix},
\]
\[
  \wt\cU_0=\wt\cU_0\left(\{\Gamma_n\}_{n\ge 0}\right)=\begin{bmatrix} \Gamma_0 & \wt\cC_0 & 0 &0   & 0 &
\cdot &
\cdot  \\
\wt\cA_0 & \wt\cB_1 & \wt\cC_1 & 0 &0& \cdot &
\cdot  \\
0    & \wt\cA_1 & \wt\cB_2 & \wt\cC_2 &0& \cdot &
\cdot   \\
\vdots & \vdots & \vdots & \vdots & \vdots & \vdots & \vdots
\end{bmatrix}.
\]
The block operator matrices $\cU_0$ and $\wt\cU_0$ we will call
\textit{block operator CMV matrices}.
 Observe that
\begin{equation}
\label{MU} \wt M_0\cU_0=\wt\cU_0\cM_0,
\end{equation}
and the following equalities hold true
\begin{equation}
\label{ADJ1} \left(\cU_0(\{\Gamma_n\}_{n\ge
0})\right)^*=\wt\cU_0(\{\Gamma^*_n\}_{n\ge
0}),\;\left(\wt\cU_0(\{\Gamma_n\}_{n\ge
0})\right)^*=\cU_0(\{\Gamma^*_n\}_{n\ge 0})
\end{equation}
Therefore the matrix $\wt\cU_0$ can be obtained from $\cU_0$ by
passing to the adjoint $\cU^*_0$ and then by replacing $\Gamma_n$
(respect., $\Gamma^*_n$) by $\Gamma^*_n$ (respect., $\Gamma_n$) for
all $n$. In the case when the choice sequence consists of complex
numbers from the unit disk the matrix $\wt\cU_0$ is the transpose to
$\cU_0$, i.e., $\wt\cU_0=\cU^t_0$.

\subsection{Truncated block operator CMV matrices}\label{sectrunc }
 Define two contractions
\begin{equation}
\label{TRUNC}\cT_0= \cT_0(\{\Gamma_n\}_{n\ge
0}):=P_{\sH_0}\cU_0\uphar\sH_0:\sH_0\to\sH_0,
\end{equation}
\begin{equation}
\label{TRUNCT} \wt\cT_0=\wt\cT_0(\{\Gamma_n\}_{n\ge
0}):=P_{\wt\sH_0}\wt\cU_0\uphar\wt\sH_0:\wt\sH_0\to\wt\sH_0.
\end{equation}
The operators $\cT_0$ and $\wt\cT_0$ take on the three-diagonal
block operator matrix forms
\[
  \cT_0=\begin{bmatrix}
\cB_1 & \cC_1 & 0 &0& 0&
\cdot  \\
 \cA_1 & \cB_2 & \cC_2 &0& 0 &
\cdot   \\
0&\cA_2&\cB_3&\cC_3&0&\cdot\\
 \vdots & \vdots & \vdots & \vdots & \vdots
& \vdots
\end{bmatrix},\;
  \wt\cT_0=\begin{bmatrix}
\wt\cB_1 & \wt\cC_1 & 0 &0& 0&
\cdot  \\
 \wt\cA_1 & \wt\cB_2 & \wt\cC_2 &0& 0 &
\cdot   \\
0&\wt\cA_2&\wt\cB_3&\wt\cC_3&0&\cdot\\
 \vdots & \vdots & \vdots & \vdots & \vdots
& \vdots
\end{bmatrix},
\]
where $\cA_n,\cB_n,\cC_n$, $\wt\cA_n,\wt\cB_n$, and $\wt\cC_n$ are
given by \eqref{BLOKIT} and \eqref{BLOKIWT}. Since the matrices
$\cT_0$ and $\wt\cT_0$ are obtained from $\cU_0$ and $\wt\cU_0$ by
deleting the first rows and the first columns, we will call them
\textit{truncated block operator CMV matrices}. Observe that from
the definitions of $\cL_0$, $\cM_0$, $\wt\cM_0,$ $\cT_0$, and
$\wt\cT_0$ it follows that $\cT_0$ and $\wt\cT_0$ are products of
two block-diagonal matrices
\begin{equation}
\label{T0prod} \cT_0=\cT_0(\{\Gamma_n\}_{n\ge
0})=\begin{bmatrix}-\Gamma^*_0\cr &{\bf J}_{\Gamma_2}\cr & &{\bf
J}_{\Gamma_4}\cr & &&\ddots\cr &&&& {\bf J}_{\Gamma_{2n}}\cr
&&&&&\ddots
\end{bmatrix}\begin{bmatrix}{\bf J}_{\Gamma_1}\cr &{\bf
J}_{\Gamma_3}\cr & &\ddots\cr &&& {\bf J}_{\Gamma_{2n+1}}\cr
&&&\ddots\cr &&&&\end{bmatrix},
\end{equation}
\begin{equation}
\label{WT0prod} \wt\cT_0=\wt\cT_0(\{\Gamma_n\}_{n\ge
0})=\begin{bmatrix}{\bf J}_{\Gamma_1}\cr &{\bf J}_{\Gamma_3}\cr &
&\ddots\cr &&& {\bf J}_{\Gamma_{2n+1}}\cr &&&\ddots\cr
&&&&\end{bmatrix}\begin{bmatrix}-\Gamma^*_0\cr &{\bf
J}_{\Gamma_2}\cr & &{\bf J}_{\Gamma_4}\cr & &&\ddots\cr &&&& {\bf
J}_{\Gamma_{2n}}\cr &&&&&\ddots
\end{bmatrix}.
\end{equation}
In particular, it follows that
\begin{equation}
\label{ADJ}
 \left(\cT_0(\{\Gamma_n\}_{n\ge
0})\right)^*=\wt\cT_0(\{\Gamma^*_n\}_{n\ge 0}).
\end{equation}
From \eqref{T0prod} and \eqref{WT0prod} we have
\[
\cV_0\cT_0=\wt\cT_0\cV_0,
\]
where the unitary operator $\cV_0$ is defined by \eqref{V0}.
Therefore, the operators $\cT_0$ and $\wt\cT_0$ are unitarily
equivalent.
\begin{proposition}
\label{COIUN} Let $\Theta\in{\bf S}(\sM,\sN)$ and let
$\{\Gamma_n\}_{n\ge 0}$ be the Schur parameters of $\Theta$. Suppose
$\Gamma_n$ is neither isometric nor co-isometric for each $n$. Let
the function $\Omega\in{\bf S}(\sK,\sL)$ coincides with $\Theta$ and
Let $\{G_n\}_{n\ge 0}$ be the Schur parameters of $\Omega$. Then
truncated block operator CMV matrices $\cT_0(\{\Gamma_n\}_{n\ge 0})$
and $\cT_0(\{G_n\}_{n\ge 0})$ (respect.,
$\wt\cT_0(\{\Gamma_n\}_{n\ge 0})$ and $\wt\cT_0(\{G_n\}_{n\ge 0})$)
are unitarily equivalent.
\end{proposition}
\begin{proof}Since
$\Omega(\lambda)=V\Theta(\lambda)U,$ where $U\in\bL(\sK,\sM)$ and
$V\in\bL(\sN,\sL)$ are unitary operators, we get relations
\eqref{COINC1}. It follows that $\sD_{G_n}\ne\{0\}$ and
$\sD_{G^*_n}\ne \{0\}$ for all $n$. Hence, we have
\begin{equation}
\label{connect}{\bf J}_{G_n}\begin{bmatrix}U^*& 0\cr 0& V
\end{bmatrix}=\begin{bmatrix}V& 0\cr 0& U^*
\end{bmatrix}{\bf J}_{\Gamma_n},\; n=0,1,\ldots.
\end{equation}
Define the Hilbert space
\[
\sH^{\Omega}_0=\sH_0(\{G_n\}_{n\ge 0}):=\sum\limits_{n\ge
0}\bigoplus\begin{array}{l}\sD_{G_{2n}}\\\oplus\\\sD_{G^*_{2n+1}}\end{array}
\]
and truncated block operator CMV matrix
\[
\cT_0(\{G_n\}_{n\ge 0}):=\begin{bmatrix}-G^*_0\cr &{\bf J}_{G_2}\cr
& &{\bf J}_{G_4}\cr & &&\ddots\cr &&&& {\bf J}_{G_{2n}}\cr
&&&&&\ddots
\end{bmatrix}\begin{bmatrix}{\bf J}_{G_1}\cr &{\bf
J}_{G_3}\cr & &\ddots\cr &&& {\bf J}_{G_{2n+1}}\cr &&&\ddots\cr
&&&&\end{bmatrix},
\]
Define the unitary operator
\[
\cW=\begin{bmatrix}U^*\cr &V\cr & &U^*\cr & && V\cr &&&& \ddots
\end{bmatrix}:\sH_0\to \sH^{\Omega}_0
\]
From \eqref{T0prod} and \eqref{connect} we obtain
\[
\cW\cT_0(\{\Gamma_n\}_{n\ge 0})=\cT(\{G_n\}_{n\ge 0})\cW.
\]
Thus $\cT_0(\{\Gamma_n\}_{n\ge 0})$ and $\cT(\{G_n\}_{n\ge 0})$ are
unitarily equivalent.
\end{proof}
 Now we are going to find the defect operators and
defect subspaces for $\cT_0$ and $\wt\cT_0.$ Let ${\bf f}=(\vec
f_0,\vec f_1,\ldots)^T \in\sH_0$, where
\[
\vec f_n=\begin{bmatrix}h_n\cr
g_n\end{bmatrix}\in\begin{array}{l}\sD_{\Gamma_{2n}}\\\oplus\\\sD_{\Gamma^*_{2n+1}}\end{array},\;
n=0,1,\ldots.
\]
Then
\begin{equation}
\label{DEFEKTI}
\begin{array}{l}
||{\bf f}||^2-||\cT_0{\bf f}||^2=||P_\sN\cU_0{\bf
f|}|^2=\left\|\cC_0\begin{bmatrix}h_0\cr
g_0\end{bmatrix}\right\|^2=||D_{\Gamma^*_0}(\Gamma_1h_0+D_{\Gamma^*_1}g_0)||^2,\\
||{\bf f}||^2-||\cT^*_0{\bf f}||^2=||P_\sM\cU_0^*{\bf
f|}|^2=\left\|\cA^*_0\begin{bmatrix}h_0\cr
g_0\end{bmatrix}\right\|^2=||D_{\Gamma_0}h_0||^2.
\end{array}
\end{equation}
 Let ${\bf x}=(x_0,x_1,\ldots)^T\in\wt\sH_0,$ where
\[
x_n=\begin{bmatrix}h_n\cr
g_n\end{bmatrix}\in\begin{array}{l}\sD_{\Gamma^*_{2n}}\\\oplus\\\sD_{\Gamma^*_{2n+1}}\end{array},\;
n=0,1,\ldots.
\]
Then
\[
\begin{array}{l}
||{\bf x}||^2-||\wt\cT_0{\bf x}||^2=||P_\sN\wt\cU_0{\bf
x|}|^2=\left\|\wt\cC_0\begin{bmatrix}h_0\cr
g_0\end{bmatrix}\right\|^2=||D_{\Gamma^*_0}h_0||^2,\\
||{\bf x}||^2-||\wt\cT^*_0{\bf x}||^2=||P_\sM\wt\cU_0^*{\bf
x|}|^2=\left\|\wt\cA^*_0\begin{bmatrix}h_0\cr
g_0\end{bmatrix}\right\|^2=||D_{\Gamma_0}(\Gamma^*_1h_0+D_{\Gamma_1}
g_0)||^2.
\end{array}
\]
Now from Proposition \ref{ranges} it follows that
\begin{equation}
\label{kernels} \left\{ \begin{array}{l}
 \ker
D_{\cT_0}=\left\{\begin{bmatrix}D_{\Gamma_1}\f\cr-\Gamma_1\f
\end{bmatrix},\;\f\in \sD_{\Gamma_1}\right\}\bigoplus\sum\limits_{n\ge
1}\bigoplus\begin{array}{l}\sD_{\Gamma_{2n}}\\\oplus\\\sD_{\Gamma^*_{2n+1}}\end{array},\\
\ker D_{\cT^*_0}=\sD_{\Gamma^*_1}\bigoplus\sum\limits_{n\ge
1}\bigoplus\begin{array}{l}\sD_{\Gamma_{2n}}\\\oplus\\\sD_{\Gamma^*_{2n+1}}\end{array},\\
\sD_{\cT_0}=\left\{\begin{bmatrix}\Gamma^*_1\psi\cr
D_{\Gamma^*_1}\psi
\end{bmatrix},\;\psi\in \sD_{\Gamma^*_0}\right\}\bigoplus \vec 0,\;
\vec 0\in \sum\limits_{n\ge
1}\bigoplus\begin{array}{l}\sD_{\Gamma_{2n}}\\\oplus\\\sD_{\Gamma^*_{2n+1}}\end{array},\\
\sD_{\cT^*_0}=\sD_{\Gamma_0}\bigoplus \vec 0,\; \vec 0\in
\sD_{\Gamma^*_1}\bigoplus \sum\limits_{n\ge
1}\bigoplus\begin{array}{l}\sD_{\Gamma_{2n}}\\\oplus\\\sD_{\Gamma^*_{2n+1}}\end{array}.
\end{array}\right.
\end{equation}
\begin{equation}
\label{kernels1} \left\{ \begin{array}{l}
 \ker
D_{\wt\cT_0}=\sD_{\Gamma_1}\bigoplus\sum\limits_{n\ge
1}\bigoplus\begin{array}{l}\sD_{\Gamma^*_{2n}}\\\oplus\\\sD_{\Gamma_{2n+1}}\end{array},\\
\ker
D_{\wt\cT^*_0}=\left\{\begin{bmatrix}D_{\Gamma^*_1}\f\cr-\Gamma^*_1\f
\end{bmatrix},\;\f\in \sD_{\Gamma^*_0}\right\}\bigoplus\sum\limits_{n\ge
1}\bigoplus\begin{array}{l}\sD_{\Gamma^*_{2n}}\\\oplus\\\sD_{\Gamma_{2n+1}}\end{array},\\
\sD_{\wt\cT_0}=\sD_{\Gamma^*_0}\bigoplus \vec 0,\; \vec 0\in
\sD_{\Gamma_1}\bigoplus \sum\limits_{n\ge
1}\bigoplus\begin{array}{l}\sD_{\Gamma^*_{2n}}\\\oplus\\\sD_{\Gamma_{2n+1}}\end{array},\\
\sD_{\wt\cT^*_0}=\left\{\begin{bmatrix}\Gamma_1\psi\cr
D_{\Gamma_1}\psi
\end{bmatrix},\;\psi\in \sD_{\Gamma_0}\right\}\bigoplus \vec 0,\;
\vec 0\in \sum\limits_{n\ge
1}\bigoplus\begin{array}{l}\sD_{\Gamma^*_{2n}}\\\oplus\\\sD_{\Gamma_{2n+1}}\end{array}.
\end{array}
\right.
\end{equation}
\subsection{Simple conservative realizations of the Schur class
function by means of its Schur parameters}\label{RRR}
 Let
\[
\begin{array}{l}
\cG_0=\cG_0(\{\Gamma_n\}_{n\ge}
0)=\begin{bmatrix}D_{\Gamma^*_0}\Gamma_1&D_{\Gamma^*_0}D_{\Gamma^*_1}&0&0&\ldots\end{bmatrix}:
\sH_0 \to \sN,\\
\wt\cG_0=\wt\cG_0(\{\Gamma_n\}_{n\ge
0})=\begin{bmatrix}D_{\Gamma^*_0}&0&0&\ldots\end{bmatrix}:\wt\sH_0\to\sN,\\
\cF_0=\cF_0(\{\Gamma_n\}_{n\ge 0})=\begin{bmatrix} D_{\Gamma_0}\cr
0\cr0\cr\vdots
\end{bmatrix}:\sM\to\sH_0,\;
\wt\cF_0=\wt\cF_0(\{\Gamma_n\}_{n\ge 0})=\begin{bmatrix}\Gamma_1
D_{\Gamma_0}\cr D_{\Gamma_1}D_{\Gamma_0}\cr0\cr
0\cr\vdots\end{bmatrix}:\sM\to\wt\sH_0.
\end{array}
\]
The operators $\cU_0$ and $\wt\cU_0$ can be represented by $2\times
2$ block operator matrices
\[
\begin{array}{l}
\cU_0=\begin{bmatrix} \Gamma_0&\cG_0\cr \cF_0
&\cT_0\end{bmatrix}:\begin{array}{l}\sM\\\oplus\\\sH_0\end{array}\to
\begin{array}{l}\sN\\\oplus\\\sH_0\end{array},\\
\wt\cU_0=\begin{bmatrix} \Gamma_0&\wt\cG_0\cr\wt\cF_0
&\wt\cT_0\end{bmatrix}:\begin{array}{l}\sM\\\oplus\\\wt
\sH_0\end{array}\to
\begin{array}{l}\sN\\\oplus\\\wt\sH_0\end{array}.
\end{array}
\]
Define the following conservative systems
\begin{equation}
\label{CMVMODEL} \begin{array}{l} \zeta_0=\left\{\begin{bmatrix}
\Gamma_0&\cG_0\cr \cF_0
&\cT_0\end{bmatrix};\sM,\sN,\sH_0\right\}=\left\{\cU_0(\{\Gamma_n\}_{n\ge 0});\sM,\sN,\sH_0(\{\Gamma_n\}_{n\ge 0})\right\},\\
\wt\zeta_0=\left\{\begin{bmatrix} \Gamma_0&\wt\cG_0\cr \wt\cF_0
&\wt\cT_0\end{bmatrix};\sM,\sN,\wt\sH_0\right\}=\left\{\wt\cU_0(\{\Gamma_n\}_{n\ge
0});\sM,\sN,\wt\sH_0(\{\Gamma_n\}_{n\ge 0})\right\}.
\end{array}
\end{equation}
The equalities \eqref{MOV} and \eqref{MU} yield that systems
$\zeta_0$ and $\wt\zeta_0$ are unitarily equivalent. Hence,
$\zeta_0$ and $\wt\zeta_0$ have equal transfer functions.

Observe that
\[
\cF_0=\begin{bmatrix}I_\sM\cr 0\cr
0\cr\vdots\end{bmatrix}D_{\Gamma_0},\;
\cG_0=D_{\Gamma^*_0}\begin{bmatrix}
\Gamma_1&D_{\Gamma^*_1}&0&0&\ldots\end{bmatrix},
\]
\[
\wt\cF_0=\begin{bmatrix}\Gamma_1\cr D_{\Gamma_1}\cr 0\cr
0\cr\vdots\end{bmatrix}D_{\Gamma_0},\;\wt\cG_0=D_{\Gamma^*_0},
\begin{bmatrix} I_\sN&0&0&0&\ldots\end{bmatrix}
\]
and
\[
\begin{bmatrix}
\Gamma_1&D_{\Gamma^*_1}&0&0&\ldots\end{bmatrix}\begin{bmatrix}I_\sM\cr
0\cr 0\cr\vdots\end{bmatrix}=\begin{bmatrix}
I_\sN&0&0&0&\ldots\end{bmatrix}\begin{bmatrix}\Gamma_1\cr
D_{\Gamma_1}\cr 0\cr 0\cr\vdots\end{bmatrix}=\Gamma_1.
\]
\begin{theorem}
\label{CR} The unitarily equivalent conservative systems $\zeta_0$
and $\wt\zeta_0$ given by \eqref{CMVMODEL} are simple and the Schur
parameters of the transfer function of $\zeta_0$ and $\wt\zeta_0$
are $\{\Gamma_n\}_{n\ge 0}.$
\end{theorem}

\begin{proof}
 The main step is a proof that the systems $\Omega_{0,1}(\zeta_0)$ and
$\Omega_{1,0}(\wt\zeta_0)$ given by \eqref{RealFirst} take the form
\begin{equation}
\label{MainStep} \begin{array}{l}
 \Omega_{0,1}(\zeta_0)=\left\{\wt
U_0\left(\{\Gamma_n\}_{n\ge
1}\right),\sD_{\Gamma_0},\sD_{\Gamma^*_0},\wt\sH_0\left(\{\Gamma_n\}_{n\ge
1}\right)\right\},\\
\Omega_{1,0}(\wt\zeta_0)=\left\{U_0\left(\{\Gamma_n\}_{n\ge
1}\right),\sD_{\Gamma_0},\sD_{\Gamma^*_0},\sH_0\left(\{\Gamma_n\}_{n\ge
1}\right)\right\}.
\end{array}
\end{equation}
First of all we will prove that the systems $\zeta_0$ and
$\wt\zeta_0$ are simple.

Define the subspaces
\[
\sH_{2k-1}=\sum\limits_{n\ge
k}\bigoplus\begin{array}{l}\sD_{\Gamma^*_{2n-1}}\\\oplus\\\sD_{\Gamma_{2n}}\end{array},\;
\sH_{2k}=\sum\limits_{n\ge
k}\bigoplus\begin{array}{l}\sD_{\Gamma_{2n}}\\\oplus\\\sD_{\Gamma^*_{2n+1}}\end{array},\;k=1,2,\ldots
\]
Clearly,
$\sH_0\supset\sH_1\supset\sH_2\supset\cdots\supset\sH_m\supset\cdots.$
From \eqref{statespaces} it follows the equality
\[
 \bigcap\limits_{m\ge 0}\sH_m=\{0\}.
\]
 Let $\Gamma_{-1}=0:\sM\to\sN$. Then
$\sD_{\Gamma_{-1}}=\sM,$ $\sD_{\Gamma^*_{-1}}=\sN$. We can consider
$\cU_0$ as acting from $\sD_{\Gamma_{-1}}\oplus\sH_0$ onto
$\sD_{\Gamma^*_{-1}}\oplus\sH_0$ and $\wt\cU_0$ as acting from
$\sD_{\Gamma_{-1}}\oplus\wt\sH_0$ onto
$\sD_{\Gamma^*_{-1}}\oplus\wt\sH_0$. Fix $m\in\dN$ and define
\[
\Gamma^{(m)}_n=\Gamma_{n+m},\;n=-1,0,1,\ldots
\]
Then $\{\Gamma^{(m)}_n\}_{n\ge 0}=\{\Gamma_k\}_{k\ge m},$ and
\[
\begin{array}{l}
 \sH_{2k-1}=\wt\sH_0(\{\Gamma^{(2k-1)}_n\}_{n\ge
0})=\wt\sH_0\left(\{\Gamma_n\}_{n\ge 2k-1}\right),\\
\sH_{2k}=\sH_0(\{\Gamma^{(2k)}_n\}_{n\ge
0})=\sH_0\left(\{\Gamma_n\}_{n\ge 2k}\right).
\end{array}
\]
 Let
\[
\begin{array}{l}
\cW_{2k-1}=\wt \cU_0(\{\Gamma^{(2k-1)}_n\}_{n\ge
0})=\wt\cU_0(\{\Gamma_n\}_{n\ge
2k-1}):\begin{array}{l}\sD_{\Gamma_{2k-2}}\\\oplus\\\sH_{2k-1}\end{array}\to
\begin{array}{l}\sD_{\Gamma^*_{2k-2}}\\\oplus\\\sH_{2k-1}\end{array},\\
 \cW_{2k}=\cU_0(\{\Gamma^{(2k)}_n\}_{n\ge
0})=\cU_0(\{\Gamma_n\}_{n\ge
2k}):\begin{array}{l}\sD_{\Gamma_{2k-1}}\\\oplus\\\sH_{2k}\end{array}\to
\begin{array}{l}\sD_{\Gamma^*_{2k-1}}\\\oplus\\\sH_{2k}\end{array},\;k\ge
1.
 \end{array}
\]
Define the operators
\begin{equation}
\label{TR} \cT_{m}=P_{\sH_{m}}\cW_{m}\uphar \sH_{m},\;m=1,2,\ldots.
\end{equation}
Then
\begin{equation}
\label{TTTUNC} \begin{array}{l}
\cT_{2k-1}=\wt\cT_0(\{\Gamma^{(2k-1)}_n\}_{n\ge
0})=\wt\cT_0(\{\Gamma_n\}_{n\ge 2k-1}),\\
\cT_{2k}=\cT_0(\{\Gamma^{(2k)}_n\}_{n\ge
0})=\cT_0(\{\Gamma_n\}_{n\ge 2k}).
\end{array}
\end{equation}
From \eqref{kernels}, \eqref{kernels1}, \eqref{TR}, and
\eqref{TTTUNC} we get
\[
\ker D_{\cT^*_0}=\sH_1,\;\ker D_{\cT_1}=\sH_2,\dots,\ker
D_{\cT^*_{2k}}=\sH_{2k+1},\; \ker D_{\cT_{2k-1}}=\sH_{2k},\ldots.
\]
From \eqref{T0prod}, \eqref{WT0prod}, and \eqref{TTTUNC} it follows
that
\[
\begin{array}{l}
P_{\ker D_{\cT^*_0}}\cT_0\uphar\ker D_{\cT^*_0}=\cT_1, \;P_{\ker
D_{\cT_1}}\cT_1\uphar\ker D_{\cT_1}=\cT_2,\ldots,\\
P_{\ker D_{\cT_{2k-1}}}\cT_{2k-1}\uphar\ker
D_{\cT_{2k-1}}=\cT_{2k},\;P_{\ker D_{\cT^*_{2k}}}\cT_{2k}\uphar\ker
D_{\cT^*_{2k}}=\cT_{2k+1},\ldots.
\end{array}
\]
Thus,
\[
\begin{array}{l}
\sH_{2k-1}=\ker D_{\cT^{*k}_0}\cap\ker D_{\cT^{k-1}_0},\\
\sH_{2k}=\ker D_{\cT^{*k}_0}\cap\ker D_{\cT^{k}_0}.
\end{array}
\]
In notations of Section \ref{gener} the operators $\cT_{2k-1}$ and
$\cT_{2k}$ coincide with the operators $(\cT_0)_{k-1,k}$ and
$(\cT_0)_{k,k},$ respectively.
 From the definition
of $\sH_0$ we get
\[
\left(\bigcap\limits_{k\ge 1}\ker
D_{\cT^{*k}_0}\right)\bigcap\left(\bigcap\limits_{k\ge 1}\ker
D_{\cT^{k}_0}\right)= \bigcap\limits_{k\ge 1}\left( \ker
D_{\cT^{*k}_0}\cap\ker D_{\cT^{k}_0} \right)=\bigcap\limits_{k\ge
1}\sH_{2k}=\{0\}.
\]
So, the operators $\cT_0$, $\wt\cT_0$, and $\{\cT_{k}\}_{k\ge 1}$
are completely non-unitary. It follows that the conservative systems
\[
\zeta_0=\left\{\begin{bmatrix} \Gamma_0&\cG_0\cr \cF_0
&\cT_0\end{bmatrix};\sM,\sN,\sH\right\}\;\mbox{and}
\;\wt\zeta_0=\left\{\begin{bmatrix} \Gamma_0&\wt\cG_0\cr \wt\cF_0
&\wt\cT_0\end{bmatrix};\sM,\sN,\wt\sH_0\right\}
\]
are simple.

 The operators $\cW_m$ takes the following $2\times 2$
block operator matrix form
\[
\cW_{m}=\begin{bmatrix} \Gamma_{m}&\cG_{m}\cr \cF_{m}
&\cT_{m}\end{bmatrix}:\begin{array}{l}\sD_{\Gamma_{m-1}}\\\oplus\\\sH_{m}\end{array}\to
\begin{array}{l}\sD_{\Gamma^*_{m-1}}\\\oplus\\\sH_m\end{array},
\]
where
\[
\begin{array}{l}
\cG_{2k-1}=\wt\cG_0(\{\Gamma_n\}_{n\ge
2k-1})=\begin{bmatrix}D_{\Gamma^*_{2k-1}}&0&0&\ldots\end{bmatrix}:\sH_{2k-1}\to\sD_{\Gamma^*_{2k-2}},\\
\cG_{2k}=\cG_0(\{\Gamma_n\}_{n\ge
2k})=\begin{bmatrix}D_{\Gamma^*_{2k}}\Gamma_{2k+1}&D_{\Gamma^*_{2k}}D_{\Gamma^*_{2k+1}}&0&0&\ldots\end{bmatrix}:
\sH_{2k} \to \sD_{\Gamma^*_{2k-1}},\\
\cF_{2k-1}=\wt\cF_0(\{\Gamma_n\}_{n\ge
2k-1})=\begin{bmatrix}\Gamma_{2k} D_{\Gamma_{2k-1}}\cr
D_{\Gamma_2k}D_{\Gamma_{2k-1}}\cr0\cr
0\cr\vdots\end{bmatrix}:\sD_{\Gamma_{2k-2}}\to\sH_{2k-1},\\
\cF_{2k}=\cF_0(\{\Gamma_n\}_{n\ge 2k})=\begin{bmatrix}
D_{\Gamma_{2k}}\cr 0\cr0\cr\vdots
\end{bmatrix}:\sD_{\Gamma_{2k-1}}\to\sH_{2k}.
\end{array}
\]
Suppose that the system
\[
\zeta_0=\left\{\begin{bmatrix} \Gamma_0&\cG_0\cr \cF_0
&\cT_0\end{bmatrix};\sM,\sN,\sH\right\}
\]
has transfer function $\Psi(\lambda)$, i.e.,
\[
\Psi(\lambda)=\Gamma_0+\lambda\cG_0(I_{\sH_0}-\lambda\cT_0)^{-1}\cF_0.
\]
Then $\Psi(0)=\Gamma_0$. Let $\Psi_1(\lambda)$ be the first Schur
iterate of $\Psi.$ By \eqref{RealFirst} the transfer function of the
simple conservative system
\[
\Omega_{0,1}(\nu)=\left\{\begin{bmatrix}D^{-1}_{\Gamma^*_0}C(D^{-1}_{\Gamma_0}B^*)^*&D^{-1}_{\Gamma^*_0}C\uphar\ker
D_{A^*}\cr AP_{\ker D_A}D^{-1}_{A^*}B&P_{\ker D_{A^*}}A\uphar\ker
D_{A^*}\end{bmatrix};\sD_{\Gamma_0},\sD_{\Gamma^*_0},\ker
D_{A^*}\right\}
\]
is  the first Schur iterate of the transfer function of the simple
conservative system
\[
\nu=\left\{\begin{bmatrix}\Gamma_0&C\cr
B&A\end{bmatrix};\sM,\sN,\sH\right\}.
\]
We will construct the system $\zeta_1=\Omega_{0,1}(\zeta_0)$ from
the system $\zeta_0$. In our case
\[
\zeta_1=\Omega_{0,1}(\zeta_0)=\left\{\begin{bmatrix}D^{-1}_{\Gamma^*_0}\cG_0(D^{-1}_{\Gamma_0}\cF^*_0)^*&D^{-1}_{\Gamma^*_0}\cG_0\uphar\ker
D_{\cT^*_0}\cr \cT_0 P_{\ker D_{\cT_0}}D^{-1}_{\cT^*_0}\cF_0&P_{\ker
D_{\cT^*_0}}\cT_0\uphar\ker
D_{\cT^*_0}\end{bmatrix};\sD_{\Gamma_0},\sD_{\Gamma^*_0},\ker
D_{\cT^*_0}\right\}
\]
 Clearly,
\[
\begin{array}{l}
D^{-1}_{\Gamma^*_0}\cG_0=\begin{bmatrix}
\Gamma_1&D_{\Gamma^*_1}&0&0&\ldots\end{bmatrix}:\sH_0\to\sD_{\Gamma^*_0},\\
(D^{-1}_{\Gamma_0}\cF^*_0)^*= \begin{bmatrix}I_\sM\cr 0\cr
0\cr\vdots\end{bmatrix}:\sD_{\Gamma_0}\to\sH_0.
\end{array}
\]
Therefore,
\[
\begin{bmatrix}
\Gamma_1&D_{\Gamma^*_1}&0&0&\ldots\end{bmatrix}\begin{bmatrix}I_\sM\cr
0\cr 0\cr\vdots\end{bmatrix}=\Gamma_1.
\]
Thus, the first Schur parameter of $\Psi$ is equal to $\Gamma_1$.
From \eqref{kernels} it follows that $\ker D_{\cT^*_0}=\sH_1$ and
$\sD_{\cT^*_0}=D_{\Gamma_0}P_{\sD_{\Gamma_0}}$. Hence
\[
\sD^{-1}_{\cT^*_0}\cF_0=\begin{bmatrix} I_{\sD_{\Gamma_0}}\cr
0\cr0\cr\vdots
\end{bmatrix}:\sD_{\Gamma_0}\to\sH_0.
\]
As has been proved above
\[
P_{\ker D_{\cT^*_0}}\cT_0\uphar\ker D_{\cT^*_0}=\cT_1.
\]
Let $h\in\sD_{\Gamma_0}$. Let us find the projection $P_{\ker
D_{\cT_0}}h$.  According to \eqref{kernels} we have to find the
vectors $\f\in\sD_{\Gamma_1}$ and $\psi\in\sD_{\Gamma^*_0}$ such
that
\[
\begin{bmatrix} h\cr 0\end{bmatrix}=\begin{bmatrix}
D_{\Gamma_1}\f\cr-\Gamma_1\f\end{bmatrix}+\begin{bmatrix}
\Gamma^*_1\psi\cr D_{\Gamma^*_1}\psi\end{bmatrix}.
\]
We have
\[
\left\{
\begin{array}{l}
h=D_{\Gamma_1}\f+\Gamma^*_1\psi\\
\Gamma_1\f=D_{\Gamma^*_1}\psi.
\end{array}
\right.
\]
From the second equation and Proposition \ref{ranges} it follows
$\f=D_{\Gamma_1}g,$ $\psi=\Gamma_1 g$, where $g\in\sD_{\Gamma_0}$.
Therefore
\[
h=D^2_{\Gamma_1}g+\Gamma^*_1\Gamma_1 g,
\]
i.e., $g=h$. Hence
\[
P_{\ker D_{\cT_0}}\sD^{-1}_{\cT^*_0}\cF_0h= P_{\ker
D_{\cT_0}}h=\begin{bmatrix} D^2_{\Gamma_1}h\cr-\Gamma_1
D_{\Gamma_1}h\cr 0\cr 0\cr \vdots\end{bmatrix}.
\]
Now we get
\[
\begin{array}{l}
\cT_0 P_{\ker D_{\cT_0}}\sD^{-1}_{\cT^*_0}\cF_0h=\\
\quad=\begin{bmatrix}
D_{\Gamma_0}&-\Gamma^*_0\Gamma_1&-\Gamma^*_0D_{\Gamma^*_1}&0&0&0&0&0&\ldots&\ldots\cr
0&\Gamma_2D_{\Gamma_1}&-\Gamma_2\Gamma^*_1&D_{\Gamma^*_2}\Gamma_3&D_{\Gamma^*_2}D_{\Gamma^*_3}&0&0&0&\ldots&\ldots\cr
0&D_{\Gamma_2}D_{\Gamma_1}&-D_{\Gamma_2}\Gamma^*_1&-\Gamma^*_2\Gamma_3&-\Gamma^*_2D_{\Gamma^*_3}&0&0&0&\ldots&\ldots\cr
0&0&0&\Gamma_4D_{\Gamma_3}&-\Gamma_4\Gamma^*_3&D_{\Gamma^*_4}\Gamma_5&D_{\Gamma^*_4}D_{\Gamma^*_5}&0&\dots&\ldots\cr
0&0&0&D_{\Gamma_4}D_{\Gamma_3}&-D_{\Gamma_4}\Gamma^*_3&-\Gamma^*_4\Gamma_5&-\Gamma_4D_{\Gamma^*_5}&0&\ldots&\ldots\cr
0&0&0&0&0&\Gamma_6D_{\Gamma_5}&-\Gamma_6\Gamma^*_5&D_{\Gamma^*_6}\Gamma_7&D_{\Gamma^*_6}D_{\Gamma^*_7}&\ldots\cr
\vdots&\vdots&\vdots&\vdots&\vdots&\vdots&\vdots&\vdots&\vdots&\vdots
\end{bmatrix}\times\\
\times
\begin{bmatrix} D^2_{\Gamma_1}h\cr-\Gamma_1
D_{\Gamma_1}h\cr 0\cr 0\cr \vdots\end{bmatrix}=
\begin{bmatrix}0\cr\Gamma_2D_{\Gamma_1}h\cr
D_{\Gamma_2} D_{\Gamma_1}h\cr 0\cr 0\cr\vdots\end{bmatrix}\in\sH_1.
\end{array}
\]
Thus we get that $\zeta_1$ is of the form
\[
\zeta_1=\Omega_{0,1}(\zeta_0)=\left\{\begin{bmatrix}
\Gamma_{1}&\cG_{1}\cr \cF_{1}
&\cT_{1}\end{bmatrix};\sD_{\Gamma_{0}},\sD_{\Gamma^*_{0}},\sH_{1}\right\}=\left\{\wt
U_0\left(\{\Gamma_n\}_{n\ge
1}\right),\sD_{\Gamma_0},\sD_{\Gamma^*_0},\wt\sH_0\left(\{\Gamma_n\}_{n\ge
1}\right)\right\}.
\]
Similarly
\[
\Omega_{1,0}(\wt\zeta_0)=\left\{U_0\left(\{\Gamma_n\}_{n\ge
1}\right),\sD_{\Gamma_0},\sD_{\Gamma^*_0},\sH_0\left(\{\Gamma_n\}_{n\ge
1}\right)\right\}.
\]
 The transfer functions of these systems are equal to
$\Psi_1(\lambda)$ (see Section \ref{gener}), and $\Gamma_1$ is
exactly is the first Schur parameter of $\Psi(\lambda)$.

Let $\Psi_2(\lambda)$ is the second Schur iterate of $\Psi$.
Constructing the simple conservative system
$\zeta_2=\Omega_{1,0}(\zeta_1)$ of the form \eqref{RealFirst} with
the transfer function $\Psi_2$ we will get the system
\[
\zeta_2=\left\{\begin{bmatrix} \Gamma_{2}&\cG_{2}\cr \cF_{2}
&\cT_{2}\end{bmatrix};\sD_{\Gamma_{1}},\sD_{\Gamma^*_{1}},\sH_{2}\right\}=\left\{\cU_0(\{\Gamma_n\}_{n\ge
2});\sD_{\Gamma_1},\sD_{\Gamma^*_{1}},\sH_{0}(\{\Gamma_n\}_{n\ge
2})\right\}.
\]
Let $\Psi_m(\lambda)$ be the $m-$th Schur iterate of $\Psi$. Arguing
by induction we get that $\Psi_m(\lambda)$ is transfer function of
the system
\[
\begin{array}{l}
\zeta_m=\left\{\begin{bmatrix} \Gamma_{m}&\cG_{m}\cr \cF_{m}
&\cT_{m}\end{bmatrix};\sD_{\Gamma_{m-1}},\sD_{\Gamma^*_{m-1}},\sH_{m}\right\}=\\
\quad\quad=\left\{\begin{array}{l}
\left\{\wt\cU_0(\{\Gamma_n\}_{n\ge
2k-1});\sD_{\Gamma_{2k-2}},\sD_{\Gamma^*_{2k-2}},\wt\sH_0(\{\Gamma_n\}_{n\ge
2k-1})\right\},\; m=2k-1\\
\left\{\cU_0(\{\Gamma_n\}_{n\ge
2k});\sD_{\Gamma_{2k-1}},\sD_{\Gamma^*_{2k-1}},\sH_0(\{\Gamma_n\}_{n\ge
2k})\right\},\; m=2k
\end{array}
\right.
\end{array}
\]
for all $m$. Observe that
\[
\zeta_{2k-1}=\Omega_{0,1}(\zeta_{2k-2}),\;
\zeta_{2k}=\Omega_{1,0}(\zeta_{2k-1}),\; k\ge 1.
\]
 Thus, $\{\Gamma_n\}_{n\ge 0}$ are the Schur parameters
of $\Psi$.
\end{proof}
From Theorem \ref{CR} and Theorem \ref{UNIQ} we immediately arrive
at the following result.
\begin{theorem}
\label{cmvmod} Let $\Theta(\lambda)\in {\bf S}(\sM,\sN)$ and let
$\{\Gamma_n\}_{n\ge 0}$ be the Schur parameters of $\Theta$. Then
the systems \eqref{CMVMODEL} are simple conservative realizations of
$\Theta$. Moreover, for each natural number $k$ the $k$-th Schur
iterate $\Theta_k$ of $\Theta$ is the transfer function of the
simple conservative systems $\left\{\cU_0(\Gamma_n\}_{n\ge
k});\sD_{\Gamma_{k-1}},\sD_{\Gamma^*_{k-1}},\sH_0(\Gamma_n\}_{n\ge
k})\right\}$ and $\left\{\wt\cU_0(\Gamma_n\}_{n\ge
k});\sD_{\Gamma_{k-1}},\sD_{\Gamma^*_{k-1}},\wt\sH_0(\Gamma_n\}_{n\ge
k})\right\}$.
\end{theorem}
Observe that in fact we have proved Theorem \ref{SchurAlg} and our
proof is different from given in \cite{Const} and \cite{BC}.
\begin{remark}
More complicated construction of the state Hilbert space and simple
conservative realization for a Schur function $\Theta\in{\bf
S}(\sM,\sN)$ by means of a block operator matrix are given in
\cite{KailBruck} (see \cite{BC}). These constructions also involve
Schur parameters of $\Theta$ and some additional Hilbert spaces and
operators. One more model based on the Schur parameters of a scalar
Schur class function $\Theta$ is obtained in \cite{D1}. In terms of
this model in \cite{D1} are established the necessary and sufficient
conditions in order to $\Theta$ has a meromorphic pseudocontinuation
of bounded type to the exterior of the unit disk. In recent preprint
\cite{FKatsKr} a construction of a minimal conservative realization
of a scalar finite Blaschke product in terms of the Hessenberg
matrix is given.
\end{remark}

\section{Block operator CMV matrices (the rest cases)}
\label{REST}
 Let $\{\Gamma_n\}$ be the Schur parameters of the function
$\Theta\in {\bf S}(\sM,\sN)$. Suppose  $\Gamma_m$ is an isometry
(respect., co-isometry, unitary) for some $m\ge 0$. Then
$\Theta_m(\lambda)=\Gamma_m$ for all $\lambda\in\dD$ and
\[
\begin{array}{l}
\Theta_{m-1}(\lambda)=\Gamma_{m-1}+\lambda
D_{\Gamma^*_{m-1}}\Gamma_m(I_{\sD_{\Gamma_{m-1}}}+\lambda\Gamma^*_{m-1}\Gamma_m)^{-1}D_{\Gamma_{m-1}},\\
\Theta_{m-2}(\lambda)=\Gamma_{m-2}+\lambda
D_{\Gamma^*_{m-2}}\Theta_{m-1}(\lambda)(I_{\sD_{\Gamma_{m-2}}}+\lambda\Gamma^*_{m-2}\Theta_{m-1}(\lambda))^{-1}D_{\Gamma_{m-2}},\\
\ldots \quad\ldots\quad \ldots\quad \ldots\quad \ldots \quad
\ldots\quad \ldots\quad \ldots\quad \ldots\quad \ldots\quad
\ldots\quad \ldots,\\
\Theta(\lambda)=\Gamma_{0}+\lambda
D_{\Gamma^*_{0}}\Theta_{1}(\lambda)(I_{\sD_{\Gamma_0}}+\lambda\Gamma^*_{0}\Theta_{1}(\lambda))^{-1}D_{\Gamma_{0}},\;
\lambda\in\dD.
\end{array}
\]
In this case the function $\Theta$ also is the transfer function of
the simple conservative systems constructed similarly to the
situation in Section \ref{BCMVR} by means of its Schur parameters
and corresponding block operator CMV matrices $\cU_0$ and
$\wt\cU_0$. Observe that
if $\Gamma_m$ is isometric (respect., co-isometric) then
$\Gamma_n=0,$ $\sD_{\Gamma^*_n}=\sD_{\Gamma^*_{m}}$,
$D_{\Gamma^*_{n}}=I_{\sD_{\Gamma^*_{m}}}$ (respect.,
$\sD_{\Gamma_n}=\sD_{\Gamma_{m}}$,
$D_{\Gamma_{n}}=I_{\sD_{\Gamma_{m}}}$) for $n> m$. The constructions
of the state spaces $\sH_0=\sH_0(\{\Gamma_n\}_{n\ge 0})$ and
$\wt\sH_0=\wt\sH_0(\{\Gamma_n\}_{n\ge 0})$ are similar to
\eqref{statespaces} but one have to replace $\sD_{\Gamma_n}$ by
$\{0\}$ (respect., $\sD_{\Gamma^*_n}$ by $\{0\}$) for $n\ge m$, and
$\sD_{\Gamma^*_n}$
 by $\sD_{\Gamma^*_m}$ (respect.,
$\sD_{\Gamma_n}$ by $\sD_{\Gamma_m}$) for $n> m$. The relation
\[
\wt\sH_0(\{\Gamma^*_n\}_{n\ge 0})=\sH_0(\{\Gamma_n\}_{n\ge 0})
\]
remains true.  If, in addition, the operator $\Gamma_m$ is isometry
($\iff$ the operator $D_{\Gamma^*_m}$ is the orthogonal projection
in $\sD_{\Gamma_{m-1}}$ onto $\ker\Gamma^*_m$) or co-isometry
($\iff$ the operator $D_{\Gamma_m}$ is the orthogonal projection in
$\sD_{\Gamma^*_{m-1}}$ onto $\ker\Gamma_m$), then the corresponding
unitary elementary rotation takes the the row or the column form
\[
\begin{array}{l}
{\bf
J}^{(r)}_{\Gamma_m}=\begin{bmatrix}\Gamma_m&I_{\sD_{\Gamma^*_m}}
\end{bmatrix}:\begin{array}{l}\sD_{\Gamma_{m-1}}\\\oplus\\\sD_{\Gamma^*_{m}}\end{array}\to\sD_{\Gamma^*_{m-1}},\\
{\bf J}^{(c)}_{\Gamma_m}=\begin{bmatrix}\Gamma_m\cr D_{\Gamma_m}
\end{bmatrix}:\sD_{\Gamma_{m-1}}\to
\begin{array}{l}\sD_{\Gamma^*_{m-1}}\\\oplus\\\sD_{\Gamma_{m}}\end{array}.
\end{array}
\]
Therefore, in definitions \eqref{OLM} of the block diagonal operator
matrices
\[
\cL_0=\cL_0(\{\Gamma_n\}_{n\ge 0}),\; \cM_0=\cM_0(\{\Gamma_n\}_{n\ge
0}),\;\mbox{and}\; \wt\cM_0=\wt\cM_0(\{\Gamma_n\}_{n\ge 0})
\]
one should replace
\begin{itemize}
\item ${\bf J}_{\Gamma_m}$ by ${\bf J}^{(r)}_{\Gamma_m}$
 and ${\bf J}_{\Gamma_n}$ by
$I_{\sD_{\Gamma^*_m}}$ for $n>m$, when $\Gamma_m$ is isometry,
\item ${\bf J}_{\Gamma_m}$ by ${\bf J}^{(c)}_{\Gamma_m}$,
 and ${\bf J}_{\Gamma_n}$ by
$I_{\sD_{\Gamma_m}}$ for $n>m$, when $\Gamma_m$ is co-isometry,
\item
${\bf J}_{\Gamma_m}$ by $\Gamma_m$, when $\Gamma_m$ is unitary.
\end{itemize}
As in Section \ref{BCMVR} in all these cases the block operators CMV
matrices $\cU_0=\cU_0(\{\Gamma_n\}_{n\ge 0})$ and
$\wt\cU_0=\wt\cU_0(\{\Gamma_n\}_{n\ge 0})$, are given by the
products
\[
\cU_0=\cL_0\cM_0,\;\wt\cU_0=\wt\cM_0\cL_0.
\]
These matrices are five block-diagonal. In the case when the
operator $\Gamma_m$, is unitary the block operator CMV matrices
$\cU_0$ and $\wt\cU_0$ are finite and otherwise they are
semi-infinite.

As in Section \ref{BCMVR} the truncated block operator CMV matrices
$\cT_0=\cT_0((\{\Gamma_n\}_{n\ge 0})$ and
$\wt\cT_0=\wt\cT_0(\{\Gamma_n\}_{n\ge 0})$ are defined by
\eqref{TRUNC} and \eqref{TRUNCT}
\[
\cT_0=P_{\sH_0}\cU_0\uphar\sH_0,\;\wt\cT_0=P_{\wt\sH_0}\wt\cU_0\uphar\wt\sH_0.
\]
As before the operators $\cT_0$ and $\wt\cT_0$ are unitarily
equivalent completely non-unitary contractions and, moreover, the
equalities \eqref{ADJ1}, \eqref{ADJ}, and Proposition \ref{COIUN}
hold true. Unlike  Section \ref{BCMVR} the operators given by
truncated block operator CMV matrices $\cT_m$ and $\wt\cT_m$
obtaining from $\cU_0$ and $\wt\cU_0$ by deleting first $m+1$ rows
and $m+1$ columns are
\begin{itemize}
\item
co-shifts of the form
\[
\cT_m=\wt\cT_m=\begin{bmatrix}0&I_{\sD_{\Gamma^*_m}}&0&0&\ldots\cr
0&0&I_{\sD_{\Gamma^*_m}}&0&\ldots\cr
0&0&0&I_{\sD_{\Gamma^*_m}}&\ldots\cr
\vdots&\vdots&\vdots&\vdots&\vdots
\end{bmatrix}:\begin{array}{l}\sD_{\Gamma^*_m}\\\oplus\\\sD_{\Gamma^*_m}\\\oplus\\\vdots
\end{array}\to \begin{array}{l}\sD_{\Gamma^*_m}\\\oplus\\\sD_{\Gamma^*_m}\\\oplus\\\vdots
\end{array},
\]
when $\Gamma_m$ is isometry,
\item
 the unilateral shifts of the form
\[
\cT_m=\wt\cT_m=\begin{bmatrix}0&0&0&0&\ldots\cr
I_{\sD_{\Gamma_m}}&0&0&0&\ldots\cr
0&I_{\sD_{\Gamma_m}}&0&0&\ldots\cr
\vdots&\vdots&\vdots&\vdots&\vdots
\end{bmatrix}:\begin{array}{l}\sD_{\Gamma_m}\\\oplus\\\sD_{\Gamma_m}\\\oplus\\\vdots
\end{array}\to \begin{array}{l}\sD_{\Gamma_m}\\\oplus\\\sD_{\Gamma_m}\\\oplus\\\vdots
\end{array},
\]
 when $\Gamma_m$ is co-isometry.
 \end{itemize}
One can see that Proposition \ref{COIUN} remains true.

Similarly to \eqref{CMVMODEL} let us consider the conservative
systems
\[
\zeta_0=\{\cU_0;\sM,\sN,\sH_0\},\;\wt\zeta_0=\{\wt\cU_0;\sM,\sN,\wt\sH_0\}.
\]
One can check that the systems $\zeta_0$ and $\wt\zeta_0$ are simple
and unitarily equivalent. Moreover, relations \eqref{MainStep} and,
therefore, Theorem \ref{CR} and Theorem \ref{cmvmod} remain valid
for a situations considered here.

In order to obtain precise forms of $\cU_0$ and $\wt\cU_0$ one can
consider the following cases:
\begin{enumerate}
\item $\Gamma_{2N}$ is isometric (co-isometric) for some $N$,
\item $\Gamma_{2N+1}$ is isometric (co-isometric) for some
$N$,
\item the operator $\Gamma_{2N}$ is unitary for some $N$,
\item the operator $\Gamma_{2N+1}$ is unitary for some $N$.
\end{enumerate}
We shall give several examples.
\begin{example}
\label{EX1}               {\bf The operator $\Gamma_{4}$ is
isometric.} Define the state spaces
\[
\begin{array}{l}
 \sH_0:=
\begin{array}{l}\sD_{\Gamma_{0}}\\\oplus\\\sD_{\Gamma^*_{1}}\end{array}\bigoplus
\begin{array}{l}\sD_{\Gamma_{2}}\\\oplus\\\sD_{\Gamma^*_{3}}\end{array}\bigoplus
\sD_{\Gamma^*_{4}}\bigoplus\sD_{\Gamma^*_{4}}\bigoplus\ldots\bigoplus\sD_{\Gamma^*_{4}}\bigoplus\ldots,\\
\wt\sH_0:=\begin{array}{l}\sD_{\Gamma^*_{0}}\\\oplus\\\sD_{\Gamma_{1}}\end{array}\bigoplus
\begin{array}{l}\sD_{\Gamma^*_{2}}\\\oplus\\\sD_{\Gamma_{3}}\end{array}\bigoplus
\sD_{\Gamma^*_{4}}\bigoplus\sD_{\Gamma^*_{4
}}\bigoplus\ldots\bigoplus\sD_{\Gamma^*_{4}}\bigoplus\ldots.
\end{array}
\]
Then the spaces $\sM\bigoplus\wt\sH_0$ and $\sN\bigoplus\sH_0$ can
be represented as follows
\[
\begin{array}{l}
 \sM\bigoplus\wt
\sH_0=\begin{array}{l}\sM\\\oplus\\\sD_{\Gamma^*_0}\end{array}\bigoplus
\begin{array}{l}\sD_{\Gamma_1}\\\oplus\\\sD_{\Gamma^*_2}\end{array}\bigoplus
\begin{array}{l}\sD_{\Gamma_3}\\\oplus\\\sD_{\Gamma^*_4}\end{array}
\bigoplus\sD_{\Gamma^*_4}\bigoplus\sD_{\Gamma^*_4}
\bigoplus\ldots,\\
\sN\bigoplus\sH_0=
\begin{array}{l}\sN\\\oplus\\\sD_{\Gamma_0}\end{array}\bigoplus
\begin{array}{l}\sD_{\Gamma^*_{1}}\\\oplus\\\sD_{\Gamma_2}\end{array}
 \bigoplus
 \sD_{\Gamma^*_{3}}\bigoplus\sD_{\Gamma^*_{4}}\bigoplus\ldots\bigoplus\sD_{\Gamma^*_{4}}\bigoplus\ldots.
\end{array}
 \]
 Define the unitary operators
\[
\cM_0= I_\sM\bigoplus{\bf J}_{\Gamma_{1}}\bigoplus {\bf
J}_{\Gamma_{3}}\bigoplus I_{\sD_{\Gamma^*_{4}}}\bigoplus
I_{\sD_{\Gamma^*_{4}}}\bigoplus\ldots :\sM\bigoplus\sH_0\to
\sM\bigoplus\wt\sH_0,
\]
\[
\wt\cM_0= I_\sN\bigoplus{\bf J}_{\Gamma_{1}}\bigoplus {\bf
J}_{\Gamma_{3}}\bigoplus I_{\sD_{\Gamma^*_{4}}}\bigoplus
I_{\sD_{\Gamma^*_{4}}}\bigoplus\ldots
:\sN\bigoplus\sH_0\to\sN\bigoplus\wt\sH_0,
\]
\[
\cL_0= {\bf J}_{\Gamma_0}\bigoplus{\bf J}_{\Gamma_2}\bigoplus{\bf
J}^{(r)}_{\Gamma_4}\bigoplus I_{\sD_{\Gamma^*_{4}}} \bigoplus
I_{\sD_{\Gamma^*_{4}}}\bigoplus
\ldots:\sM\bigoplus\wt\sH_0\to\sN\bigoplus\sH_0.
\]
Then
\[
\cU_0=\cL_0\cM_0=\begin{bmatrix}\Gamma_0&D_{\Gamma^*_0}\Gamma_1&
D_{\Gamma^*_0}D_{\Gamma^*_1}&0&0&0&0&0&0&\ldots\cr
D_{\Gamma_0}&-\Gamma^*_0\Gamma_1&-\Gamma^*_0D_{\Gamma^*_1}&0&0&0&0&0&0&\ldots\cr
0&\Gamma_2D_{\Gamma_1}&-\Gamma_2\Gamma^*_1&D_{\Gamma^*_2}\Gamma_3&D_{\Gamma^*_2}D_{\Gamma^*_3}&0&0&0&0&\ldots\cr
0&D_{\Gamma_2}D_{\Gamma_1}&-D_{\Gamma_2}\Gamma^*_1&-\Gamma^*_2\Gamma_3&-\Gamma^*_2D_{\Gamma^*_3}&0&0&0&0&\ldots\cr
0&0&0&\Gamma_4D_{\Gamma_3}&-\Gamma_4\Gamma^*_3&I_{\sD_{\Gamma^*_4}}&0&0&0&\ldots\cr
0&0&0&0&0&0&I_{\sD_{\Gamma^*_4}}&0&0&\ldots\cr
0&0&0&0&0&0&0&I_{\sD_{\Gamma^*_4}}&0&\ldots\cr
\vdots&\vdots&\vdots&\vdots&\vdots&\vdots&\vdots&\vdots&\vdots&\vdots
\end{bmatrix},
\]
\[
\wt\cU_0=\wt\cM_0\cL_0=\begin{bmatrix}\Gamma_0&D_{\Gamma^*_0}&
0&0&0&0&0&0&0&\ldots\cr
\Gamma_1D_{\Gamma_0}&-\Gamma_1\Gamma^*_0&D_{\Gamma^*_1}\Gamma_2&D_{\Gamma^*_1}D_{\Gamma^*_2}&0&0&0&0&0&\ldots\cr
D_{\Gamma_1}D_{\Gamma_0}&-D_{\Gamma_1}\Gamma^*_0&-\Gamma^*_1\Gamma_2&-\Gamma_1D_{\Gamma^*_2}
&0&0&0&0&0&\ldots\cr 0&0&\Gamma_3
D_{\Gamma_2}&-\Gamma_3\Gamma^*_2&D_{\Gamma^*_3}\Gamma_4&D_{\Gamma^*_3}&0&0&0&\ldots\cr
0&0&D_{\Gamma_3}D_{\Gamma_2}&-D_{\Gamma_3}\Gamma^*_2&-\Gamma^*_3\Gamma_4&-\Gamma^*_3
&0&0&0&\ldots\cr 0&0&0&0&0&0&I_{\sD_{\Gamma^*_4}}&0&0&\ldots\cr
0&0&0&0&0&0&0&I_{\sD_{\Gamma^*_4}}&0&\ldots\cr
\vdots&\vdots&\vdots&\vdots&\vdots&\vdots&\vdots&\vdots&\vdots&\vdots
\end{bmatrix},
\]
\end{example}

\begin{example}               {\bf The operator $\Gamma_{0}$ is co-isometric}. Then
\[
\sH_0=\wt\sH_0=\sum\limits_{n=0}^\infty\bigoplus\sD_{\Gamma_0},
\]
\[
\begin{array}{l}
\cM_0=I_\sM\bigoplus I_{\sD_{\Gamma_0}}\bigoplus
I_{\sD_{\Gamma_0}}\bigoplus\ldots,\\
\wt\cM_0=I_\sN\bigoplus I_{\sD_{\Gamma_0}}\bigoplus
I_{\sD_{\Gamma_0}}\bigoplus\ldots,\\
\cL_0={\bf J}^{(c)}_{\Gamma_0}\bigoplus I_{\sD_{\Gamma_0}}\bigoplus
I_{\sD_{\Gamma_0}}\bigoplus\ldots,
\end{array}
\]
\[
\cU_0=\wt\cU_0=\begin{bmatrix}\Gamma_0&0&0&0&\ldots\cr
D_{\Gamma_0}&0&0&0&\ldots\cr
 0&I_{\sD_{\Gamma_0}}&0&0&\ldots\cr
0&0&I_{\sD_{\Gamma_0}}&0&\ldots\cr
\vdots&\vdots&\vdots&\vdots&\vdots
  \end{bmatrix}.
\]
\end{example}
\begin{example}
                          {\bf The operator $\Gamma_{2}$ is co-isometric}. In this case
\[
\begin{array}{l}
\sH_0=\begin{array}{l}\sD_{\Gamma_{0}}\\\oplus\\\sD_{\Gamma^*_{1}}\end{array}\bigoplus
\sD_{\Gamma_{2}}\bigoplus\sD_{\Gamma_{2}}\bigoplus\ldots\bigoplus\sD_{\Gamma_{2}}\bigoplus\ldots,\\
 \wt\sH_0=\begin{array}{l}\sD_{\Gamma^*_{0}}\\\oplus\\\sD_{\Gamma_{1}}\end{array}\bigoplus
\sD_{\Gamma_{2}}\bigoplus\sD_{\Gamma_{2}}\bigoplus\ldots\bigoplus\sD_{\Gamma_{2}}\bigoplus\ldots,
\end{array}
\]
\[
\begin{array}{l}
\cU_0=\cL_0\cM_0=\begin{bmatrix}{\bf J}_{\Gamma_0}\cr & {\bf
J}^{(c)}_{\Gamma_{2}}\cr &&  I_{\sD_{\Gamma_{2}}} \cr &&&
I_{\sD_{\Gamma_{2}}} \cr &&&&\ddots\end{bmatrix}
\begin{bmatrix}I_\sM\cr&{\bf
J}_{\Gamma_1}\cr &&I_{\sD_{\Gamma_{2}}}\cr
&&&I_{\sD_{\Gamma_{2}}}\cr &&&&\ddots
\end{bmatrix}=\\
=\begin{bmatrix}\Gamma_0&D_{\Gamma^*_0}\Gamma_1&
D_{\Gamma^*_0}D_{\Gamma^*_1}&0&0&0&\ldots\cr
D_{\Gamma_0}&-\Gamma^*_0\Gamma_1&-\Gamma^*_0D_{\Gamma^*_1}&0&0&0&\ldots\cr
0&\Gamma_2D_{\Gamma_1}&-\Gamma_2\Gamma^*_1&0&0&0&\ldots\cr
0&D_{\Gamma_2}D_{\Gamma_1}&-D_{\Gamma_2}\Gamma^*_1&0&0&0&\ldots\cr
0&0&0&I_{\sD_{\Gamma_2}}&0&0&\ldots\cr
0&0&0&0&I_{\sD_{\Gamma_2}}&0&\ldots\cr
\vdots&\vdots&\vdots&\vdots&\vdots&\vdots&\vdots
\end{bmatrix},
\end{array}
\]
\[
\begin{array}{l}
\wt\cU_0=\wt\cM_0\cL_0=\begin{bmatrix}I_\sM\cr&{\bf J}_{\Gamma_1}\cr
&&I_{\sD_{\Gamma_{2}}}\cr &&&I_{\sD_{\Gamma_{2}}}\cr &&&&\ddots
\end{bmatrix}
\begin{bmatrix}{\bf J}_{\Gamma_0}\cr & {\bf
J}^{(c)}_{\Gamma_{2}}\cr &&  I_{\sD_{\Gamma_{2}}} \cr &&&
I_{\sD_{\Gamma_{2}}} \cr &&&&\ddots\end{bmatrix}=\\
=\begin{bmatrix}\Gamma_0&D_{\Gamma^*_0}&0&0&0&0&\ldots\cr
\Gamma_1D_{\Gamma_0}&-\Gamma_1\Gamma^*_0&D_{\Gamma^*_1}\Gamma_2&0&0&0&\ldots\cr
D_{\Gamma_1}D_{\Gamma_0}&-D_{\Gamma_1}\Gamma^*_0&-\Gamma^*_1\Gamma_2&0&0&0&\ldots\cr
0&0&D_{\Gamma_2}&0&0&0&\ldots\cr
0&0&0&I_{\sD_{\Gamma_2}}&0&0&\ldots\cr
0&0&0&0&I_{\sD_{\Gamma_2}}&0&\ldots\cr
\vdots&\vdots&\vdots&\vdots&\vdots&\vdots&\vdots
\end{bmatrix}.
\end{array}
\]
\end{example}
\begin{example}
{\bf The operator $\Gamma_1$ is isometric}.
 In this case
\[
\begin{array}{l}
\sH_0=
\begin{array}{l}\sD_{\Gamma_{0}}\\\oplus\\\sD_{\Gamma^*_{1}}\end{array}\bigoplus
\sD_{\Gamma^*_{1}}\bigoplus\sD_{\Gamma^*_{1}}\bigoplus\ldots\bigoplus\sD_{\Gamma^*_{1}}\bigoplus\ldots,\\
  \wt \sH_0=\sD_{\Gamma^*_0}\bigoplus\sD_{\Gamma^*_1}\bigoplus\sD_{\Gamma^*_1}\bigoplus\ldots
  \bigoplus\sD_{\Gamma^*_1}\bigoplus\ldots.
\end{array}
\]
\[
\begin{array}{l}
\cU_0=\cL_0\cM_0=\begin{bmatrix}{\bf J}_{\Gamma_0}\cr &
I_{\sD_{\Gamma^*_{1}}} \cr && I_{\sD_{\Gamma^*_{1}}} \cr &&&
I_{\sD_{\Gamma^*_{1}}} \cr &&&&\ddots\end{bmatrix}
\begin{bmatrix}I_\sM\cr&{\bf
J}^{(r)}_{\Gamma_1}\cr &&I_{\sD_{\Gamma^*_{1}}}\cr
&&&I_{\sD_{\Gamma^*_{1}}}\cr &&&&\ddots
\end{bmatrix}=\\
=\begin{bmatrix}\Gamma_0&D_{\Gamma^*_0}\Gamma_1&
D_{\Gamma^*_0}&0&0&0&0&0&\ldots\cr
D_{\Gamma_0}&-\Gamma^*_0\Gamma_1&-\Gamma^*_0&0&0&0&0&0&\ldots\cr
0&0&0&I_{\sD_{\Gamma^*_1}}&0&0&0&0&\ldots\cr
0&0&0&0&I_{\sD_{\Gamma^*_1}}&0&0&0&\ldots\cr
0&0&0&0&0&I_{\sD_{\Gamma^*_1}}&0&0&\ldots \cr
\vdots&\vdots&\vdots&\vdots&\vdots&\vdots&\vdots&\vdots&\vdots
\end{bmatrix},
\end{array}
\]
\[
\begin{array}{l}
\wt\cU_0=\wt\cM\cL_0=\begin{bmatrix}I_\sN\cr&{\bf
J}^{(r)}_{\Gamma_1}\cr &&I_{\sD_{\Gamma^*_{1}}}\cr
&&&I_{\sD_{\Gamma^*_{1}}}\cr &&&&\ddots
\end{bmatrix}
\begin{bmatrix}{\bf J}_{\Gamma_0}\cr &
I_{\sD_{\Gamma^*_{1}}} \cr && I_{\sD_{\Gamma^*_{1}}} \cr &&&
I_{\sD_{\Gamma^*_{1}}} \cr &&&&\ddots\end{bmatrix}=\\
 =\begin{bmatrix}\Gamma_0&D_{\Gamma^*_0}&
0&0&0&0&0&0&\ldots\cr
\Gamma_1D_{\Gamma_0}&-\Gamma_1\Gamma^*_0&I_{\sD_{\Gamma^*_1}}&0&0&0&0&0&\ldots\cr
0&0&0&I_{\sD_{\Gamma^*_1}}&0&0&0&0&\ldots\cr
0&0&0&0&I_{\sD_{\Gamma^*_1}}&0&0&0&\ldots\cr
\vdots&\vdots&\vdots&\vdots&\vdots&\vdots&\vdots&\vdots&\vdots
\end{bmatrix}.
\end{array}
\]
\end{example}

\begin{example}
{\bf The operator $\Gamma_3$ is isometric}.
\[
\begin{array}{l}
\sH_0=
\begin{array}{l}\sD_{\Gamma_{0}}\\\oplus\\\sD_{\Gamma^*_{1}}\end{array}\bigoplus
\begin{array}{l}\sD_{\Gamma_{2}}\\\oplus\\\sD_{\Gamma^*_{3}}\end{array}\bigoplus
\sD_{\Gamma^*_{3}}\bigoplus\sD_{\Gamma^*_{3}}\bigoplus\ldots\bigoplus\sD_{\Gamma^*_{3}}\bigoplus\ldots,\\
  \wt \sH_0=\begin{array}{l}\sD_{\Gamma^*_{0}}\\\oplus\\\sD_{\Gamma^*_{1}}\end{array}\bigoplus
 \sD_{\Gamma^*_2}\bigoplus\sD_{\Gamma^*_3}\bigoplus\sD_{\Gamma^*_3}\bigoplus\ldots
  \bigoplus\sD_{\Gamma^*_1}\bigoplus\ldots.
\end{array}
\]
\[
\begin{array}{l}
\cU_0=\cL_0\cM_0=\begin{bmatrix}{\bf J}_{\Gamma_0}\cr &{\bf
J}_{\Gamma_2}\cr && I_{\sD_{\Gamma^*_{3}}} \cr &&&
I_{\sD_{\Gamma^*_{3}}} \cr &&&& I_{\sD_{\Gamma^*_{3}}} \cr
&&&&&\ddots\end{bmatrix}
\begin{bmatrix}I_\sM \cr &{\bf
J}_{\Gamma_1}\cr&&{\bf J}^{(r)}_{\Gamma_3}\cr
&&&I_{\sD_{\Gamma^*_{3}}}\cr &&&&I_{\sD_{\Gamma^*_{3}}}\cr
&&&&&\ddots
\end{bmatrix}=\\
=\begin{bmatrix}\Gamma_0&D_{\Gamma^*_0}\Gamma_1&
D_{\Gamma^*_0}D_{\Gamma^*_1}&0&0&0&0&0&\ldots\cr
D_{\Gamma_0}&-\Gamma^*_0\Gamma_1&-\Gamma^*_0D_{\Gamma^*_1}&0&0&0&0&0&\ldots\cr
0&\Gamma_2D_{\Gamma_1}&-\Gamma_2\Gamma^*_1&D_{\Gamma^*_2}\Gamma_3&D_{\Gamma^*_2}&0&0&0&\ldots\cr
0&D_{\Gamma_2}D_{\Gamma_1}&-D_{\Gamma_2}\Gamma^*_1&-\Gamma^*_2\Gamma_3&-\Gamma^*_2&0&0&0&\ldots\cr
0&0&0&0&0&I_{\sD_{\Gamma^*_3}}&0&0&\ldots\cr
0&0&0&0&0&0&I_{\sD_{\Gamma^*_3}}&0&\ldots \cr
\vdots&\vdots&\vdots&\vdots&\vdots&\vdots&\vdots&\vdots&\vdots
\end{bmatrix},
\end{array}
\]
\[
\begin{array}{l}
\wt\cU_0=\wt\cM_0\cL_0=\begin{bmatrix}I_\sN \cr &{\bf
J}_{\Gamma_1}\cr&&{\bf J}^{(r)}_{\Gamma_3}\cr
&&&I_{\sD_{\Gamma^*_{3}}}\cr &&&&I_{\sD_{\Gamma^*_{3}}}\cr
&&&&&\ddots
\end{bmatrix}
\begin{bmatrix}{\bf J}_{\Gamma_0}\cr &{\bf
J}_{\Gamma_2}\cr && I_{\sD_{\Gamma^*_{3}}} \cr &&&
I_{\sD_{\Gamma^*_{3}}} \cr &&&& I_{\sD_{\Gamma^*_{3}}} \cr
&&&&&\ddots\end{bmatrix}=\\
=\begin{bmatrix}\Gamma_0&D_{\Gamma^*_0}& 0&0&0&0&0&0&\ldots\cr
\Gamma_1D_{\Gamma_0}&-\Gamma_1\Gamma^*_0&D_{\Gamma^*_1}\Gamma_2&D_{\Gamma^*_1}D_{\Gamma^*_2}&0&0&0&0&\ldots\cr
D_{\Gamma_1}D_{\Gamma_0}&-D_{\Gamma_1}\Gamma^*_0&-\Gamma^*_1\Gamma_2&-\Gamma^*_1D_{\Gamma^*_2}&0&0&0&0&\ldots\cr
0&0&\Gamma_3D_{\Gamma_2}&-\Gamma_3\Gamma^*_2&I_{\sD_{\Gamma^*_3}}&0&0&0&\ldots\cr
0&0&0&0&0&I_{\sD_{\Gamma^*_3}}&0&0&\ldots\cr
0&0&0&0&0&0&I_{\sD_{\Gamma^*_3}}&0&\ldots \cr
\vdots&\vdots&\vdots&\vdots&\vdots&\vdots&\vdots&\vdots&\vdots
\end{bmatrix}.
\end{array}
\]

\end{example}
\begin{example}
{\bf The operator $\Gamma_{5}$ is co-isometric}.
\[
\begin{array}{l}
\sH_0=\begin{array}{l}\sD_{\Gamma_{0}}\\\oplus\\\sD_{\Gamma^*_{1}}\end{array}\bigoplus
\begin{array}{l}\sD_{\Gamma_{2}}\\\oplus\\\sD_{\Gamma^*_{3}}\end{array}
\bigoplus
\sD_{\Gamma_{4}}\bigoplus\sD_{\Gamma_{5}}\bigoplus\ldots\bigoplus\sD_{\Gamma_{5}}\bigoplus\ldots,\\
\wt\sH_0=\begin{array}{l}\sD_{\Gamma^*_{0}}\\\oplus\\\sD_{\Gamma_{1}}\end{array}\bigoplus
\begin{array}{l}\sD_{\Gamma^*_{2}}\\\oplus\\\sD_{\Gamma_{3}}\end{array}\bigoplus
\begin{array}{l}\sD_{\Gamma^*_{4}}\\\oplus\\\sD_{\Gamma_{5}}\end{array}\bigoplus
\sD_{\Gamma_{5}}\bigoplus\sD_{\Gamma_{5}}\bigoplus\ldots\bigoplus\sD_{\Gamma_{5}}\bigoplus\ldots,
\end{array}
\]
\[
\begin{array}{l}
\cL_0= {\bf J}_{\Gamma_0}\bigoplus {\bf J}_{\Gamma_2}\bigoplus{\bf
J}_{\Gamma_4} \bigoplus I_{\sD_{\Gamma_{5}}}\bigoplus
I_{\sD_{\Gamma_{5}}}\bigoplus\ldots,\\
 \cM_0=I_\sM\bigoplus{\bf
J}_{\Gamma_{1}}\bigoplus {\bf
J}_{\Gamma_{3}}\bigoplus{\bf J}^{(c)}_{\Gamma_{5}}
\bigoplus I_{\sD_{\Gamma_{5}}} \bigoplus
I_{\sD_{\Gamma_{5}}}\bigoplus \ldots,\\
\wt\cM_0=I_\sN\bigoplus{\bf J}_{\Gamma_{1}}\bigoplus {\bf
J}_{\Gamma_{3}}\bigoplus{\bf J}^{(c)}_{\Gamma_{5}}
\bigoplus I_{\sD_{\Gamma_{5}}} \bigoplus
I_{\sD_{\Gamma_{5}}}\bigoplus \ldots.
\end{array}
\]
\[
\cU_0=\cL_0\cM_0=\begin{bmatrix}\Gamma_0&D_{\Gamma^*_0}\Gamma_1&
D_{\Gamma^*_0}D_{\Gamma^*_1}&0&0&0&0&0&0&\ldots\cr
D_{\Gamma_0}&-\Gamma^*_0\Gamma_1&-\Gamma^*_0D_{\Gamma^*_1}&0&0&0&0&0&0&\ldots\cr
0&\Gamma_2D_{\Gamma_1}&-\Gamma_2\Gamma^*_1&D_{\Gamma^*_2}\Gamma_3&D_{\Gamma^*_2}D_{\Gamma^*_3}&0&0&0&0&\ldots\cr
0&D_{\Gamma_2}D_{\Gamma_1}&-D_{\Gamma_2}\Gamma^*_1&-\Gamma^*_2\Gamma_3&-\Gamma^*_2D_{\Gamma^*_3}&0&0&0&0&\ldots\cr
0&0&0&\Gamma_4D_{\Gamma_3}&-\Gamma_4\Gamma^*_3&D_{\Gamma^*_4}\Gamma_5&0&0&0&\ldots\cr
0&0&0&D_{\Gamma_4}D_{\Gamma_3}&-D_{\Gamma_4}\Gamma^*_3&-\Gamma^*_4\Gamma_5&0&0&0&\ldots\cr
 0&0&0&0&0&{D_{\Gamma_5}}&0&0&0&\ldots\cr
0&0&0&0&0&0&I_{\sD_{\Gamma_5}}&0&0&\ldots \cr
0&0&0&0&0&0&0&I_{\sD_{\Gamma_5}}&0&\ldots\cr
\vdots&\vdots&\vdots&\vdots&\vdots&\vdots&\vdots&\vdots&\vdots&\vdots
\end{bmatrix},
\]
\[
\wt\cU_0=\wt\cM_0\cL_0=\begin{bmatrix}\Gamma_0&D_{\Gamma^*_0}&
0&0&0&0&0&0&0&\ldots\cr
\Gamma_1D_{\Gamma_0}&-\Gamma_1\Gamma^*_0&D_{\Gamma^*_1}\Gamma_2&D_{\Gamma^*_1}D_{\Gamma^*_2}&0&0&0&0&0&\ldots\cr
D_{\Gamma_1}D_{\Gamma_0}&-D_{\Gamma_1}\Gamma^*_0&-\Gamma^*_1\Gamma_2&-\Gamma^*_1D_{\Gamma^*_2}&0&0&0&0&0&\ldots\cr
0&0&\Gamma_3D_{\Gamma_2}&-\Gamma_3\Gamma^*_2&-D_{\Gamma^*_3}\Gamma_4&D_{\Gamma^*_3}D_{\Gamma^*_4}&0&0&0&\ldots\cr
0&0&0&\Gamma_4D_{\Gamma_3}&-\Gamma_4\Gamma^*_3&D_{\Gamma^*_4}\Gamma_5&0&0&0&\ldots\cr
0&0&0&0&\Gamma_5D_{\Gamma_4}&-\Gamma_5\Gamma^*_4&0&0&0&\ldots\cr
 0&0&0&0&D_{\Gamma_5}D_{\Gamma_4}&-D_{\Gamma_5}\Gamma^*_4&0&0&0&\ldots\cr
0&0&0&0&0&0&I_{\sD_{\Gamma_5}}&0&0&\ldots \cr
0&0&0&0&0&0&0&I_{\sD_{\Gamma_5}}&0&\ldots\cr
\vdots&\vdots&\vdots&\vdots&\vdots&\vdots&\vdots&\vdots&\vdots&\vdots
\end{bmatrix}.
\]

\end{example}
\begin{example}
\label{even}{\bf The operator $\Gamma_{2N}$ is unitary}. In this
case
\[
\begin{array}{l}
\sH_0=\sum\limits_{n=0}^{N-1}
\bigoplus\begin{array}{l}\sD_{\Gamma_{2n}}\\\oplus\\\sD_{\Gamma^*_{2n+1}}\end{array},\\
\wt\sH_0=\sum\limits_{n=0}^{N-1}
\bigoplus\begin{array}{l}\sD_{\Gamma^*_{2n}}\\\oplus\\\sD_{\Gamma_{2n+1}}\end{array}
\end{array},
\]
\[
\cU_0=\begin{bmatrix}{\bf J}_{\Gamma_0}\cr &{\bf J}_{\Gamma_2}\cr &
&\ddots\cr &&& {\bf J}_{\Gamma_{2(N-1)}}\cr
&&&&\Gamma_{2N}\end{bmatrix}
\begin{bmatrix}I_\sM\cr&{\bf J}_{\Gamma_1}\cr &
&{\bf J}_{\Gamma_3}\cr & & &\ddots\cr &&&& {\bf
J}_{\Gamma_{2N-1}}\end{bmatrix},
\]
\[
\wt\cU_0=\begin{bmatrix}I_\sN\cr&{\bf J}_{\Gamma_1}\cr & &{\bf
J}_{\Gamma_3}\cr & & &\ddots\cr &&&& {\bf
J}_{\Gamma_{2N-1}}\end{bmatrix}\begin{bmatrix}{\bf J}_{\Gamma_0}\cr
&{\bf J}_{\Gamma_2}\cr & &\ddots\cr &&& {\bf J}_{\Gamma_{2(N-1)}}\cr
&&&&\Gamma_{2N}\end{bmatrix}.
\]
If $N=1$ ($\Gamma_2$ is unitary) then we have
\[
\cU_0=\begin{bmatrix}\Gamma_0&D_{\Gamma^*_0}\Gamma_1&
D_{\Gamma^*_0}D_{\Gamma^*_1}\cr
D_{\Gamma_0}&-\Gamma^*_0\Gamma_1&-\Gamma^*_0D_{\Gamma^*_1}\cr
0&\Gamma_2 D_{\Gamma_1}&-\Gamma_2\Gamma^*_1
\end{bmatrix},\;\wt\cU_0=\begin{bmatrix}\Gamma_0&D_{\Gamma^*_0}&0&\cr
\Gamma_1D_{\Gamma_0} &-\Gamma_1\Gamma^*_0&D_{\Gamma^*_1}\Gamma_2\cr
D_{\Gamma_1}D_{\Gamma_0}&-D_{\Gamma_1}\Gamma^*_0&-\Gamma^*_1\Gamma_2
\end{bmatrix}
\]
\end{example}

\begin{example}
\label{odd}{\bf The operator $\Gamma_{2N+1}$ is unitary}.
\[
\begin{array}{l}
\sH_0=\sD_{\Gamma_0},\;\wt\sH_0=\sD_{\Gamma^*_0}\;\mbox{if}\;N=0,\\
 \sH_0=\sum\limits_{n=0}^{N-1}
\bigoplus\begin{array}{l}\sD_{\Gamma_{2n}}\\\oplus\\\sD_{\Gamma^*_{2n+1}}\end{array}\bigoplus\sD_{\Gamma_{2N}},\;
\wt\sH_0=\sum\limits_{n=0}^{N-1}
\bigoplus\begin{array}{l}\sD_{\Gamma^*_{2n}}\\\oplus\\\sD_{\Gamma_{2n+1}}\end{array}\bigoplus\sD_{\Gamma^*_{2N}}\;\mbox{if}\;
N\ge 1
\end{array},
\]
\[
\begin{array}{l}
\cU_0=\begin{bmatrix}\Gamma_0&D_{\Gamma^*_0}\cr
D_{\Gamma_0}&-\Gamma^*_0\end{bmatrix}\begin{bmatrix}I_\sM&0\cr
0&\Gamma_1
\end{bmatrix}=\begin{bmatrix}\Gamma_0&D_{\Gamma^*_0}\Gamma_1\cr D_{\Gamma_0}&-\Gamma^*_0\Gamma_1
\end{bmatrix},\\
\wt\cU_0=\begin{bmatrix}I_\sN&0\cr 0&\Gamma_1
\end{bmatrix}\begin{bmatrix}\Gamma_0&D_{\Gamma^*_0}\cr
D_{\Gamma_0}&-\Gamma^*_0\end{bmatrix}=\begin{bmatrix}\Gamma_0&D_{\Gamma^*_0}\cr
\Gamma_1D_{\Gamma_0}&-\Gamma_1\Gamma^*_0
\end{bmatrix} ,\;\mbox{if}\; N=0,
\end{array}
\]
\[
\cU_0=\begin{bmatrix}{\bf J}_{\Gamma_0}\cr &{\bf J}_{\Gamma_2}\cr &
&\ddots\cr &&& {\bf J}_{\Gamma_{2N}}
\end{bmatrix}
\begin{bmatrix}I_\sM\cr&{\bf J}_{\Gamma_1}\cr &
&{\bf J}_{\Gamma_3}\cr & & &\ddots\cr &&&& {\bf
J}_{\Gamma_{2N-1}}\cr &&&&&\Gamma_{2N+1}\end{bmatrix},
\]
\[
\wt\cU_0=\begin{bmatrix}I_\sN\cr&{\bf J}_{\Gamma_1}\cr & &{\bf
J}_{\Gamma_3}\cr & & &\ddots\cr &&&& {\bf J}_{\Gamma_{2N-1}}\cr
&&&&&\Gamma_{2N+1}\end{bmatrix}\begin{bmatrix}{\bf J}_{\Gamma_0}\cr
&{\bf J}_{\Gamma_2}\cr & &\ddots\cr &&& {\bf J}_{\Gamma_{2N}}
\end{bmatrix},\;\mbox{if}\; N\ge 1.
\]
If $N=1$ ($\Gamma_3$ is unitary) then
\[
\cU_0=\begin{bmatrix}\Gamma_0&D_{\Gamma^*_0}\Gamma_1&
D_{\Gamma^*_0}D_{\Gamma^*_1}&0\cr
D_{\Gamma_0}&-\Gamma^*_0\Gamma_1&-\Gamma^*_0D_{\Gamma^*_1}&0\cr
0&\Gamma_2
D_{\Gamma_1}&-\Gamma_2\Gamma^*_1&D_{\Gamma^*_2}\Gamma_3\cr
0&D_{\Gamma_2}D_{\Gamma_1}&-D_{\Gamma_2}\Gamma^*_1&-\Gamma^*_2\Gamma_3
\end{bmatrix},\;\wt\cU_0=\begin{bmatrix}\Gamma_0&D_{\Gamma^*_0}&0&0&\cr
\Gamma_1D_{\Gamma_0}
&-\Gamma_1\Gamma^*_0&D_{\Gamma^*_1}\Gamma_2&D_{\Gamma^*_1}D_{\Gamma^*_2}\cr
D_{\Gamma_1}D_{\Gamma_0}&-D_{\Gamma_1}\Gamma^*_0&-\Gamma^*_1\Gamma_2&-\Gamma^*_1D_{\Gamma^*_2}\cr
0&0&\Gamma_3 D_{\Gamma_2}&-\Gamma_3\Gamma^*_2
\end{bmatrix}.
\]
\end{example}
\section{Unitary operators with cyclic subspaces, dilations, and block operator CMV matrices}\label{DIL}

\subsection{Carath\'{e}odory class functions associated with conservative systems}
\begin{definition}
\label{Carath} Let $\sM$ be a separable Hilbert space. The class
${\bf C}(\sM)$ of $\bL(\sM)$-valued functions holomorphic on the
unit disk $\dD$ and having positive real part for all
$\lambda\in\dD$ is called the Carath\'{e}odory class.
\end{definition}

Consider a conservative systems $\tau =\left\{\begin{bmatrix} D&C
\cr B&A\end{bmatrix};\sM,\sM, \sH\right\}$ whose input and output
spaces coincide. Put
\[
\cH=\sM\oplus\sH
\]
 and let the
function $F_\tau(z)$ be defined as follows
\begin{equation}
\label{Carat} F_\tau(\lambda)=P_\sM(U_\tau+\lambda
I_\cH)(U_\tau-\lambda I_\cH)^{-1}\uphar\sM,\; \lambda\in\dD,
\end{equation}
where
\[
 U_\tau=\begin{bmatrix} D&C \cr B&A\end{bmatrix} :
\begin{array}{l} \sM \\\oplus\\ \sH \end{array} \to
\begin{array}{l} \sM \\\oplus\\ \sH \end{array}
\]
is unitary operator in $\cH$ associated with the system $\tau$. The
function $F_\tau(z)$ is holomorphic in $\dD$ and
\[
F_\tau(\lambda)+F^*_\tau(\lambda)=2(1-|\lambda|^2)P_\sM(U^*_\tau-\bar
\lambda I_\cH)^{-1}(U_\tau-\lambda I_\cH)^{-1}\uphar\sM.
\]
It follows that $F_\tau(\lambda)+F^*_\tau(\lambda)\ge 0$ for all
$\lambda\in \dD$.

The function $F_\tau(\lambda)$ defined by \eqref{Carat} belongs to
the Carath\'{e}odory class ${\bf C}(\sM)$ and, in addition,
$F_\tau(0)=I_\sM$. We also shall consider the function
\[
\wt
F_\tau(\lambda):=F^*_\tau(\bar\lambda)=P_\sM(I_\cH+\lambda
U_\tau)(I_\cH-\lambda U_\tau)^{-1}.
\]
 The
functions $F_\tau$ and $\wt F_\tau$ we will call the
Carath\'{e}odory functions associated with conservative system $\tau
=\left\{\begin{bmatrix} D&C \cr B&A\end{bmatrix};\sM,\sM,
\sH\right\}$.
\begin{proposition}
\label{CarSch} Let
\[
\tau =\left\{\begin{bmatrix} D&C \cr B&A\end{bmatrix};\sM,\sM,
\sH\right\}
\]
be a conservative system. Then the transfer function
$\Theta_\tau(\lambda)$ and the  Carath\'{e}odory function
$F_\tau(\lambda)$ are connected by the following relations
\begin{equation}
\label{FTh}
\begin{array}{l}
\Theta^*_\tau(\bar\lambda)=\cfrac{1}{\lambda}\,(F_\tau(\lambda)-I_\sM)(F_\tau(\lambda)+I_\sM)^{-1},\\
F_\tau(\lambda)=(I_\sM+\lambda\Theta^*_\tau(\bar\lambda))(I_\sM-\lambda\Theta^*_\tau(\bar\lambda))^{-1},\;
\lambda\in\dD.
\end{array}
\end{equation}
\end{proposition}
\begin{proof} We use the well known Schur--Frobenius formula for the inverse
of block operators.  Let $\Phi$ be a bounded linear operator given
by the block operator matrix
\[
\Phi=\begin{pmatrix}X&Y \cr Z&W \end{pmatrix}:
\begin{array}{l} \sM \\\oplus\\ \sH \end{array} \to
\begin{array}{l} \sM \\\oplus\\ \sH \end{array}.
\]
 Suppose that
$W^{-1}\in\bL(\sH)$ and $(X-YW^{-1}Z)^{-1}\in\bL(\sM)$. Then
$\Phi^{-1}\in \bL(\sM\oplus\sH,\sM\oplus\sH)$ and
\[
\Phi^{-1}=\begin{pmatrix}K^{-1}& -K^{-1}YW^{-1}\cr
-W^{-1}ZK^{-1}&W^{-1}+W^{-1}ZK^{-1}YW^{-1}
\end{pmatrix},
\]
where $K=X-YW^{-1}Z$. Applying this formula for
\[
\Phi=I_\cH-\lambda U_\tau=\begin{pmatrix}I_\sM-\lambda D&-\lambda
C\cr-\lambda B&I_\sH-\lambda A\end{pmatrix}, \quad \lambda\in\dD,
\]
we get $K=I_\sM-\lambda D-\lambda^2C(I_\sH-\lambda
A)^{-1}B=I_\sM-\lambda\Theta_\tau(\lambda)$. Therefore
\[
P_\sM(I_\cH- \lambda U_\tau)^{-1}\uphar\sM=(I_\sM-\lambda
\Theta_\tau(\lambda))^{-1}, \quad \lambda\in\dD.
\]
Hence
\[
P_\sM (I_\cH-\lambda U^*_\tau)^{-1}\uphar\sM=(I_\sM-\lambda
\Theta^*_\tau(\bar\lambda))^{-1}, \quad \lambda\in\dD.
\]
Since $U_\tau$ is unitary, from \eqref{Carat} we get
\[
\begin{array}{l}
F_\tau(\lambda)=P_\sM(I_\cH+\lambda U^*_\tau)(I_\cH-\lambda
U^*_\tau)^{-1}\uphar\sM=\\
=-I_\sM+2P_\sM(I_\cH-\lambda
U^*_\tau)^{-1}\uphar\sM=-I_\sM+2(I_\sM-\lambda
\Theta^*_\tau(\bar\lambda))^{-1}=\\
=(I_\sM+\lambda \Theta^*_\tau(\bar\lambda))(I_\sM-\lambda
\Theta^*_\tau(\bar\lambda))^{-1}, \quad \lambda\in\dD.
\end{array}
\]
\end{proof}
The following theorem is well known (see \cite{Br1}).
\begin{theorem}
\label{CarOper} Let $\sM$ be a separable Hilbert space and let
$F(\lambda)\in {\bf C}(\sM)$. Then
\begin{enumerate}
\item $F(\lambda)$
admits the integral representation
\[
F(\lambda)=\frac{1}{2}(F(0)-F^*(0))+\int\limits_{0}^{2\pi}\frac{e^{it}+\lambda}{e^{it}-\lambda}\,
d\Sigma(t),
\]
where $\Sigma(t)$ is a non-decreasing  and nonnegative
$\bL(\sM)$-valued function on $[0, 2\pi]$;
\item under the condition $F(0)=I_\sM$ there exists a Hilbert space $\cH$ containing $\sM$ as a
subspace, and a unitary operator $U$ in $\cH$ such that
\[
F(\lambda)=P_\sM(U+\lambda I_\cH)(U-\lambda I_\cH)^{-1}\uphar\sM;
\]
moreover, the pair $\{\cH,U\}$ can be chosen minimal in the sense
\[
\cspan\{U^n\sM,\; n\in\dZ\}=\cH.
\]
\end{enumerate}
\end{theorem}
\begin{proposition}\cite{D1}.
\label{consmin1} The conservative system
\[
\tau=\left\{\begin{bmatrix}D&C \cr B&A
\end{bmatrix}; \sM,\sM,\sH\right\}
\]
is simple if and only if
\[
\cspan\{U^n_\tau \sM,\;n\in\dZ\}=\cH.
\]
\end{proposition}
\begin{proof}
Let $\tau$ be a simple conservative system. Suppose $h\in\cH$ and
$h$ is orthogonal to $U^n_\tau\sM$ for all $n\in\dZ$. Then the
vectors $U^{*n}_\tau h$ are orthogonal to $\sM$ in $\cH$ for all
$n\in\dZ$. It follows that $h\in\sH$ and
\begin{equation}
\label{SIM}
\begin{array}{l}
Ch=CAh=CA^2h=\ldots=CA^nh=\ldots=0,\\
B^*h=B^*A^*h=B^*A^{*2}h=\ldots B^*A^{*n}h=\ldots=0.
\end{array}
\end{equation}
Hence $h\in \left(\bigcap_{n\ge 0}\ker(CA^n)\right)\cap\left(\bigcap_{n\ge 0}\ker(B^*A^{*n})\right)$. Since $\tau$ is simple we get $h=0$, i.e.,
\[\cspan\{U^n_\tau \sM,\;n\in\dZ\}=\cH.
\]
Conversely, let $\cspan\{U^n_\tau \sM,\;n\in\dZ\}=\cH$. Suppose thar
relations \eqref{SIM} hold for some $h\in\sH$. Then $h\perp
U^n_\tau\sM$ for all $n\in\dZ$. Hence $h=0$ and $\tau$ is simple.
\end{proof}

\subsection{Unitary operators with cyclic subspaces}
Let $U$ be a unitary operator in a separable Hilbert space $\sK$ and
let $\sM$ be a subspace of $\sK$. Put $\sH=\sK\ominus\sM$. Then $U$
takes the block operator matrix form
\[
U=\begin{bmatrix}D&C\cr B&A  \end{bmatrix}:\begin{array}{l}
\sM\\\oplus\\\sH\end{array}\to \begin{array}{l}
\sM\\\oplus\\\sH\end{array}.
\]
Since $U$ is unitary, the system
\[
\eta=\left\{\begin{bmatrix}D&C \cr B&A
\end{bmatrix}; \sM,\sM,\sH\right\}
\]
is conservative. By Proposition \ref{consmin1} the system $\eta$ is
simple if and only if
\begin{equation}
\label{uncycle}
 \cspan\{U^n \sM,\;n\in\dZ\}=\sK.
\end{equation}
A subspace $\sM$ of $\sK$ is called \textit{cyclic} for $U$ if the
condition \eqref{uncycle} is satisfied.

Define the Carath\'{e}odory function
\[
F_\sM(\lambda)=P_\sM(U+\lambda I_\cH)(U-\lambda
I_\cH)^{-1}\uphar\sM,\; \lambda\in\dD
\]
and a Schur function
\[
E_\sM(\lambda)=\cfrac{1}{\lambda}\,(F_\sM(\lambda)-I_\sM)(F_\sM(\lambda)+I_\sM)^{-1},\;
\lambda\in\dD.
\]
According to Proposition \ref{CarSch} the transfer function
$\Theta(\lambda)$ of the system $\eta$ and the function
$E_\sM(\lambda)$ are connected by the relation
\[
\Theta(\lambda)=E^*_\sM(\bar\lambda),\; \lambda\in\dD.
\]
\begin{theorem}
\label{CYCLE}Let $U$ be a unitary operator in a separable Hilbert
space and let $\sM$ be a cyclic subspace for $U$. Then $U$ is
unitarily equivalent to the block operator CMV matrices
$\cU_0(\{\Gamma_n\}_{n\ge 0})$ and $\wt\cU_0(\{\Gamma_n\}_{n\ge 0})$
in the Hilbert spaces $\cH=\sM\oplus\sH_0(\{\Gamma_n\}_{n\ge 0})$
and $\wt\cH=\sM\oplus\wt\sH_0(\{\Gamma_n\}_{n\ge 0})$, respectively,
where $\{\Gamma_n\}_{n\ge 0}$ are the Schur parameters of the
function
\[
\Theta(\lambda)=\cfrac{1}{\lambda}\,(F^*_\sM(\bar\lambda)-I_\sM)(F^*_\sM(\bar\lambda)+I_\sM)^{-1}.
\]
\end{theorem}
\begin{proof} Because $\sM$ is a cyclic subspace for $U$, the conservative system $\eta$ is simple.
  By Theorem \ref{cmvmod} the system $\eta$ is unitarily equivalent to the systems $\zeta_0$ and $\wt\zeta_0$ given by
\eqref{CMVMODEL}. From \eqref{UNSYS} it follows that $U$ is
unitarily equivalent to $\cU_0(\{\Gamma_n\}_{n\ge 0})$ and
$\wt\cU_0(\{\Gamma_n\}_{n\ge 0})$.
\end{proof}
Suppose that the cyclic subspace $\sM$ for unitary operator $U$ in
$\sK$ is one-dimensional. Let $\f\in\sM$, $||\f||=1,$ and let
$\mu(\zeta)=(\cE(\zeta)\f,\f)_\sK$, where $\cE(\zeta)$,
$\zeta\in\dT$, is the resolution of the identity for $U$.  Then the
scalar Carath\'{e}odory function $F(\lambda)$ is of the form
\[
F(\lambda)=\left((U+\lambda I_\cH)(U-\lambda
I_\cH)^{-1}\f,\f\right)_\sK=\int_{\dT}\frac{\zeta+\lambda}{\zeta-\lambda}\,d\mu(\zeta)\,,
\quad \lambda\in\dD.
\]
Thus, the function $F(\lambda)$ is associated with the probability
measure $\mu$ on $\dT$. The Schur function associated with $\mu$
\cite{Si1} is the function
\[
E(\lambda)=\cfrac{1}{\lambda}\,\cfrac{F(\lambda)-1}{F(\lambda)+1},\;\lambda\in\dD.
\]
By Geronimus theorem \cite{ger1} the Schur parameters
 of the function $E(\lambda)$ coincide with
Verblunsky coefficient $\{\alpha_n\}_{n\ge 0}$ of the measure $\mu$
(see \cite{Si1}). Let $\Theta(\lambda):=\overline{E(\bar\lambda)}$,
$\lambda\in\dD$ and let $\{\gamma_n\}_{n\ge 0}$ be the Schur
parameters of $\Theta$. Then $\bar\alpha_n=\gamma_n$ for all $n$ and
the CMV matrices $\cU_0=\cU_0(\{\gamma_n\}_{n\ge 0})$ and
$\wt\cU_0=\wt\cU_0(\{\gamma_n\}_{n\ge 0})$ coincide with the CMV
matrices $\cC$ and $\wt\cC$ given by \eqref{CMV11} and
\eqref{CMV12}, correspondingly. Observe that $\dim\sK=m$ $\iff$ the
function $E(\lambda)$ is the Blaschke product of the form
\[
E(\lambda)=e^{i\f}\prod_{k=1}^m \frac{\lambda-\lambda_k}{1-\bar
\lambda_k \lambda}.
\]

\subsection{Unitary dilations of a contraction} 
Let $T$ be a contraction acting in a Hilbert space $H$. The unitary
operator $U$ in a Hilbert space $\cH$ containing $H$ as a subspace
is called the unitary dilation of $T$ if $T^n=P_H U^n$ for all
$n\in\dN$ \cite{SF}. Two unitary dilations $U$ in $\cH$ and $U'$ in
$\cH'$ of $T$ are called isomorphic if there exists a unitary
operator $W\in\bL(\cH,\cH')$ such that
\[
W\uphar H=I_H \quad\mbox{and}\quad WU=U'W.
\]
It is established in \cite{SF} that for every contraction $T$ in the
Hilbert space $H$ there exists a unitary dilation $U$ in a space $H$
such that $U$ is \textit{is minimal} \cite{SF}, i.e.,
\[
\cspan\{U^n H,\; n\in\dZ\}=\cH.
\]
Moreover, two minimal unitary dilations of $T$ are isomorphic
\cite{SF}. The minimal unitary dilation by means of the infinite
matrix form is constructed in \cite{SF} on the base of Sch\"affer
paper \cite{Sch}. Below we show that the minimal unitary dilations
can be given by the operator CMV matrices.
\begin{theorem}
\label{UnitDil} Let $T$ be a contraction in a Hilbert space $H$.
Define the Hilbert spaces
\begin{equation}
\label{SPA}
\begin{array}{l}
\sH_0=\begin{array}{l}\sD_T\\\oplus\\\sD_{T^*}\end{array}\bigoplus
\begin{array}{l}\sD_T\\\oplus\\\sD_{T^*}\end{array}\bigoplus\cdots,\\
\wt\sH_0=\begin{array}{l}\sD_{T^*}\\\oplus\\\sD_{T}\end{array}\bigoplus
\begin{array}{l}\sD_{T^*}\\\oplus\\\sD_{T}\end{array}\bigoplus\cdots,
\end{array}
\end{equation}
and the Hilbert spaces $\cH_0=H\oplus\sH_0$, and $\wt\cH_0=
H\oplus\wt\sH_0.$  Let
\[
{\bf J}_0=\begin{bmatrix}0& I_{\sD_{T^*}}\cr
I_{\sD_T}&0\end{bmatrix}:
\begin{array}{l}\sD_T\\\oplus\\\sD_{T^*}\end{array}\to\begin{array}{l}\sD_{T^*}\\\oplus\\\sD_{T}\end{array}
\]
Define operators
\begin{equation}
\label{OPER}
\begin{array}{l} \cM_0=I_H\bigoplus {\bf J}_0\bigoplus
 {\bf J}_0\bigoplus\cdots:\cH_0\to \wt\cH_0,\\
 \cL_0={\bf J}_T\bigoplus{\bf J}_0\bigoplus{\bf
 J}_0\bigoplus\cdots:\wt\cH_0\to  \cH_0,
\end{array}
 \end{equation}
and
\begin{equation}
\label{DILA} \cU_0=\cL_0\cM_0: \cH_0\to \cH_0,\;
\wt\cU_0=\cM_0\cL_0:\wt\cH_0\to \wt\cH_0.
\end{equation}
Then $\{\cH_0, \cU_0\}$ and $\{\wt\cH_0,\wt\cU_0\}$ are unitarily
equivalent minimal unitary dilations of the operator $T$.
\end{theorem}
\begin{proof}
Define the $\bL(H)$-valued function
\[
\wt F(\lambda)=(I_H+\lambda T)(I_H-\lambda T)^{-1},\; \lambda\in\dD.
\]
Then the function $\wt F$ belongs to the Carath\'{e}odory class
${\bf C}(H)$ and
\[
\Theta(\lambda)=T=\frac{1}{\lambda}(\wt F(\lambda)-I_H)(\wt
F(\lambda)+I_H)^{-1},\;\lambda\in\dD,
\]
belongs to the Schur class ${\bf S}(H,H)$. The Schur parameters of
$\Theta$ is the sequence
\[
\Gamma_0=T,\Gamma_n= 0\in {\bf S}(\sD_T,\sD_{T^*}), \; n\in \dN.
\]
Let $\sH_0$ and $\wt\sH_0$ be defined by \eqref{SPA},
$\cH_0=H\oplus\sH_0$, $\wt\cH_0=H\oplus\wt\sH_0$. Then the operators
$\cU_0$ and $\wt\cU_0$ defined by \eqref{DILA} are the block
operator CMV matrices constructed by means of the Schur parameters
of $\Theta$. Let $\zeta_0=\{\cH_0,\sM,\sM,\sH_0\}$ and
$\wt\zeta_0=\{\wt\cU_0,\sM,\sM,\wt\sH_0\}$ be the corresponding
conservative systems. By Theorem \ref{cmvmod} the systems $\zeta_0$
and $\wt\zeta_0$ are simple, unitary equivalent, and their transfer
functions are equal $\Theta$. By Proposition \ref{CarSch} we have
\[
\begin{array}{l}
(I_H+\lambda T)(I_H-\lambda T)^{-1}=\wt
F(\lambda)=(I_H+\lambda\Theta(\lambda))(I_H-\lambda\Theta(\lambda))^{-1}=\\
=P_H(I_{\cH_0}+\lambda \cU_0)(I_{\cH_0}-\lambda \cU_0)^{-1}\uphar H.
\end{array}
\]
Hence
\[
T^n=P_H\cU_0^n\uphar H, \;n=0,1,\ldots.
\]
Therefore $\cU_0$ is a unitary dilation of $T$ in $\cH_0$. By
Proposition \ref{consmin1} this dilation is minimal. Similarly the
operator $\wt\cU_0$ is a minimal unitary dilation of $T$ in
$\wt\cH_0.$
\end{proof}
Taking into account \eqref{OPER}, \eqref{BLOKIT}, and
\eqref{BLOKIWT} we obtain the following operator matrix forms for
minimal unitary dilations $\cU_0$ and $\wt\cU_0$:
\[
\cU_0=\begin{bmatrix} T&0&D_{T^*}&0&0&0&0&0&0&\ldots\cr
D_T&0&-T^*&0&0&0&0&0&0&\ldots\cr
0&0&0&0&I_{\sD_{T^*}}&0&0&0&0&\ldots\cr
0&I_{\sD_T}&0&0&0&0&0&0&0&\ldots\cr
0&0&0&0&0&0&I_{\sD_{T^*}}&0&0&\ldots\cr
0&0&0&I_{\sD_T}&0&0&0&0&0&\ldots\cr
0&0&0&0&0&0&0&0&I_{\sD_{T^*}}&\ldots\cr
\vdots&\vdots&\vdots&\vdots&\vdots&\vdots&\vdots&\vdots&\vdots&\vdots
\end{bmatrix},
\]
\[
\wt\cU_0=\begin{bmatrix} T&D_{T^*}&0&0&0&0&0&0&0&\ldots\cr
0&0&0&I_{\sD_{T^*}}&0&0&0&0&0&\ldots\cr
 D_T&-T^*&0&0&0&0&0&0&0&\ldots\cr
0&0&0&0&0&I_{\sD_{T^*}}&0&0&0&\ldots\cr
0&0&I_{\sD_T}&0&0&0&0&0&0&\ldots\cr
0&0&0&0&0&0&0&I_{\sD_{T^*}}&0&\ldots\cr
0&0&0&0&I_{\sD_T}&0&0&0&0&\ldots\cr
\vdots&\vdots&\vdots&\vdots&\vdots&\vdots&\vdots&\vdots&\vdots&\vdots
\end{bmatrix}.
\]

\subsection{The Naimark dilation} \label{Nai} Let $\sM$ be a
separable Hilbert space. Denote by $\sB(\dT)$ the $\sigma$-algebra
of Borelian subsets of the unit circle $\dT=\{\xi\in\dC:|\xi|=1\}$.
Let $\mu$ be a $\bL(\sM)$-valued Borel measure on $\sB(\dT)$, i.e.,
\begin{enumerate}\def\labelenumi{\rm (\alph{enumi})}
\item for any $\delta\in\sB(\dT)$ the operator $\mu(\delta)$ is
nonnegative,
\item $\mu(\emptyset)=0,$
\item $\mu$ is $\sigma$-additive with respect to the strong operator
convergence.
\end{enumerate}
Denote by ${\bf M}(\dT,\sM)$ the set of all $\bL(\sM)$-valued Borel
measures.
\begin{definition}
 \label{Borel} \cite{Const1}, \cite{BC},\ \cite{D1}. Let $\mu\in {\bf M}(\dT,\sM)$ be a probability measure ( $\mu(\dT)=I_\sM$) and let the operators
 $\{S_n\}_{n\in\dZ}$ be the sequence of Fourier
 coefficients of $\mu$, i.e.,
 \[
S_n=\int\limits_{\dT} \xi^{-n}\mu(d\xi),\; n\in\dZ.
 \]
 A Naimark dilation of $\mu$ is a pair $\{\cH,\cU\}$, where $\cH$ is a
 separable Hilbert space containing $\sM$ as a subspace, $\cU$ is
 unitary operator in $\cH$ such that
 \[
S_n=P_\sM\cU^n\uphar\sM,\; n\in\dZ.
 \]
 A Naimark dilation is called minimal if
 \[
\cspan\{\cU^n\sM,\; n\in\dZ\}=\cH.
 \]
\end{definition}
\begin{proposition}
\label{two} \cite{Const1}, \cite{BC}, \cite{D1}. Let
$\{\cH_1,\cU_1\}$ and $\{\cH_2,\cU_2\}$ be two minimal Naimark
dilations of a probability measure $\mu\in {\bf M}(\dT,\sM)$. Then
there exists a unitary operator $\cW\in\bL(\cH_1,\cH_2)$ such that
\[
\cW\cU_1=\cU_2\cW\quad\mbox{and}\quad \cW\uphar\sM=I_\sM.
\]
\end{proposition}
The minimal Naimark dilation is constructed by T.~Constantinescu in
\cite{Const1} by means of the infinite in both sides block operator
matrix whose entries depend on some choice sequence. Here we
construct the minimal Naimark dilations in the form of block
operator CMV matrices.

\begin{theorem}
\label{NaiCMV}Let $\sM$ be a separable Hilbert space and let $\mu\in
{\bf M}(\dT,\sM)$ be a probability measure. Define the functions
\[
F(\lambda)=\int\limits_\dT\frac{\xi+\lambda}{\xi-\lambda}\,\mu(d\xi),\;\lambda\in\dD,
\]
\[
E(\lambda)=\frac{1}{\lambda}\,(F(\lambda)-I_\sM)(F(\lambda)+I_\sM)^{-1}.
\]
Then $E(\lambda)$ belongs to the Schur class ${\bf S}(\sM,\sM)$. Let
$\{G_n\}_{n\ge 0}$ be the Schur parameters of $E$. Construct the
Hilbert spaces
\[
\sH_0=\sH_0(\{G_n\}_{n\ge 0}),\; \wt\sH_0=\wt\sH_0(\{G_n\}_{n\ge 0})
\]
 and the Hilbert spaces
\[
 \cH_0=\sM\oplus\sH_0,\;\wt\cH_0=\sM\oplus\wt\sH_0.
 \]
 Let
\[
\cU_0=\cU_0(\{G_n\}_{n\ge 0}),\; \wt\cU_0=\wt\cU_0(\{G_n\}_{n\ge 0})
\]
be the block operator CMV matrices constructing by means of
$\{G_n\}$. Then the pairs $\{\cH_0,\cU_0\}$ and
$\{\wt\cH_0,\wt\cU_0\}$ are unitarily equivalent minimal Naimark
dilations of the measure $\mu$.
\end{theorem}
\begin{proof}
The function $F(\lambda)$ has the Taylor expansion
\[
F(\lambda)=I_\sM+2\sum\limits_{n=1}^\infty
\lambda^n\int\limits_{\dT}\xi^{-n}\mu(d\xi)=I_\sM+2\sum\limits_{n=1}^\infty
\lambda^n S_n.
\]
Then
\[
F^*(\bar\lambda)=I_\sM+2\sum\limits_{n=1}^\infty \lambda^n S_{-n}.
\]
 Because $F(\lambda)+F^*(\lambda)\ge 0$ for $\lambda\in\dD$, the
$\bL(\sM)$-valued function $E(\lambda)$ belongs to the Schur class
${\bf S}(\sM,\sM)$. Construct the Hilbert spaces
$\sH_0=\sH_0(\{G_n\}_{n\ge 0}),$ $\cH_0=\sM\oplus\sH_0$ and let
$\cU_0=\cU_0(\{G_n\}_{n\ge 0})=\left(\cU_0(\{G_n\}_{n\ge
0})\right)^*$ be the block operator CMV matrix.
 Then $\cU_0$
is unitary operator in the Hilbert space $\cH_0$. The system
$\zeta_0=\left\{\cU_0;\sM,\sM,\sH_0 \right\}$ is a conservative and
simple, and its transfer function is equal to $E(\lambda)$
 (see Subsection \ref{RRR}, \eqref{CMVMODEL}, Theorem \ref{cmvmod}).
 Hence the transfer of the adjoint system $\zeta^*_0=\left\{\cU^*_0;\sM,\sM,\sH_0
 \right\}$ is equal to $\Theta(\lambda)=E^*(\bar\lambda)$.
 By definition of $F(\lambda)$ and $E(\lambda)$, and by Proposition \ref{CarSch}, and \eqref{FTh} we have
\[
\begin{array}{l}
F(\lambda)=(I_\sM+\lambda E(\lambda))(I_\sM-\lambda
E(\lambda))^{-1}=
(I_\sM+\lambda\Theta^*(\bar\lambda))(I_\sM-\lambda\Theta^*(\bar\lambda))^{-1}=\\
=P_\sM(\cU^*_0+\lambda I_{\cH_0})(\cU^*_0-\lambda
I_{\cH_0})^{-1}\uphar\sM.
\end{array}
\]
Hence
\[
\begin{array}{l}
F(\lambda)=I_\sM+2\sum\limits_{n=1}^\infty\lambda^n P_\sM
\cU_0^{n}\uphar\sM,\\
F^*(\bar\lambda)=I_\sM+2\sum\limits_{n=1}^\infty\lambda^n P_\sM
\cU_0^{-n}\uphar\sM.
\end{array}
\]
Thus, the pair $\{\cH_0,\cU_0\}$ is the minimal Naimark dilation of
the measure $\mu$. The same is true for the pair
$\{\wt\cH_0,\wt\cU_0\}.$
\end{proof}
\section{The block operator CMV matrix models for completely non-unitary contractions}
\label{MATRMOD}
\begin{theorem}
\label{MatrMod1} Let $T$ be a completely non-unitary contraction in
a separable Hilbert space $H$. Let
\[
\Phi_T(\lambda)=(-T+\lambda D_{T^*}(I_H-\lambda
T^*)^{-1}D_T)\uphar\sD_T
\]
be the Sz.-Nagy--Foias characteristic function of $T$ \cite{SF}. If
$ \{\Gamma_n\}_{n\ge 0}$ are the Schur parameters of
$\Phi_T(\lambda)$, then the operator $T$ is unitarily equivalent to
the truncated block operator CMV matrices
$\cT_0(\{\Gamma^*_n\}_{n\ge 0})$ and $\wt\cT_0(\{\Gamma^*_n\}_{n\ge
0})$.
\end{theorem}
\begin{proof}
Consider the simple conservative system
\[
\eta=\left\{\begin{bmatrix} -T^*&D_T\cr D_{T^*}&T
\end{bmatrix};\sD_{T^*},\sD_T, H\right\}.
\]
The transfer function of $\eta$ is given by
\[
\Theta_\eta(\lambda)=\left(- T^*+\lambda D_{T}(I_H-\lambda
T)^{-1}D_{T^*}\right)\uphar\sD_{T^*},\;\lambda\in\dD.
\]
Since
\[
\Phi_T(\lambda)=\left(- T+\lambda D_{T^*}(I_H-\lambda
T^*)^{-1}D_{T}\right)\uphar\sD_{T},\;\lambda\in\dD,
\]
we get $\Phi_T(\lambda)=\Theta^*_\eta(\bar\lambda),$
$\lambda\in\dD$. Hence, if $\{\Gamma_n\}_{n\ge 0}$ are the Schur
parameters of $\Phi_T(\lambda)$, then $\{\Gamma^*_n\}_{n\ge 0}$ are
the Schur parameters of $\Theta_\eta(\lambda)$. Construct the
Hilbert spaces $\sH_0=\sH_0(\{\Gamma^*_n\}_{n\ge 0})$,
$\wt\sH_0=\wt\sH_0(\{\Gamma^*_n\}_{n\ge 0})$, the block operator CMV
matrices $\cU_0=\cU_0(\{\Gamma^*_n\}_{n\ge 0})$,
$\wt\cU_0=\wt\cU_0(\{\Gamma^*_n\}_{n\ge 0})$, truncated block CMV
matrices $\cT_0=\cT_0(\{\Gamma^*_n\}_{n\ge 0})$ and
$\wt\cT_0=\wt\cT_0(\{\Gamma^*_n\}_{n\ge 0})$. Consider the
corresponding conservative systems
\[
\zeta_0=\left\{\cU_0;\sD_{T^*},\sD_T,\sH_0\right\},\;\wt\zeta_0=\left\{\wt\cU_0;\sD_{T^*},\sD_T,\wt\sH_0\right\}.
\]
By Theorem \ref{cmvmod} the systems $\zeta_0$ and $\wt\zeta_0$ are
simple conservative realizations of the function $\Theta$. It
follows that the operator $T$ is unitarily equivalent to the
operators $\cT_0(\{\Gamma^*_n\}_{n\ge 0})$ and
$\wt\cT_0(\{\Gamma^*_n\}_{n\ge 0})$.
\end{proof}
Observe that $\cT_0(\{\Gamma^*_n\}_{n\ge
0})=\left(\wt\cT_0(\{\Gamma_n\}_{n\ge 0}) \right)^*$ and
$\wt\cT_0(\{\Gamma^*_n\}_{n\ge 0})=\left(\cT_0(\{\Gamma_n\}_{n\ge
0}) \right)^*$.

The results of Sz.-Nagy--Foias  \cite[Theorem VI.3.1]{SF} states
that if  the function $\Theta\in{\bf S}(\sM,\sN)$ is purely
contractive ($||\Theta(0)f||<||f||$ for all $f\in\sM\setminus\{0\}$)
then there exists a completely non-unitary contraction $T$ whose
characteristic function coincides with $\Theta$. Here we give
another proof of this result.
\begin{theorem}
\label{MatrMod2} Let the function $\Theta(\lambda)\in {\bf
S}(\sM,\sN)$ be purely contractive.
If $\{\Gamma_n\}_{n\ge 0}$ are the Schur parameters of
$\Theta(\lambda)$ then the characteristic functions of completely
non-unitary contractions given by truncated block operator CMV
matrices $\cT_0(\{\Gamma^*_n\}_{n\ge 0})$ and
$\wt\cT_0(\{\Gamma^*_n\}_{n\ge 0})$ coincide with $\Theta$.
\end{theorem}
\begin{proof}
Let $\wt\Theta(\lambda):=\Theta^*(\bar\lambda).$ Then
 $\{\Gamma^*_n\}_{n\ge 0}$ are the Schur parameters of $\wt\Theta$.
Construct the Hilbert spaces $\sH_0=\sH_0(\{\Gamma^*_n\}_{n\ge 0})$,
$\wt\sH_0=\wt\sH_0(\{\Gamma^*_n\}_{n\ge 0})$, the block operator CMV
matrices $\cU_0=\cU_0(\{\Gamma^*_n\}_{n\ge 0})$,
$\wt\cU_0=\wt\cU_0(\{\Gamma^*_n\}_{n\ge 0})$, truncated block CMV
matrices $\cT_0=\cT_0(\{\Gamma^*_n\}_{n\ge 0})$,
$\wt\cT_0=\wt\cT_0(\{\Gamma^*_n\}_{n\ge 0})$, and consider the
corresponding conservative systems
\[
\zeta_0=\left\{\cU_0;\sM,\sN,\sH_0\right\},\;\wt\zeta_0=\left\{\wt\cU_0;\sM,\sN,\wt\sH_0\right\}.
\]
Then the transfer functions of $\zeta_0$ and $\wt\zeta_0$ are equal
to $\wt\Theta(\lambda).$ Since the operator
\[
\cU_0=\begin{bmatrix}\Gamma^*_0&\cG_0\cr \cF_0&\cT_0
\end{bmatrix}:\begin{array}{l}\sM\\\oplus\\\sH_0\end{array}\to
\begin{array}{l}\sN\\\oplus\\\sH_0\end{array}
\]
is a contraction, there exist contractions (see \cite{AG},
\cite{DaKaWe}, \cite{ShYa}) $\cK\in\bL(\sD_{\cT_0},\sN)$,
$\cM\in\bL(\sM,\sH_0)$, $\cX\in\bL(\sD_\cM,\sD_{\cK^*})$ such that
\[
\cG_0=\cK
D_{\cT_0},\;\cF_0=\cD_{\cT^*_0}\cM,\;\Gamma^*_0=-\cK\cT^*_0\cM+D_{\cK^*}\cX
D_{\cM}.
\]
Because $\cU_0$ is unitary, the operators $\cK$, $\cM^*$ are
isometries and $\cX$ is unitary (see \cite{ARL1},\cite{Arlarxiv}).
The characteristic function of $\cT^*_0$ and the transfer function
of the system $\zeta_0$ are connected by the relation (see
\cite{AHS1}, \cite{Arlarxiv})
\[
\wt\Theta(\lambda)=\cK\Phi_{\cT^*_0}(\lambda)\cM+\cX D_{\cM}.
\]
Because the operator $D_{\cM}$ is the orthogonal projection in $\sM$
onto $\ker\cM$, and
\[
\wt\Theta(\lambda)\uphar\ker \cM=\cX,
\]
we have for $f\in\ker\cM$
\[
||\Gamma^*_0f||=||\wt\Theta(0)f||=||\cX f||=||f||.
\]
Since $\Gamma^*_0$ is a pure contraction, we obtain $\ker\cM=\{0\}$.
Similarly $\ker \cK^*=\{0\}$, i.e.,  $\cK$ and $\cM$ are unitary
operators, and $\wt\Theta(\lambda)=\cK\Phi_{\cT^*_0}(\lambda)\cM,$
$\lambda\in\dD.$ Thus the characteristic function $\Phi_{\cT_0}$ of
$\cT_0$ coincides with $\Theta.$ Similarly, the characteristic
function $\Phi_{\wt\cT_0}$ of $\wt\cT_0$ coincides with $\Theta.$
\end{proof}
\begin{remark} For completely non-unitary contractions with one-dimensional defect operators and for
a scalar Schur class functions Theorem \ref{MatrMod1} and Theorem
 \ref{MatrMod2} have been established in \cite{AGT}.

\end{remark}


\begin{thebibliography}{AHS}
\bibitem{AA}
V.M.~Adamjan and D.Z.~Arov, On unitary coupling of semi-unitary
operators, Math Issled. Kishinev I(2) (1966), 3--65 [Russian].
English translation in AMS Tranl. 95 (1970), 75--129.

\bibitem{AADL1}
D.~Alpay, T.~Azizov, A.~Dijksma, and H.~Langer, The Schur algorithm
for generalized Schur functions. I. Coisometric realizations,
 Oper. Theory:  Adv. Appl., 129 (2001), 1--36.

\bibitem{AADLW1}
D.~Alpay, T.~Azizov, A.~Dijksma, H.~Langer, and G.~Wanjala, The
Schur algorithm for generalized Schur functions. II. Jordan chains
and transformations of characteristic functions.  Monatsh. Math. 138
(2003),  no. 1, 1--29.

\bibitem{AADLW2}
D.~Alpay, T.~Azizov, A.~Dijksma, H.~Langer, and G.~Wanjala, The
Schur algorithm for generalized Schur functions. IV. Unitary
realizations, Oper. Theory Adv. Appl., 149 (2004), 23--45.

\bibitem{ADL}
D.~Alpay, A.~Dijksma, and H.~Langer, The transformation of Issai
Schur and related topics in indefinite setting, Oper. Theory  Adv.
Appl., 176 (2007), 1--98.

\bibitem{ADRS}
D.~Alpay, A.~Dijksma, J.~Rovnyak, and H.S.V. de~Snoo, \textit{Schur
functions, operator colligations, and Pontryagin spaces}, Oper.
Theory  Adv. Appl., 96, Birkh\"auser Verlag, Basel-Boston, 1997.

\bibitem{Ando}
T.~Ando, \textit{De Branges spaces and analytic operator functions},
Division of Applied Mathematics, Research Institute of Applied
Electricity, Hokkaido University, Sapporo, Japan, 1990.


\bibitem{ARL1}
Yu.~Arlinski\u{\i}, The Kalman--Yakubovich--Popov inequality for
passive discrete time-invariant systems, Operators and Matrices
2(2008), No.1, 15--51.

\bibitem{Arlarxiv}
Yu.M.~Arlinski\u{\i}, Iterates of the Schur class operator-valued
function and their conservative realizations, arXiv: 0801.4267
[math.FA] v 1, 28 Jan 2008.



\bibitem{AHS1}
Yu.M.~Arlinski\u{\i}, S.~Hassi, and H.S.V.~de~Snoo, Parametrization
of contractive block-operator matrices and passive disrete-time
systems, Complex Analysis and Operator Theory, 1 (2007), No.2,
211--233.
\bibitem{AHS2}
Yu.M.~Arlinski\u{\i}, S.~Hassi, and H.S.V. de~Snoo. Q-functions of
Hermitian contractions of Kre\u{\i}n -- Ovcharenko type, Integral
Equations and Operator Theory. 53 (2005), No.2, 153--189.
\bibitem{AGT}
Yu.~Arlinski\u{\i}, L.~Golinski\u{\i}, and E.~Tsekanovski\u{\i},
Contractions with rank one defect operators and truncated CMV
matrices, Journ. of Funct. Anal., 254 (2008) 154--195.

\bibitem{A}
 D.Z.~Arov, Passive linear stationary dynamical
systems, Sibirsk. Math. Journ.(1979) 20, No.2., 211-228 [Russian].
English translation in Siberian Math. Journ., 20 (1979), 149--162.

\bibitem{Arov}
D.Z.~Arov, Stable dissipative linear stationary dynamical scattering
systems, J. Operator Theory, 1 (1979), 95--126 (Russian). English
translation in Interpolation Theory, Systems, Theory and Related
Topics. The Harry Dym Anniversary Volume, Oper. Theory: Adv. Appl.
134 (2002), 99--136, Birkhauser Verlag.


\bibitem{ArKaaP}
D.Z.~Arov, M.A.~Kaashoek, and D.P.~Pik, Minimal and optimal linear
discrete time-invariant dissipative scattering systems, Integral
Equations Operator Theory,  29 (1997), 127--154.

\bibitem{ArNu1}
D.Z.~Arov and M.A.~Nudel'man, A criterion for the unitary similarity
of minimal passive systems of scattering with a given transfer
function, Ukrain. Math. J., 52 no. 2 (2000), 161--172 [Russian].
English translation in Ukrainian Math. J. 52 (2000), no. 2,
161--172.

\bibitem{ArNu2}
D.Z.~Arov and M.A.~Nudel'man, Tests for the similarity of all
minimal passive realizations of a fixed transfer function
(scattering and resistance matrix), Mat. Sb., 193 no. 6 (2002),
3--24 [Russian]. English translation in Sb. Math. 193 (2002), no.
5-6, 791--810.

\bibitem{ArSt}
D.Z.~ Arov and O.~Staffans, State/signal linear time-invariant
systems theory: passive discrete time systems. Internat. J. Robust
Nonlinear Control 17 (2007), no. 5-6, 497--548.

\bibitem{AG}
Gr.~Arsene and A.~Gheondea, Completing matrix contractions, J.
Operator Theory, 7 (1982), 179-189.

\bibitem{BC}
M.~Bakonyi and T.~Constantinescu, \textit{Schur's algorithm and
several applications}, Pitman Research Notes in Mathematics Series,
v. 261, Longman Scientific and Technical, 1992.

\bibitem{Ball}
J.A.~Ball, Linear systems, operator model theory and scattering:
multivariable generalizations, in Operator Theory and its
Applications (Winnipeg, MB, 1998), 151--178, Fields Inst. Commun.,
25, Amer. Math. Soc., Providence, RI, 2000.



\bibitem{Ball-Coehn}
J.A.~Ball and N.~Cohen, de Branges-Rovnyak operator models and
systems theory: a survey. Topics in matrix and operator theory
(Rotterdam, 1989), 93--136, Oper. Theory Adv. Appl., 50, Birkhäuser,
Basel, 1991.

\bibitem{Ball-C-Ue}
J.A.~Ball, Ph.T.~Carrol, and Yo.~Uetake, Lax-Phillips scattering
theory and well-posed linear systems: a coordinate-free approach,
Math. Control Signals Systems 20 (2008), no. 1, 37--79.

\bibitem{BHJ03}
 O.~Bourget, J.S.~Howland, and A.~Joye, Spectral
analysis of unitary band matrices, Commun. Math. Phys. 234 (2003),
191--227.


\bibitem{BrR1}
L.~de Branges and J.~Rovnyak, \textit{Square summable power series},
Holt, Rinehart and Winston, New-York, 1966.

\bibitem{BrR2}
L.~de Branges and J.~Rovnyak, Appendix on square summable power
series, Canonical models in quantum scattering theory, in:
\textit{Perturbation theory and its applications in quantum
mechanics} (ed. C.H.~Wilcox), New-York, 1966, pp. 295--392.

\bibitem{VilBr}
V.M.~Brodski\u{\i}, On operator colligations and their
characteristic functions, Soviet Math. Dokl., 12 (1971), 696--700
[Russian].
\bibitem{Br}
M.S.~Brodski\u{\i}, \textit{Triangular and Jordan representations of
linear operators}, Translations of Mathematical Monographs, Vol. 32.
American Mathematical Society, Providence, R.I., 1971.

\bibitem{Br1}
M.S.~Brodski\u{\i}, Unitary operator colligations and their
characteristic functions, Uspekhi Mat. Nauk, 33 (1978), No.4,
141--168 [Russian]. English translation in Russian Math. Surveys, 33
(1978), No.4, 159--191.

\bibitem{BGE91}
 A.~Bunse-Gerstner and L.~Elsner, Schur parameter
pencils for the solution of unitary eigenproblem, Lin. Algebra Appl.
154/156 (1991), 741--778.



\bibitem{CMV1}
M.J.~Cantero, L.~Moral, and L.~Vel\'azquez, Five-diagonal matrices
and zeros of orthogonal polynomials on the unit circle, Linear
Algebra Appl. 362 (2003), 29--56.




\bibitem{CF}
Z.~Ceausescu, C.~Foias, On intertwining dilations. V, Acta Sci.
Math. (Szeged), 40(1978), 9--32; Corrections, 41 (1979).

\bibitem{Const1}
T.~Constantinscu, On the structure of Naimark dilation, J. Operator
Theory, 12 (1984), 159--175.


\bibitem{Const}
T.~Constantinscu, Operator Schur algorithm and associated functions,
Math. Balcanica 2 (1988), 244-252.

\bibitem{Const3}
T.~Constantinscu, On the structure of positive Toeplitz forms, in
"Dilation theory, Toeplitz operators, and other topics"
(Timisoara/Herculane, 1982), Oper. Theory Adv. Appl., 11 (1983),
127-149. Birkhauser Verlag.


\bibitem{Const2}
T.~Constantinscu, \textit{Schur parameters, factorization and
dilation problems}, Operator Theory: Advances and Applications, 82.
Bikh\"{a}user Verlag, Basel, 1996.



\bibitem{DPS}
D.~Damanik, A.~Pushnitskii, and B.~Simon, The analytical theory of
matrix orthogonal polynomials,  Surv. Approx. Theory 4 (2008),
1--85.


\bibitem{DaKaWe}
Ch.~Davis, W.M.~Kahan, and H.F.~Weinberger, Norm preserving
dilations and their applications to optimal error bounds, SIAM
J. Numer. Anal., 19 
(1982), 445--469.
\bibitem{DGK1}
Ph.~Delsarte, Y.~Genin, and Y.~Kamp, Orthogonal polynomial matrices
on the unit circle. IEEE Trans. Circuits and Systems (1978), no. 3,
149--160.

\bibitem{DGK}
Ph.~Delsarte, Y.~Genin, and Y.~Kamp, Schur parametrization of
positive definite block-Toeplitz systems.  SIAM J. Appl. Math.  36
(1979), no. 1, 34--46.

\bibitem{D1}
V.K.~Dubovoy, Shift operators contained in contractions, Schur
parameters and pseudocontinuable Schur functions, Operator
Theory:Advances and Applications. 165 (2006), 175--250.

\bibitem{DFK}
V.K.~Dubovoj, B.~Fritzsche, and B.~Kirstein, Matricial version of
the classical Schur problem. Teubner-Texte zur Mathematik [Teubner
Texts in Mathematics], 129. B.G.~Teubner Verlagsgesellschaft mbH,
Stuttgart, 1992. 355 pp.
\bibitem{FKatsKr}
B.~Fritzsche, V.~Katsnelson, and B.~Kirstein, The Schur algorithm in
terms of system realizations, arXiv: 0805.4732 [math. CV] v 1, 30
May 2008.

\bibitem{ger1}
Ya.L.~Geronimus, On polynomials orthogonal on the circle, on
trigonometric moment problem, and on allied Carath\'eodory and Schur
functions, Mat. Sb., {\bf 15} (1944), 99--130. [Russian].




\bibitem{Helton1}
J.W.~Helton, The characteristic functions of operator theory and
electrical network realization. Indiana Univ. Math. J. 22 (1972/73),
403--414.

\bibitem{Helton2}
J.W.~Helton, Discrete time systems, operator models, and scattering
theory. J. Functional Analysis 16 (1974), 15--38.

\bibitem{KailBruck}
T.~Kailath and A.M.~Bruckstein, Naimark dilations, state-space
generators and transmission lines, Oper. Theory, Adv. Appl. 17
(1986), 173--186.


\bibitem{LPh}
P.D.~Lax and R.S.~Phillips, \textit{Scattering Theory}. Academic
Press, New York,  1967.

\bibitem{NV3}
N.K.~Nikolski\u{\i} and V.I.~Vasyunin, Notes on two function models.
The Bieberbach conjecture (West Lafayette, Ind., 1985), 113--141,
Math. Surveys Monogr., 21, Amer. Math. Soc., Providence, RI, 1986.


\bibitem{NV1}
N.K.~Nikolski\u{\i} and V.I.~Vasyunin, A unified approach to
function models, and the transription problem, Operator Theory: Adv.
and Appl., 41 (1988), 405--434.

\bibitem{NV2}
N.K.~Nikolski\u{\i} and V.I.~Vasyunin, Elements of spectral theory
in terms of the free functional model. I. Basic constructions.
Holomorphic spaces (Berkely, CA, 1995), 211--302, Math. Sci. Res.
Inst. Publ., 33, Cambridge Univ. Press, Cambridge, 1998.


\bibitem{Pavlov1}
B.S.~Pavlov, The continuous spectrum of resounances on a nonphysical
sheet, Dokl. Akad. Nauk SSSR, 206(1972), 1301--1304 [Russian].

\bibitem{Pavlov2}
B.S.~Pavlov, Conditions for separation of the spectral components of
a dissipative operator. Izv. Akad. Nauk SSSR, Ser. Mat., 39 (1975),
123--148 [Russian].

\bibitem{Pavlov3}
B.S.~Pavlov, A remark on spectral meaning of the symmetric
functional model, Operator Theory: Adv. Appl., 154 (2004), 163--177.

\bibitem{Sch}
J.~Sch\"affer, On unitary dilations of contractions, Proc. AMS, 6
(1955), 322.



\bibitem{Schur} I.~Schur, \"Uber Potenzreihen, die im Innern des
Einheitskreises beschr\"ankt send, I, II, {\it J. Reine Angew.
Math.}, {\bf 147}(1917), 205--232; {\bf 148}(1918), 122--145.


\bibitem{Shmul1}
Yu.L.~Shmul'yan, Generalized fractional-linear transformations of
operator balls, Sibirsk. Mat. Zh. 21 (1980), No.5, 114--131
[Russian]. English translation in Siberian Mathematical Jour. 21
(1980), No.5, 728--740.

\bibitem{Shmul2}
Yu.L.~Shmul'yan, Certain stability properties for analytic
operator-valued functions, Mat. Zametki 20 (1976), No.4, 511--520
[Russian]. English translation in Mathematical Notes, 20 (1976),
No.4, 843--848.


\bibitem{ShYa}
Yu.L.~Shmul'yan and R.N.~Yanovskaya, Blocks of a contractive
operator matrix, Izv. Vuzov, Mat., 7 (1981), 72-75 [Russian].

\bibitem{Si1}
B.~ Simon, \textit{Orthogonal polynomials on the unit circle. Part
1. Classical theory, Part 2. Spectral theory.} American Mathematical
Society Colloquium Publications, Providence, RI, 2005.

\bibitem{S3}
B.~Simon, CMV matrices: five years after, J. Comput. Appl. Math. 208
(2007), No. 1, 120--154.

\bibitem{Staf1}
O.~Staffans, Well-posed linear systems.
 Encyclopedia of Mathematics and its Applications, 103.
 Cambridge University Press, Cambridge, 2005.

\bibitem{Staf2}
O.~Staffans, Passive linear discrete time-invariant systems,
International Congress of Mathematicians. Vol. III, 1367--1388, Eur.
Math. Soc., Zurich, 2006~.

\bibitem{SF}
B.~Sz.-Nagy and C.~Foias, \textit{Harmonic analysis of operators on
Hilbert space}, North-Holland, New York, 1970.

\bibitem{Wa93}
 D.S.~Watkins, Some perspectives on the eigenvalue
problem, SIAM Rev.  35 (1993), 430--471.

\end{thebibliography}
\end{document}